\documentclass[aop]{imsart}

\RequirePackage{amsthm,amsmath,amsfonts,amssymb}
\RequirePackage[numbers,sort]{natbib}
\RequirePackage[colorlinks,citecolor=blue,urlcolor=blue]{hyperref}
\RequirePackage{graphicx}

\startlocaldefs
\theoremstyle{plain}

\newtheorem{theorem}{Theorem}[section]
\newtheorem{lemma}[theorem]{Lemma}
\newtheorem{corollary}[theorem]{Corollary}
\newtheorem{proposition}[theorem]{Proposition}

\theoremstyle{remark}

\newtheorem{conjecture}{Conjecture}

\newcommand{\rme}{\mathrm{e}}

\newcommand{\bbN}{\mathbb{N}}

\newcommand{\1}{\operatorname{\uppercase\expandafter{\romannumeral1}}}
\newcommand{\2}{\operatorname{\uppercase\expandafter{\romannumeral2}}}
\newcommand{\3}{\operatorname{\uppercase\expandafter{\romannumeral3}}}
\newcommand{\4}{\operatorname{\uppercase\expandafter{\romannumeral4}}}
\newcommand{\5}{\operatorname{\uppercase\expandafter{\romannumeral5}}}
\newcommand{\6}{\operatorname{\uppercase\expandafter{\romannumeral6}}}
\newcommand{\7}{\operatorname{\uppercase\expandafter{\romannumeral7}}}
\newcommand{\8}{\operatorname{\uppercase\expandafter{\romannumeral8}}}
\newcommand{\9}{\operatorname{\uppercase\expandafter{\romannumeral9}}}

\def\eqref#1{\textup{(\ref{#1})}}
\newcommand{\eref}[1]{Eq.~\textup{(\ref{#1})}}
\newcommand{\Eref}[1]{Equation~\textup{(\ref{#1})}}

\newcommand{\eqsref}[2]{Eqs.~(\ref{#1}) and (\ref{#2})}
\newcommand{\Eqsref}[2]{Equations~(\ref{#1}) and (\ref{#2})}

\newcommand{\eqssref}[3]{Eqs.~(\ref{#1}), (\ref{#2}), and (\ref{#3})}

\newcommand{\fref}[1]{Fig.~\ref{#1}}

\newcommand{\thref}[1]{Theorem~\ref{#1}}
\newcommand{\Thref}[1]{Theorem~\ref{#1}}
\newcommand{\thsref}[1]{Theorems~\ref{#1}}

\newcommand{\lref}[1]{Lemma~\ref{#1}}
\newcommand{\Lref}[1]{Lemma~\ref{#1}}
\newcommand{\lsref}[1]{Lemmas~\ref{#1}}

\newcommand{\pref}[1]{Proposition~\ref{#1}}

\newcommand{\psref}[1]{Propositions~\ref{#1}}

\newcommand{\crref}[1]{Corollary~\ref{#1}}
\newcommand{\Crref}[1]{Corollary~\ref{#1}}

\newcommand{\cref}[1]{Conjecture~\ref{#1}}
\newcommand{\Cref}[1]{Conjecture~\ref{#1}}

\newcommand{\sref}[1]{Sec.~\ref{#1}}
\newcommand{\Sref}[1]{Section~\ref{#1}}

\newcommand{\rcite}[1]{Ref.~\cite{#1}}
\newcommand{\rscite}[1]{Refs.~\cite{#1}}

\newcommand{\aref}[1]{Appendix~\ref{#1}}

\def\<{\langle}  
\def\>{\rangle}  

\endlocaldefs

\begin{document}

\begin{frontmatter}
\title{Nearly tight universal bounds for the binomial tail probabilities}

\runtitle{Tight bounds for the binomial tail}

\begin{aug}

\author[A]{\fnms{Huangjun}~\snm{Zhu}\ead[label=e1]{zhuhuangjun@fudan.edu.cn}\orcid{0000-0001-7257-0764}}
\author[A]{\fnms{Zihao}~\snm{Li}\ead[label=e2]{zihaoli20@fudan.edu.cn}}
\and
\author[B]{\fnms{Masahito}~\snm{Hayashi}\ead[label=e3]{hayashi@sustech.edu.cn}
	\orcid{0000-0003-3104-1000}}

\address[A]{Department of Physics and State Key Laboratory of Surface Physics, Fudan University\printead[presep={,\ }]{e1,e2}}

\address[B]{Shenzhen Institute for Quantum Science and Engineering, Southern University of Science and Technology\printead[presep={,\ }]{e3}}
\end{aug}

\begin{abstract}
We derive simple but nearly tight upper and lower bounds for the binomial   lower tail probability (with straightforward generalization to the upper tail probability) that apply to the whole parameter regime. These bounds are easy to compute and are tight within a constant factor of $89/44$. Moreover, they are  asymptotically tight in the regimes of large deviation and moderate deviation. By virtue of a surprising connection with Ramanujan's equation, we also provide strong evidences  suggesting that the lower bound is 
 tight within a factor of $1.26434$. It may even be regarded as the natural lower bound, given its simplicity and appealing properties. Our bounds significantly outperform the familiar Chernoff bound and reverse Chernoff bounds known in the literature and may find applications in various research areas. 
\end{abstract}



\end{frontmatter}
\date{\today}
\maketitle

\tableofcontents

\section{Introduction}
The evaluation of tail probabilities is one of central topics in probability theory 
because it is tied to many important applications, including hypothesis testing, statistical  inference, information theory,  statistical physics, machine learning, insurance, and risk management. However, it is in general not easy to derive accurate bounds for tail probabilities even for many simple probability distributions. Here we are particularly interested in the binomial distribution, which is one of the oldest probability distributions studied in the literature \cite{Bern1713book,Moiv1738book,Lapl1812book,Hald03book}. It characterizes the probability of obtaining $k$ successes after $n$ independent Bernoulli trials, assuming that the success probability of each trial is $p$. To be concrete this probability and the probability of obtaining at most $k$ successes are give by
\begin{align}\label{eq:Bnkp}
\quad b_{n,k}(p):= {n \choose k} p^k q^{n-k},\quad B_{n,k}(p):= \sum_{j=0}^k b_{n,j}(p)= \sum_{j=0}^k {n \choose j} p^j 
q^{n-j},
\end{align}
where $q=1-p$. To avoid trivial exceptions, we  assume that  $0<p<1$ (so $0<q<1$) unless stated otherwise. When  $k\leq pn$,
the probability $B_{n,k}(p)$ is referred to as a lower tail probability, which has been studied by numerous researchers in various research areas for a long history \cite{Bern1713book,Moiv1738book,Lapl1812book,Cher52,Hoef63,PeizP68,Litt69,McKa89,Hald03book,MacWS77book,Ash92book, JohnKK05,CsisK11,Csis98,Haya17book}.

One of the most popular upper bounds for $B_{n,k}(p)$ is the Chernoff bound \cite{Cher52,Hoef63},
\begin{align}
B_{n,k}(p)\leq \rme^{-n D(\frac{k}{n}\| p)}\quad \forall k\leq np \label{eq:ChernoffB},
\end{align}
which correctly characterizes the exponential decay rate of the tail probability. Here
\begin{align}
D(f\|p): =f \ln \frac{f}{p}+(1-f) \ln \frac{1-f}{1-p}
\end{align}
is the familiar relative entropy (information divergence).  Two popular reverse  Chernoff bounds are given by  
\begin{gather}
\frac{1}{n+1} \rme^{-nD(\frac{k}{n}\| p)}\leq 	b_{n,k}(p)\leq B_{n,k}(p), \label{eq:ChernoffRevType}\\
\frac{\sqrt{n}}{\sqrt{8k(n-k)}}\rme^{-nD(\frac{k}{n}\| p)}\leq b_{n,k}(p)\leq 
B_{n,k}(p). \label{eq:ChernoffRev2}
\end{gather}
Here the first bound can be derived    with the method of types \cite{CsisK11,Csis98}; the second bound  follows from Lemma~4.7.1  in Ref.~\cite{Ash92book} and from [Chapter 10, Lemma 7] in Ref.~\cite{MacWS77book}. 
Unfortunately,  the Chernoff bound has a major drawback: its ratio over the tail probability is not bounded by any given constant. This is the case even if we only consider the asymptotic regime in which $k,n\to \infty$. The two reverse Chernoff bounds in \eqsref{eq:ChernoffRevType}{eq:ChernoffRev2}  have a similar problem. Although  many alternative bounds are  known in the literature \cite{MacWS77book,Ash92book, JohnKK05,CsisK11,Csis98,Haya17book},  almost all bounds share the same problem unfortunately. Can we construct much better bounds?

The main goal of the current study is to establish good upper bound $B_{n,k}^{\uparrow}(p)$ and lower bound $B_{n,k}^{\downarrow}(p)$ for the tail probability $B_{n,k}(p)$ that bear  certain desired properties. 
To be concrete, such bounds should satisfy the following three reasonable criteria, which 
are related to criteria in \rcite{WataH17}.  Our criteria are applicable when both upper and lower bounds are available, but it is straightforward to formulate similar criteria for the upper bound or lower bound alone by replacing  $B_{n,k}^{\downarrow}(p)$ or $B_{n,k}^{\uparrow}(p)$ with $B_{n,k}(p)$.
\begin{description}
	\item[(C1)]
Computability: The bounds have computational complexity $O(1)$; in other words, they are  $O(1)$-computable.	
	
	\item[(C2)]
Universal boundedness: the ratio
$B_{n,k}^\uparrow(p) /B_{n,k}^\downarrow(p)$ with $k\leq pn$ is bounded by 
 a universal constant.
	
	\item[(C3)]
	Asymptotic tightness: The bounds are tight in the limit $n\to\infty$ when $0<k/n<p$ is fixed, that is, 
	\begin{align}
\lim_{n \to \infty}\frac{B_{n,fn}^\uparrow(p)}{B_{n,fn}^\downarrow(p)}=1\quad \forall 0<f<p.  \label{eq:ATratio}
	\end{align}
\end{description}
 Here we assume that elementary operations, such as addition and multiplication, are $O(1)$ when evaluating the computational complexity. Besides computability, we prefer bounds that are simple and explicit  functions that do not involve integration because merely numerical bounds for the tail probability are not enough for many applications.

To better understand the criterion of asymptotic tightness, we need to introduce some additional concepts. 
Let $\bbN$ be the set of natural numbers (positive integers) and $\bbN_0$ the set of nonnegative integers. Given a real number $f$, define
\begin{align}
\bbN_f:=\{n\in \bbN_0\,|\,fn\in \bbN_0\}. \label{eq:Nf}
\end{align}
When  $f$ is a rational number that satisfies $0<f<p$ and $n\in \bbN_f$, as a simple corollary of Theorem 2 in \rcite{AG} we can deduce that 
\begin{align}
&\lim_{n \to \infty}\sqrt{n} B_{n,fn}(p) \rme^{nD (f\|p)}
=\frac{1}{(1-r)\sqrt{2\pi  f(1 -f)} }=\frac{\sqrt{(1-f)}\,p}{\sqrt{2\pi f} \,(p -f) },\label{ACP}
\end{align}
where  
\begin{align}
r=r(f,p):=
\frac{f q }{(1-f)p}\label{eq:OddsRatio}
\end{align}
is the odds ratio. 	This result can also be derived by virtue of the theory of  strong large deviation  \cite{BlacH59,BahaR60,DembZ10} as shown in Appendix \ref{A1}.
Compared with the theory of large deviation  \cite{Cram38,Cher52,Vara84book,DembZ10}, which characterizes the exponential decay rate,   strong large deviation focuses on more accurate expansion of the tail probability that is beyond conventional large deviation.

In view of \eref{ACP}, the condition of asymptotic tightness in \eref{eq:ATratio} can also be formulated as follows, 
\begin{align}
\lim_{n \to \infty} \sqrt{n} B_{n,fn}^\downarrow(p) \rme^{nD (f\|p)}
=\lim_{n \to \infty}\sqrt{n} B_{n,fn}^\uparrow (p) \rme^{nD (f\|p)}
=\frac{\sqrt{(1-f)}\,p}{\sqrt{2\pi f}\, (p -f) }.\label{eq:TIGHT}
\end{align}
Such bounds are of special interest in the study of strong large deviation \cite{BlacH59,BahaR60,DembZ10}.  Recently such bounds have found numerous applications in classical and quantum information theory, 
including 
finite-length analysis for channel coding  \cite{Ferr21},
channel coding with higher orders \cite{Moul17,Haya18}, 
security analysis with higher orders \cite{Haya18, Haya19},
quantum thermodynamics \cite{TajiH17,ItoH18}, and local discrimination \cite{HayaO17}.
Unfortunately, it is in general not easy to derive bounds that satisfy the condition of asymptotic tightness, that is, criterion (C3).
Actually,  most bounds  for this regime \cite{Cher52,Hoef63,MacWS77book,Ash92book, JohnKK05,CsisK11,Csis98,WataH17} known in the literature satisfy criterion (C1), but few bounds satisfy criterion (C2) or (C3). Notably,  the Chernoff and reverse Chernoff bounds  reproduced in \eqssref{eq:ChernoffB}{eq:ChernoffRevType}{eq:ChernoffRev2} satisfy neither (C2) nor (C3). 
As exceptions, the bounds derived by McKay \cite{McKa89}  satisfy criteria (C2) and (C3), but does not satisfy (C1)
because the bounds involve the probability $b_{n-1,k-1}(p)$;
in addition, the bounds involve integrals and are not so explicit compared with the Chernoff and reverse Chernoff bounds mentioned above. 
The bounds derived recently by Ferrante \cite{Ferr21} satisfy criteria (C1) and (C3), but do not satisfy criterion (C2). 

	In addition to the regime of large deviation, 	the regime of moderate deviation \cite{Wu95,Acos97,DembZ10,WataH17} is of independent interest. 
Here $k$ behaves as 
	$pn -\alpha_n$ with the  sequence $\alpha_n$ satisfying the conditions 
	$ \alpha_n/n \to 0$ and $\alpha_n/\sqrt{n}\to \infty$. This regime interpolates between the regime of large deviation  and the regime of central limit theorem (CLT). 	
The asymptotics of this regime is useful to the analysis of various types of information processing \cite{AltuW14,PolyV10,Tan12,HayaW16,HayaW20,ChubTT17,ChenH18}. Although several works have studied the exponential decay rate of the tail probability in this regime \cite{Wu95,Acos97,DembZ10,WataH17}, few papers have derived its asymptotic behavior up to  constant multiplicative factors.

In this paper, to find the desired upper and lower bounds for 
$B_{n,k}(p)$,
we derive nearly tight  upper and lower bounds for 
the ratio $B_{n,k}(p)/b_{n,k}(p)$ in the first step. 
Then, we prepare various 
nearly tight  upper and lower bounds for the probability $b_{n,k}(p)$.
Combining these results, we derive 
nearly tight  upper and lower bounds for 
the tail probability $B_{n,k}(p)$ that satisfy criteria (C1-C3). Notably, our bounds are tight within a constant factor of $89/44$ and are  asymptotically tight in the regime of moderate deviation besides the regime of large deviation.  In addition, we conjecture that our lower bound for the ratio $B_{n,k}(p)/b_{n,k}(p)$ is tight within a factor of $180451625/143327232$. If this conjecture holds, then our lower bound $B_{n,k}^\downarrow(p)$ for the tail probability $B_{n,k}(p)$  is tight within a factor of $1.26434$. Furthermore
we prove this conjecture in a special case by virtue of a surprising connection with Ramanujan's equation \cite{Rama27,JogdS68}. This connection  indicates that our work is of interest beyond probability theory.

The rest of this paper is organized as follows.
\Sref{sec:summary} summarizes the main results.
\Sref{sec:BinProb} prepares fundamental knowledges on the binomial distribution.
\Sref{sec:Bbratio} 
proposes nearly tight  upper and lower bounds for 
the ratio $B_{n,k}(p)/b_{n,k}(p)$.
\Sref{sec:TailBounds} 
proposes nearly tight  upper and lower bounds for 
the tail probability $B_{n,k}(p)$ 
by virtue of good bounds for $b_{n,k}(p)$ and $B_{n,k}(p)/b_{n,k}(p)$.
\Sref{sec:conjecture} 
presents a conjecture on 
the tail probability and provides strong evidences based on a  connection with Ramanujan's equation \cite{Rama27,JogdS68}.
 \Sref{sec:conclusion} concludes this paper.

\section{Summary of results}\label{sec:summary}
\subsection{Evaluation of the ratio
$B_{n,k}(p)/b_{n,k}(p)$} In the first step, to evaluate the ratio $B_{n,k}(p)/b_{n,k}(p)$
 we define the following functions, assuming that $n\geq 0$,  $0<p<1$,  and $0\leq k\leq p n$. It is not necessary to assume that $k$ and $n$ are integers in the following definitions.
\begin{gather}
L(n,k,p):=\frac{k+1-p n+\sqrt{(p n-k+1)^2+4q k}}{2},\label{eq:Lnkp} \\
\kappa_1(n,p):=p(n+1)-\sqrt{p q(n+1)}, \label{eq:kappa1} \\
V(n,k,p,a):=a+\frac{p(n-k+a +1)}{p n+p-k+a},\label{eq:Vnkpa}\\
\begin{split}
U(n,k,p):=\min_{a\in \bbN_0} V(n,k,p,a)=\begin{cases}
V(n,k,p,0) &\! k< \kappa_1(n,p),\\
\min\bigl\{V(n,k,p,\lfloor \tilde{a} \rfloor), V(n,k,p,\lceil \tilde{a}\rceil)\bigr\}\!\! & k\geq \kappa_1(n,p),
\end{cases}\label{eq:Unkp}
\end{split}
\end{gather}
where $q=1-p$, $\tilde{a}= k-\kappa_1(n,p)$.
The second equality in \eref{eq:Unkp} follows from the fact that	$V(n,k,p,x)$ is strictly convex in $x$ for $x\geq k-p n-p$  and has a unique minimum point at $x=\tilde{a}$. In addition, it is easy to verify that 
\begin{align} 
V(n,k,p,\lceil \tilde{a}\rceil)\leq   1+V(n,k,p,\tilde{a} )=1+k-p n+2\sqrt{p q(n+1)}. 
\end{align}
If $0\leq f<p$ and $n$ is sufficiently large; then $k,f n<\kappa_1(n,p)$, so \eref{eq:Unkp} yields
\begin{align}
U(n,k,p)= V(n,k,p,0),\quad 
U(n,f n,p)=  V(n,f n,p,0). \label{eq:UnkpLargen}
\end{align}

Then, as shown in \sref{sec:Bbratio}, we have the following theorem.
\begin{theorem}\label{thm:BnkbnkLBUB}
	Suppose  $k\in \bbN_0$, $n\in \bbN$, $0<p<1$,   $k\leq p n$, and $f=k/n$. Then 
	\begin{gather}
	1\leq L(n,k,p)\leq 	\frac{B_{n,k}(p)}{b_{n,k}(p)}\leq U(n,k,p)< 2L(n,k,p),\label{eq:BnkbnkLBUB}
\end{gather}	
where all inequalities are strict when $k\geq1$. If in addition	$0<f<p$, then
	\begin{gather}
	1<L(n,k,p)<\frac{B_{n,k}(p)}{b_{n,k}(p)}< U(n,k,p)\leq V(n,k,p,0)<  \frac{(1-f)p}{p-f}=\frac{1}{1-r}.	\label{eq:BnkbnkUBasymp}
	\end{gather}
\end{theorem}

Here $r$ is the odds ratio defined in \eref{eq:OddsRatio}. The upper bound $U(n,k,p)$ and lower bound $L(n,k,p)$ can be computed in $O(1)$ time by definitions, assuming that elementary operations, such as addition and multiplication, are $O(1)$. In addition, they are  tight within a factor of 2 by \eref{eq:BnkbnkLBUB}.
Furthermore, the two bounds $U(n,k,p)$ and $L(n,k,p)$
are  asymptotically tight according to the following equations,
\begin{gather}
\lim_{n\rightarrow \infty} L(n,k,p)=\lim_{n\rightarrow \infty} U(n,k,p)=\lim_{n\rightarrow \infty} V(n,k,p,0)=1,\label{eq:LUnkplim}\\
\lim_{n\rightarrow \infty} L(n,f n,p)
=\lim_{n\rightarrow \infty} U(n,f n,p)=\lim_{n\rightarrow \infty} V(n,f n,p,0)=\frac{(1-f)p}{p-f}\quad \forall  0\leq f<p<1,
\label{eq:LUnfplim}
\end{gather} 
which follow from  Eqs.~\eqref{eq:Lnkp}-\eqref{eq:Unkp} and \eqref{eq:UnkpLargen}.  Note that \eref{eq:LUnfplim} still holds if $fn$ is replaced by $\lfloor fn\rfloor$. Numerical calculation illustrated in \fref{fig:BbLratio} further shows that the lower bound $L(n,k,p)$ is more accurate than what can be proved rigorously (cf. \cref{con:RatioConjecture} in \sref{sec:conjecture} for potential improvement).
The combination of \eqsref{eq:BnkbnkUBasymp}{eq:LUnfplim} also 
 implies the following result
 	\begin{align}
	\lim_{n\to \infty} \frac{B_{n,\lfloor fn\rfloor}(p)}{b_{n,\lfloor fn\rfloor}(p)}=\frac{(1-f)p}{p-f}=\frac{1}{1-r}\quad  \forall  0\leq f<p<1.  \label{eq:Bnkbnklim7}
	\end{align}

\begin{figure}
	\includegraphics[width=13.5cm]{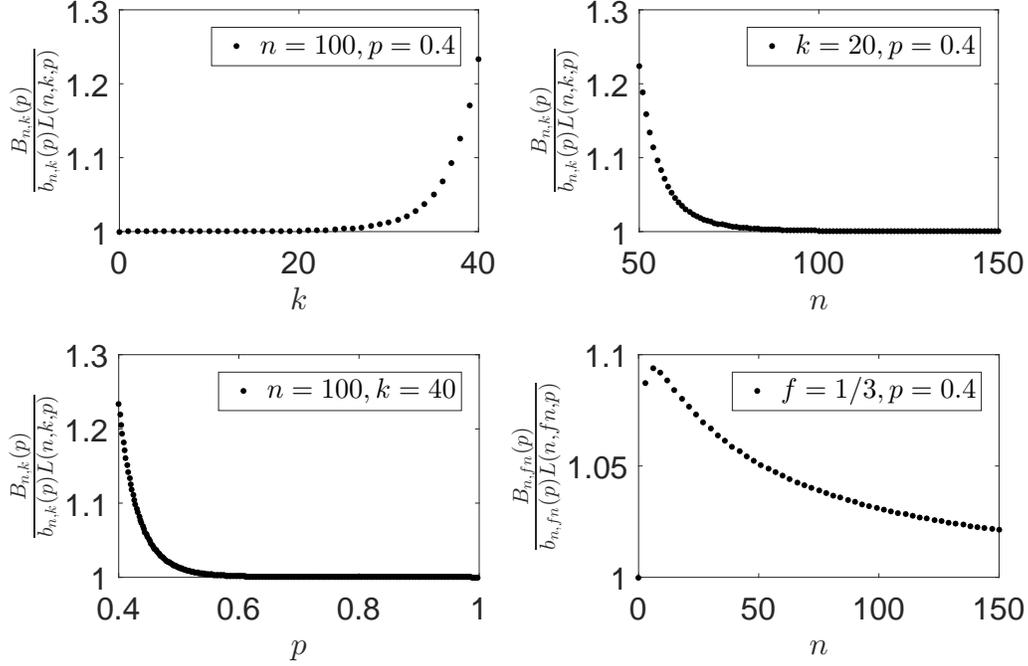}
	\caption{\label{fig:BbLratio} The ratios $B_{n,k}(p)/[b_{n,k}(p)L(n,k,p)]$ and $B_{n,fn}(p)/[b_{n,fn}(p)L(n,fn,p)]$. 
	}
\end{figure}

\subsection{Evaluation of the tail probability $B_{n,k}(p)$}
In the next step, 
to describe our upper and lower bounds for the binomial tail probability 
$B_{n,k}(p)$,  we introduce two quantities, assuming that $0<p<1$ and $1\leq k\leq n-1$,
\begin{equation}\label{eq:BnkpArrow}
\begin{aligned} 
B_{n,k}^{\downarrow}(p) &:=\frac{\sqrt{n}\,L(n,k,p)}{ \sqrt{2\pi k(n-k)}}
\rme^{-n D(\frac{k}{n}\| p )}\rme^{\frac{1}{12n}-\frac{1}{12k}-\frac{1}{12(n-k)}} ,\\
B_{n,k}^{\uparrow}(p)&:=\frac{\sqrt{n}\,U(n,k,p)}{ \sqrt{2\pi k(n-k)}}\rme^{-n D(\frac{k}{n}\| p )}\rme^{\frac{1}{12n+1}-\frac{1}{12k+1}-\frac{1}{12(n-k)+1}}.
\end{aligned}
\end{equation}
With the above definition,
the limit formulas in \eref{eq:LUnfplim} guarantee that
 $B_{n,k}^{\downarrow}(p)$
and $B_{n,k}^{\uparrow}(p)$
satisfy the condition of asymptotic tightness in \eqsref{eq:ATratio}{eq:TIGHT}.
Then, as shown in  \sref{sec:TailBounds}, we have the following theorem.
\begin{theorem}\label{TH2}
Suppose  $n,k\in \bbN$, $0<p<1$,  $k\leq p n$, and $f=k/n$. Then 
\begin{gather}
\frac{\sqrt{n}\,L(n,k,p)}{\sqrt{8k(n-k)}}\rme^{-nD(\frac{k}{n}\| p)}<
B_{n,k}(p)< \frac{\sqrt{n}\, U(n,k,p)}{\sqrt{2\pi k(n-k)}}
\rme^{-nD(\frac{k}{n}\| p)}, \label{eq:TH2E0}\\
B_{n,k}^{\downarrow}(p)< B_{n,k}(p)< B_{n,k}^{\uparrow}(p)<\frac{89}{44} 
B_{n,k}^{\downarrow}(p),\label{eq:TH2E1}\\
B_{n,k}(p)< B_{n,k}^{\uparrow}(p)< \frac{\sqrt{(1-f)}\,p}{\sqrt{2\pi f} \,(p -f) }\rme^{-nD (f\|p)}, \quad k<pn. \label{eq:TH2E2}
\end{gather}
\end{theorem}
The upper and lower bounds in \eref{eq:TH2E0}   are tight within a factor of $4/\sqrt{\pi}\approx 2.25676$ given that $U(n,k,p)<2L(n,k,p)$ by \thref{thm:BnkbnkLBUB}; the lower bound  improves over the popular reverse Chernoff bound in \eref{eq:ChernoffRev2} given that $L(n,k,p)>1$ under the assumptions in \thref{TH2}.
By definitions the bounds $B_{n,k}^{\downarrow}(p)$
and $B_{n,k}^{\uparrow}(p)$  are $O(1)$-computable and thus comply with criterion (C1). In addition, they 
are tight within  a factor of $89/44$ by \eref{eq:TH2E1} and thus comply with   criterion (C2). Furthermore,  they  are  asymptotically tight because they satisfy  \eqsref{eq:ATratio}{eq:TIGHT} and thus comply with  criterion (C3), which also yields another proof of \eref{ACP}.
In a word, the bounds $B_{n,k}^{\downarrow}(p)$
and $B_{n,k}^{\uparrow}(p)$ satisfy all  three criteria of good bounds. The
second upper bound in \eref{eq:TH2E2} is equivalent to an upper bound derived in \rcite{Ferr21}.
 Alternatively bounds for $B_{n,k}(p)$ can be constructed from \thref{thm:TailBound} in \sref{sec:TailBounds}.

In addition, Theorem \ref{TH2} implies the following result in the regime of moderate deviation.
\begin{corollary}\label{cor:ModerateDev}
Suppose $n\in \bbN$ and $\alpha_n$ is a sequence with the properties $\alpha_n/n \to 0$ and $\alpha_n/\sqrt{n}\to \infty$ when $n\to\infty$. Let $ k_n=pn-\alpha_n$ and $f_n=k_n/n$; then 
	\begin{align}
&\lim_{n \to \infty}
	B_{n,\lfloor k_n\rfloor}^{\downarrow}(p)
	\frac{\alpha_n}{\sqrt{n }}
	\rme^{nD(f_n\|p)}=\lim_{n \to \infty}
	B_{n,\lfloor k_n\rfloor}(p)
	\frac{\alpha_n}{\sqrt{n }}
	\rme^{nD(f_n\|p)}\\
	&= \lim_{n \to \infty}
	B_{n,\lfloor k_n\rfloor}^{\uparrow}(p)
	\frac{\alpha_n}{\sqrt{n }}
	\rme^{nD(f_n\|p)}
	= \sqrt{\frac{pq }{2\pi}}. 	\nonumber
	\end{align}
\end{corollary}
\Crref{cor:ModerateDev} shows that 
\begin{align}
B_{n,\lfloor k_n\rfloor}(p)=\biggl[ \sqrt{\frac{pq }{2\pi}}+o(1)\biggr]
\frac{\sqrt{n}}{\alpha_n}\rme^{-nD(k_n\|p)}.
\end{align}
This corollary follows from \thref{TH2} and the following equations [cf. \eref{eq:LUnfplim}]:
\begin{align}
&\lim_{n\to\infty}\bigl[nD(\lfloor k_n\rfloor/n\|p)-nD(f_n\|p)\bigr]=0,\\
&\lim_{n\rightarrow \infty} L(n,\lfloor k_n\rfloor,p)\frac{\alpha_n}{n}=\lim_{n\rightarrow \infty} L(n,k_n,p)\frac{\alpha_n}{n}
=\lim_{n\rightarrow \infty} U(n,\lfloor k_n\rfloor,p)\frac{\alpha_n}{n}\label{eq:LUlimModerate}\\
&=\lim_{n\rightarrow \infty} U(n,k_n,p)\frac{\alpha_n}{n}
=pq.\nonumber
\end{align}

\section{Binomial probabilities}\label{sec:BinProb}

\subsection{Preliminary results}
To obtain upper and lower bounds for the tail probability
$B_{n,k}(p)$, here we prepare several preliminary results on  $b_{n,k}(p)$ and $B_{n,k}(p)$ as well as their relations, assuming that $k,n\in \bbN_0$, $k\leq n$, and $0<p<1$.  If there is no danger of confusion, we shall use $b_{n,k}$ and $B_{n,k}$ as shorthands for $b_{n,k}(p)$ and $B_{n,k}(p)$, respectively.

The definitions in \eref{eq:Bnkp} imply that 
\begin{gather}
B_{n,k}=B_{n,k-1}+b_{n,k},\quad 
B_{n+1,k}=p B_{n,k-1}+q B_{n,k}, \label{eq:Bn+1k=Bnk-1+Bnk}\\
B_{n,k}-B_{n+1,k}=p (B_{n,k}-B_{n,k-1})=p b_{n,k}, \label{eq:Deltank}\\
\frac{b_{n,k-1}}{b_{n,k}}=\frac{q k}{p(n-k+1)}, \quad \frac{B_{n,k}}{b_{n,k}}\geq1+\frac{b_{n,k-1}}{b_{n,k}}\geq 1+\frac{q k}{p n },
\label{eq:bnk-1/bnk}\\ \frac{b_{n+1,k}}{b_{n,k}}=\frac{(n+1)q}{n+1-k}. \label{eq:bn+1k/bnk}
\end{gather}
Here it is understood that $B_{n,k-1}=b_{n,k-1}=0$ whenever $k=0$. Based on these simple observations we can derive various preliminary results on $b_{n,k}$ and $B_{n,k}$ as follows.

\begin{lemma}\label{lem:Bbnk-jp}
	Suppose  $j,k,n\in \bbN_0$ satisfy  $0\leq j\leq k\leq n$ and $k\geq 1$, and $0<p<1$. Then  $B_{n,k-j}(p)/b_{n,k}(p)$ is strictly increasing in  $k$, but strictly decreasing in $n$ and $p$. If in addition $j\geq 1$, then  $b_{n,k-j}(p)/b_{n,k}(p)$ and $B_{n,k-j}(p)/B_{n,k}(p)$ are strictly increasing in $k$, but strictly decreasing in $n$ and $p$. Furthermore,
	\begin{align}\label{eq:Bbnk-jplim}
	\lim_{n\to\infty}\frac{b_{n,k-j}(p)}{b_{n,k}(p)}=\lim_{n\to\infty}\frac{B_{n,k-j}(p)}{b_{n,k}(p)}=\lim_{n\to\infty}\frac{B_{n,k-j}(p)}{B_{n,k}(p)}=\begin{cases} 1 & j=0,\\
	0 &j\geq 1.
	\end{cases}
	\end{align}	
\end{lemma}
Here we assume that $n,k,p$ can vary independently under the constraint specified in the lemma. To be concrete, the monotonicity of $B_{n,k-j}(p)/b_{n,k}(p)$ with respect to $k$ means 
\begin{align} 
\frac{B_{n,k-j}}{b_{n,k}}< \frac{B_{n,k'-j}}{b_{n,k'}},\quad 0\leq j\leq k< k'\leq n. \label{eq:bnkmjbnk}
\end{align}
Similar remarks apply to other conclusions concerning monotonicity properties.

\begin{proof}
	From \eref{eq:bnk-1/bnk} we can deduce that
	\begin{align}
	\frac{b_{n,k-j}}{b_{n,k}}&=\prod_{i=0}^{j-1}\frac{b_{n,k-i-1}}{b_{n,k-i}}=\prod_{i=0}^{j-1}
	\frac{q (k-i)}{p(n-k+i+1)}, \quad 1\leq j\leq k,\label{eq:bnk-j/bnk}
	\end{align}
	which implies that $b_{n,k-j}/b_{n,k}$ is strictly increasing in $k$, but strictly decreasing in $n$ and $p$ when $1\leq j\leq k$. Then, according to the following equation,
	\begin{align}
	\frac{B_{n,k-j}}{b_{n,k}}=\sum_{i=j}^{k} \frac{b_{n,k-i}}{b_{n,k}}, \quad \frac{B_{n,k}}{B_{n,k-j}}=1+\sum_{i=0}^{j-1} \frac{b_{n,k-i}}{B_{n,k-j}},
	\end{align}
	$B_{n,k-j}/b_{n,k}$ is strictly increasing in $k$, but strictly decreasing in $n$ and $p$ when $0\leq j\leq k$; by contrast, $B_{n,k-j}/B_{n,k}$ is strictly increasing in $k$, but strictly decreasing in $n$ and $p$ when $1\leq j\leq k$.
	
	Furthermore, \eref{eq:bnk-j/bnk} implies that
	\begin{align}
	\lim_{n\to\infty}\frac{b_{n,k-j}}{b_{n,k}}=\begin{cases} 1 & j=0,\\
	0 &j\geq 1,
	\end{cases}
	\end{align}
	which in turn implies \eref{eq:Bbnk-jplim}. 
\end{proof}

\begin{lemma}\label{lem:Bnkpmono}
	Suppose  $k,n\in \bbN_0$ satisfy  $0\leq k\leq n$ and $0<p<1$. Then 	$B_{n,k}(p)$ is strictly increasing in $k$, but strictly decreasing in $n$. In addition, $B_{n,k}(p)$ is strictly decreasing in $p$ when $k<n$. 
\end{lemma}
If  $0\leq p\leq 1$ instead, then $B_{n,k}(p)$ is nondecreasing in  $k$ and nonincreasing in  $n$ and $p$ by continuity.

\begin{proof}
	According to \eqsref{eq:Bn+1k=Bnk-1+Bnk}{eq:Deltank}, $B_{n,k}(p)$ is strictly increasing in $k$, but strictly decreasing in $n$. When $k<n$, $B_{n,k}(p)=B_{n,k}(p)/B_{n,n}(p)$ is strictly decreasing in $p$ according to \lref{lem:Bbnk-jp}.
\end{proof}

\begin{lemma}\label{lem:Bnkpbnkpf}
	Suppose $0\leq  f\leq 1$, $0<p<1$, and  $n\in \bbN$. Then $B_{n,fn}(p)/b_{n,fn}(p)=1$ is independent of $n$ when $f=0$, but
 is strictly increasing in $n\in \bbN_f$ when $f>0$. If in addition  $f<p$, then
	\begin{align}
\frac{B_{n,\lfloor fn\rfloor}(p)}{b_{n,\lfloor fn\rfloor}(p)}\leq  \frac{1}{1-r}=\frac{(1-f)p}{p-f},\quad 	\lim_{n\to \infty} \frac{B_{n,\lfloor fn\rfloor}(p)}{b_{n,\lfloor fn\rfloor}(p)}=\frac{1}{1-r}=\frac{(1-f)p}{p-f}, \label{eq:BnfnbnfnUBlim}
	\end{align}	
 and the inequality is saturated iff $f=0$.
\end{lemma}
Here $\bbN_f$ is defined in \eref{eq:Nf} and  $r=fq/[(1-f)p]$ is the odds ratio defined in \eref{eq:OddsRatio}.
The limit in \eref{eq:BnfnbnfnUBlim} recovers \eref{eq:Bnkbnklim7}.

\begin{proof}
	If $f=0$, then $r=0$ and $B_{n,fn}/b_{n,fn}=1$  is independent of $n$, so  \eref{eq:BnfnbnfnUBlim} holds and the inequality is saturated.
	
Next, we suppose $f>0$ and  $fn<1$; then $\lfloor fn\rfloor=0$ and  $B_{n,\lfloor fn\rfloor}/b_{n,\lfloor fn\rfloor}=1$. In addition, $0<r<1$ when  $0<f<p$, so the inequality in \eref{eq:BnfnbnfnUBlim} holds and is strict.

Next, suppose $f>0$ and  $fn\geq1$; then $B_{n,\lfloor fn\rfloor}/b_{n,\lfloor fn\rfloor}>1=B_{n,0}/b_{n,0}$. 
Let $j$ be a nonnegative integer that satisfies $j\leq \lfloor fn\rfloor-1$. By virtue of  \eref{eq:bnk-1/bnk} we can deduce that	
\begin{gather}
\frac{b_{n,\lfloor fn\rfloor-j-1}}{b_{n,\lfloor fn\rfloor-j}}=\frac{q (\lfloor fn\rfloor-j)}{p(n-\lfloor fn\rfloor+j+1)}\leq \frac{q (fn-j)}{p(n-fn+j+1)}. \label{eq:bnfnProof0}
\end{gather}
If in addition $n\in \bbN_f$, then 	$\lfloor fn\rfloor=fn$ and
$b_{n,fn-j-1}/b_{n,fn-j}$ is strictly increasing in $n$, so $B_{n,fn}/b_{n,fn}$ is strictly increasing in $n$.

If in addition $0<f<1$ (it is not necessary to assume that $n\in \bbN_f$), then \eref{eq:bnfnProof0} implies that 
\begin{gather}
\frac{b_{n,\lfloor fn\rfloor-j-1}}{b_{n,\lfloor fn\rfloor-j}}<\frac{fq}{(1-f)p}=r,  \quad 
\lim_{n\to \infty} \frac{b_{n,\lfloor fn\rfloor-j-1}}{b_{n,\lfloor fn\rfloor-j}}=\frac{fq}{(1-f)p}=r. \label{eq:bnfnProof2}
\end{gather}	
If in addition  $0<f<p$, then $0<r<1$ and \eref{eq:bnfnProof2} implies that
\begin{align} 
\frac{B_{n,\lfloor fn\rfloor}}{b_{n,\lfloor fn\rfloor}}&< \sum_{l=0}^{\infty}r^l=\frac{1}{1-r}=\frac{(1-f)p}{p-f},\quad 
\lim_{n\to \infty} \frac{B_{n,\lfloor fn\rfloor}}{b_{n,\lfloor fn\rfloor}}\leq \frac{1}{1-r}= \frac{(1-f)p}{p-f},\\
\lim_{n\to \infty} \frac{B_{n,\lfloor fn\rfloor}}{b_{n,\lfloor fn\rfloor}}&\geq 
\lim_{n\to \infty} \frac{B_{n,\lfloor fn\rfloor}-B_{n,\lfloor fn\rfloor-j}}{b_{n,\lfloor fn\rfloor}}=\sum_{l=0}^{j-1}r^l=\frac{(1-f)p}{p-f}(1-r^j)\quad \forall j\in \bbN.
\end{align}
The two equations above imply \eref{eq:BnfnbnfnUBlim}, and the inequality in \eref{eq:BnfnbnfnUBlim} is saturated iff $f=0$ given the above discussion.
\end{proof}

\begin{lemma}\label{lem:Bbn+mkp}
	Suppose  $k,m,n\in \bbN_0$ satisfy  $0\leq k\leq n$ and $m\geq 1$,  and $0<p<1$.	Then $b_{n+m,k}(p)/b_{n,k}(p)$ and  $B_{n+m,k}(p)/B_{n,k}(p)$ are strictly increasing in $k$ and  strictly decreasing in  $p$. In addition,
	$b_{n+m,k}(p)/b_{n,k}(p)=B_{n+m,k}(p)/B_{n,k}(p)=q^m$ is independent of $n$ when $k=0$, but	
	$b_{n+m,k}(p)/b_{n,k}(p)$ and  $B_{n+m,k}(p)/B_{n,k}(p)$ are strictly decreasing in  $n$ when $k\geq 1$. Furthermore,
	\begin{align}\label{eq:Bbn+mkp}
	q^m =\frac{b_{n+m,0}(p)}{b_{n,0}(p)}\leq  \frac{B_{n+m,k}(p)}{B_{n,k}(p)}\leq \frac{b_{n+m,k}(p)}{b_{n,k}(p)} \leq q^m\Bigl(\frac{n+1}{n+1-k}\Bigr)^m. 
	\end{align}	
\end{lemma}

\begin{proof}
	From \eref{eq:bn+1k/bnk} we can deduce that	
	\begin{align}\label{eq:bn+mkp}
\frac{b_{n+m,k}}{b_{n,k}}&=\prod_{l=0}^{m-1}\frac{b_{n+l+1,k}}{b_{n+l,k}}=q^m\prod_{l=0}^{m-1}
\frac{(n+l+1)}{n+l+1-k}\leq q^m
\Bigl(\frac{n+1}{n+1-k}\Bigr)^m, 
\end{align}			
which implies that $b_{n+m,k}/b_{n,k}$ is strictly increasing $k$, but strictly decreasing in $p$. In addition,
	$b_{n+m,k}/b_{n,k}=B_{n+m,k}/B_{n,k}=q^m$ is independent of $n$ when $k=0$, but	
	$b_{n+m,k}/b_{n,k}$ is strictly decreasing in  $n$ when $k\geq 1$.

From \eref{eq:Bn+1k=Bnk-1+Bnk} we can deduce that	
	\begin{align}
	\frac{B_{n+1,k}}{B_{n,k}}&=\frac{p B_{n,k-1}+q B_{n,k}}{B_{n,k}}=q+p \frac{B_{n,k-1}}{B_{n,k}}=1-p\biggl(1-\frac{B_{n,k-1}}{B_{n,k}}\biggr).
	\end{align}
According to \lref{lem:Bbnk-jp}, 
$B_{n+1,k}/B_{n,k}$ is strictly increasing  in $k$, but strictly decreasing in  $p$, and so does $B_{n+m,k}/B_{n,k}$. 
	If in addition $k\geq 1$, then 
	$B_{n+1,k}/B_{n,k}$ is strictly decreasing in $n$, and so does $B_{n+m,k}/B_{n,k}$. In the special case $k=0$, $B_{n+m,k}/B_{n,k}=b_{n+m,k}/b_{n,k}=q^m$ is independent of $n$; note that $B_{n,-1}=0$.
	
	Next, we consider \eref{eq:Bbn+mkp}. The equality and third inequality in \eref{eq:Bbn+mkp} follow from \eref{eq:bn+mkp}.  The first and second inequalities in \eref{eq:Bbn+mkp} follow from  the following equation 
	\begin{align}
	\frac{B_{n+m,k}}{B_{n,k}}
	=\frac{\sum_{j=0}^k b_{n+m,j}}{\sum_{j=0}^k b_{n,j}}
	=\frac{\sum_{j=0}^k \frac{b_{n+m,j}}{b_{n,j}} b_{n,j}}{\sum_{j=0}^k b_{n,j}}
	\end{align}
	and the fact that $b_{n+m,j}/b_{n,j}$ is strictly increasing in $j$ as proved above.  In addition, both inequalities are strict when $k\geq 1$.
\end{proof}

\subsection{Connection with the partial mean}
 The partial mean is defined as 
\begin{align}
\mu_{n,k}(p):=\sum_{j=0}^k \frac{j b_{n,j}(p)}{B_{n,k}(p)},  \label{eq:munk}
\end{align}
which is abbreviated as $\mu_{n,k}$ if there is no danger of confusion. Here we establish a simple but important  connection between the ratio $B_{n,k}/b_{n,k}$ and the partial mean $\mu_{n,k}$, which will play a crucial role in evaluating the the tail probability $B_{n,k}$ as we shall see later.

By definition the partial mean  $\mu_{n,k}$ satisfies
$0\leq \mu_{n,k}\leq k$, where both inequalities are saturated when $k=0$, but are strict when $k\geq 1$. Additional properties of the partial mean is summarized in the  following lemma.
\begin{lemma}\label{lem:munk}
	Suppose  $k,n\in \bbN_0$, $k\leq n$, and $0<p<1$. Then $\mu_{n,k}(p)$ and $k-\mu_{n,k}(p)$ are strictly increasing in $k$. In addition,  $\mu_{n,k}(p)=0$  for $k=0$, while   $\mu_{n,k}(p)$ is strictly increasing in $n$ and $p$ for $k\geq 1$. Moreover, 
	\begin{align}
	k+1-\frac{B_{n,k}(p)}{b_{n,k}(p)}\leq \mu_{n,k}(p)\leq pn, \label{eq:munkBnkbnk}
	\end{align}
	where the lower bound is saturated iff $k=0$, while the upper bound is saturated iff $k=n$. 	
\end{lemma}

\begin{proof}[Proof of \lref{lem:munk}]
	According to the following equation,
	\begin{align}
	\mu_{n,k+1}=\sum_{j=0}^{k+1} \frac{j b_{n,j}}{B_{n,k+1}}
	=\frac{(k+1)b_{n,k+1}+\sum_{j=0}^kj b_{n,j}}{B_{n,k}+b_{n,k+1}}>\sum_{j=0}^k \frac{j b_{n,j}}{B_{n,k}}=\mu_{n,k},\quad k\leq n-1,
	\end{align}
	$\mu_{n,k}$ is strictly increasing in $k$. 
	When $k=0$, $\mu_{n,k}=k-\mu_{n,k}=0$ are independent of $n$ and $p$. When $k\geq1$, we have $0<\mu_{n,k}, k-\mu_{n,k}<k$. In addition,
	according to \lref{lem:Bbnk-jp} and 
	the following equation,
	\begin{align}
	k-\mu_{n,k}
	&=\sum_{j=0}^k \frac{(k-j) b_{n,j}}{B_{n,k}}
	=\sum_{j=0}^{k-1}\sum_{i=0}^j \frac{b_{n,i}}{B_{n,k}}
	=\sum_{j=0}^{k-1} \frac{B_{n,j}}{B_{n,k}}=\sum_{j=1}^{k} \frac{B_{n,k-j}}{B_{n,k}},
	\end{align}
	$k-\mu_{n,k}$ is strictly increasing in $k$, but strictly decreasing in $n$ and $p$, so $\mu_{n,k}$  is strictly increasing in $n$ and $p$.

	The upper bound in \eref{eq:munkBnkbnk} follows from the facts that $\mu_{n,k}\leq \mu_{n,n}$ and $\mu_{n,n}= pn$; it is saturated iff $k=n$ since the inequality  $\mu_{n,k}\leq \mu_{n,n}$ is saturated iff $k=n$.
	
Finally, we turn to the lower bound in \eref{eq:munkBnkbnk}. If $k=0$, then $\mu_{n,k}=0$, $B_{n,k}=b_{n,k}$, and $B_{n,k}/b_{n,k}=1$, so the lower bound in \eref{eq:munkBnkbnk} is saturated. 
	
	If $k\geq 1$,  let $s=b_{n,k}/B_{n,k}$;  then $0<s<1$. In addition,
	\lref{lem:Bbnk-jp} implies that  $b_{n,i}/B_{n,i}> s$ for $0\leq i< k$, which in turn implies that
	\begin{gather}
	\frac{B_{n,i-1}}{B_{n,i}}=1-\frac{b_{n,i}}{B_{n,i}}\leq 1-s,\quad  \frac{B_{n,i-1}}{B_{n,k}}\leq (1-s)^{k-i+1}
	\quad \forall 1\leq i\leq k,
	\end{gather}
	where both inequalities are strict when $i<k$.
	Therefore,
	\begin{align}
	k-\mu_{n,k}
	&=\sum_{j=0}^{k-1} \frac{B_{n,j}}{B_{n,k}}\leq \sum_{j=0}^{k-1} (1-s)^{k-j}
	=\frac{1-s-(1-s)^{k+1}}{s}< \frac{1-s}{s}=\frac{1}{s} -1,
	\end{align}
	which implies that 
	\begin{align}
	\mu_{n,k}>  k+1-\frac{1}{s}= k+1-\frac{B_{n,k}}{b_{n,k}}
	\end{align} 
and confirms the lower bound in \eref{eq:munkBnkbnk} with strict inequality. 	
This observation  completes the proof of \lref{lem:munk}. 
\end{proof}

Next, we clarify the relations between the partial mean and the ratios $B_{n,k}/B_{n+1,k}$,  $B_{n,k}/b_{n,k}$. By definitions  in \eref{eq:Bnkp} we can deduce that
\begin{align}
\frac{B_{n,k}}{B_{n+1,k}}&=\frac{\sum_{j=0}^kb_{n,j}}{\sum_{j=0}^kb_{n+1,j}}=\frac{\sum_{j=0}^k\frac{n+1-j}{(n+1)q}b_{n+1,j}}{\sum_{j=0}^kb_{n+1,j}}=\frac{n+1-\mu_{n+1,k}}{(n+1)q}. \label{eq:BnkBn+1kMunk}
\end{align}	
On the other hand, by virtue of \eref{eq:Bn+1k=Bnk-1+Bnk} we can deduce that 
\begin{align}
\frac{B_{n,k}}{B_{n+1,k}}&=\frac{B_{n,k}}{p B_{n,k-1}+q B_{n,k}}=\frac{B_{n,k}}{B_{n,k}-pb_{n,k}}=\frac{\frac{B_{n,k}}{b_{n,k}}}{\frac{B_{n,k}}{b_{n,k}}-p}.
\end{align}
The two equations above together imply that 
\begin{align}
\frac{B_{n,k}}{b_{n,k}}=\frac{p(n+1-\mu_{n+1,k})}{p(n+1)-\mu_{n+1,k}},\quad\mu_{n+1,k}=\frac{p(n+1)\bigl(\frac{B_{n,k}}{b_{n,k}}-1\bigr)}{\frac{B_{n,k}}{b_{n,k}}-p}.\label{eq:BnkbnkMunk}
\end{align}

\section{Nearly tight Bounds for the ratio $B_{n,k}(p)/b_{n,k}(p)$}\label{sec:Bbratio}

\subsection{Upper and lower bounds for $B_{n,k}(p)/b_{n,k}(p)$}\label{sec:MainBounds}
In this section,
we evaluate the ratio $B_{n,k}/b_{n,k}$ in preparation for the study of the tail probability $B_{n,k}(p)$.
The main goal  of this section is to prove \thref{thm:BnkbnkLBUB}.
To this end, we recall the functions $L(n,k,p)$,  $\kappa_1(n,p)$, $V(n,k,p,a) $,
and $U(n,k,p)$ defined in Eqs.~\eqref{eq:Lnkp}-\eqref{eq:Unkp}, and here we may consider wider parameter ranges. In addition, we prepare the following two lemmas,
which are proved in  \sref{S4-3}.

\begin{lemma}\label{lem:BnkbnkRatioLB}
	Suppose  $k\in \bbN_0$, $n\in \bbN$, $k\leq n$, $f=k/n$, and $0<p<1$.  Then 
	\begin{gather}
	\frac{B_{n,k}(p)}{b_{n,k}(p)}\geq L(n,k,p)\geq 1,
	\label{eq:BnkbnkRatioLB}
\end{gather}
where both inequalities are saturated when $k=0$, but are strict when $k\geq 1$.  In addition, the lower bound 
$L(n,k,p)$ satisfies
\begin{gather}	
 L(n,k,p)\geq
\begin{cases}
\max\Bigl\{1,\frac{(1-f)p}{p-f}-\frac{fpq(1-f) }{(p-f)^3n} \Bigr\} & k<pn, \\[1ex]
\frac{1}{2}(1+\sqrt{4q k+1}\,) & k\geq pn. \label{eq:BnkbnkRatioLB2}
\end{cases}
\end{gather}
\end{lemma}

\begin{lemma}\label{lem:BnkbnkUB2L}
	Suppose $k,n\in \bbN_0$,  $0<p<1$, and $k\leq pn$.  Then 
	\begin{gather}
	\frac{B_{n,k}(p)}{b_{n,k}(p)}\leq U(n,k,p) < 2L(n,k,p),	\label{eq:BnkbnkUB2L}
\end{gather}
where the first inequality is saturated iff $k=0$, and the upper bound
$U(n,k,p)$ satisfies
\begin{gather}		
 U(n,k,p)\leq   1+2\sqrt{q (k+p)}. 	\label{eq:BnkbnkUB2L2}
\end{gather}
\end{lemma}

By virtue of  the two lemmas, we can establish \thref{thm:BnkbnkLBUB} as follows.
\begin{proof}[Proof of \thref{thm:BnkbnkLBUB}]
	\Eref{eq:BnkbnkLBUB} in Theorem~\ref{thm:BnkbnkLBUB} follows from 
	\eref{eq:BnkbnkRatioLB} in \lref{lem:BnkbnkRatioLB} and
	\eref{eq:BnkbnkUB2L} in  \lref{lem:BnkbnkUB2L}, note that all the inequalities in these equations are strict when $k\geq 1$. The first three inequalities in 
	\eref{eq:BnkbnkUBasymp} follow from \eref{eq:BnkbnkLBUB}, the fourth inequality follows from the definition in 
\eref{eq:Unkp}, and the fifth inequality follows from the limit formulas in \eref{eq:LUnfplim} and the fact  that $V(n,f n,p,0)$ is
strictly increasing in $n$ given that $0<f<p$.
\end{proof}

When $0<f<p$ and $n$ is sufficiently large, the bounds in \thref{thm:BnkbnkLBUB} [cf. Eqs.~\eqref{eq:Lnkp}-\eqref{eq:Unkp}] can be approximated as follows,
\begin{align}
L(n,k,p)&=1+\frac{kq }{p n}+\frac{q k(k-1)}{p^2 n^2}+O(n^{-3}),\\
U(n,k,p)&=V(n,k,p,0)=1+\frac{kq }{p n}+\frac{q k(k-p)}{p^2 n^2}+O(n^{-3}),\\
L(n,f n,p)&=\frac{(1-f)p}{p-f}-\frac{fp(1-f)q }{(p-f)^3n}+O(n^{-2}),\label{NM1}\\
U(n,f n,p)&=V(n,f n,p,0)=\frac{(1-f)p}{p-f}-\frac{fp q }{(p-f)^2n}+O(n^{-2}).\label{NM2}
\end{align}
These results complement the limit formulas in  \eqsref{eq:LUnkplim}{eq:LUnfplim}. Together with 
\thref{thm:BnkbnkLBUB} they imply that 
\begin{align}
\frac{B_{n,k}(p)}{b_{n,k}(p)}&=1+\frac{kq }{p n}+O(n^{-2}),\quad
\frac{B_{n,\lfloor fn\rfloor}(p)}{b_{n,\lfloor fn\rfloor}(p)}=\frac{(1-f)p}{p-f}+O(n^{-1}).
\end{align}

\subsection{Properties of upper and lower bounds}\label{S4-2}
In this section 
we clarify the properties of upper  and lower bounds for the ratio $B_{n,k}/b_{n,k}$ that appear in \thref{thm:BnkbnkLBUB} (and  \lsref{lem:BnkbnkRatioLB}, \ref{lem:BnkbnkUB2L}). Recall that the bounds $L(n,k,p)$, $V(n,k,p,a)$, and $U(n,k,p)$ are defined explicitly in  Eqs.~\eqref{eq:Lnkp}-\eqref{eq:Unkp}. 
As we shall see shortly, many properties of these bounds match the counterparts of the ratio $B_{n,k}/b_{n,k}$, which further corroborates the significance of  \thref{thm:BnkbnkLBUB}. 
Here we will focus on these bounds directly and do not consider $b_{n,k}$ and $B_{n,k}$ explicitly. 
Accordingly, we may consider wider parameter ranges and do not assume that $k$ and $n$ are integers, unlike \thref{thm:BnkbnkLBUB}, because technically  it is easier to deal with continuous variables than discrete variables. Notably, taking derivatives is very useful in technical analysis. 
The proofs of \lsref{lem:U1nkp2L}-\ref{lem:U1nfpMono} below are relegated to \aref{app:LUproperties}.
Some of the following lemmas will be used in the proofs of \lsref{lem:BnkbnkRatioLB} and \ref{lem:BnkbnkUB2L}. Other lemmas will be useful to deriving auxiliary results later. Nevertheless, only \lref{lem:U1nkp2L} is required to prove our key result \thref{thm:BnkbnkLBUB}.

\begin{lemma}\label{lem:U1nkp2L}
	Suppose $n\geq 0$, $0<p<1$ and $0\leq k\leq p n$; then
	\begin{align}\label{eq:U1nkp2L}
	U(n,k,p)< 2L(n,k,p).
	\end{align} 
\end{lemma}

\begin{lemma}\label{lem:Lnkpmono}
	Suppose  $n\geq  -1$,  $k\geq  0$, and $0\leq p\leq 1$.  Then $L(n,k,p)$ is nonincreasing and convex in $n$ and is nondecreasing and convex in $k$. In addition,  $L(n,k,p)$ is nonincreasing and convex in $p$ when $n\geq k$.  Furthermore, 
	\begin{gather}
	1\leq L(n,k,p)\leq 1+k, \label{eq:LnkpLBUB}\\ 
	\begin{split}
	L(n,k,p)&\leq \frac{1}{2}\bigl(1+\sqrt{4q k+1}\,\bigr) \quad \mbox{if} \quad p n\geq k,\\
	L(n,k,p)&\geq \frac{1}{2}\bigl(1+\sqrt{4q k+1}\,\bigr)\quad \mbox{if} \quad p n\leq  k.  
	\end{split}
	\label{eq:LnkpLBUB2}
	\end{gather} 	
\end{lemma}

\begin{lemma}\label{lem:LnkpLUB2}
	Suppose $n>0$, $0< p< 1$,  and $0\leq k\leq p n$. Then
	\begin{gather}
	\frac{k(\sqrt{1+4q k}-1)}{2p n}+1\leq L(n,k,p)\leq 
	\frac{q k}{p n}+1 \quad \mbox{if}\quad  0\leq k\leq 1, \label{eq:LnkpUBkleq1}\\[1ex]
	{\min}\biggl\{\frac{q k}{p n},\frac{k(\sqrt{1+4q k}-1)}{2p n} \biggr\}
	+1\leq L(n,k,p)\leq 
	\frac{k\sqrt{q k}}{p n}+1 \quad \mbox{if}\quad  1\leq  k\leq p n, \label{eq:LnkpUBkg1}
	\end{gather}
	where  
	\begin{align}
	{\min}\biggl\{\frac{q k}{p n},\frac{k(\sqrt{1+4q k}-1)}{2p n } \biggr\}
	=\begin{cases}
	\frac{q k}{p n } &k \geq 1+q, \\[1ex]
	\frac{k(\sqrt{1+4q k}-1)}{2p n } &k \leq 1+q. 
	\end{cases}
	\end{align}
\end{lemma}

\begin{lemma}\label{lem:Lnfpmono}
	Suppose $0\leq f,p\leq 1$ and $n\geq 0$.  Then $L(n,f n,p)$  is nondecreasing and concave in $n$, nondecreasing and convex in $f$,  and   nonincreasing and convex in $p$. If in addition $0<f,p<1$, then $L(n,f n,p)$ is strictly increasing and strictly concave in $n$. 
	If in addition $n>0$ and  $0\leq f<p<1$, then
	\begin{gather}
	\max\biggl\{1,\frac{(1-f)p}{p-f}-\frac{fpq(1-f) }{(p-f)^3n} \biggr\}\leq L(n,f n,p)\leq \frac{(1-f)p}{p-f}=\frac{1}{1-r}, 
	\label{eq:LnfpLBUB}
	\end{gather}	
where $r=fq/[(1-f)p]$ is the odds ratio defined in \eref{eq:OddsRatio}, and both inequalities are strict when $f>0$. 	
\end{lemma}

\begin{lemma}\label{lem:U1nkpAlt}
	Suppose $n\geq 0$, $0<p<1$, and $0\leq k\leq p n$; then 
	\begin{align}
	U(n,k,p)&=\min_{a\in \bbN_0,\, a< k+1} V(n,k,p,a). \label{eq:U1nkpAlt}
	\end{align} 
	If in addition $p(n+2)\geq 2$, $k=0$, or $k\geq 1$, then 
	\begin{align}
	U(n,k,p)&=\min_{a\in \bbN_0,\, a\leq k} V(n,k,p,a).\label{eq:U1nkpAlt3}
	\end{align}
\end{lemma}

\begin{lemma}\label{lem:U1nkpMonoUB}
	Suppose $n\geq 0$, $0<p<1$, and $0\leq k\leq p n$. 
	Then $U(n,k,p)$ is strictly increasing in $k$ and
	\begin{align}\label{eq:U1nkpUB}
	U(n,k,p)\leq  1+2\sqrt{p q (n+1)}.
	\end{align}
	If in addition $p(n+2)\geq 2$,  $k=0$, or $k\geq 1$, then   $U(n,k,p)$ is  nonincreasing in $n$ and $p$, and  
	\begin{align}\label{eq:U1nkpUB2}
	U(n,k,p)&\leq 1+2\sqrt{q (k+p)}.
	\end{align}
\end{lemma}

\begin{lemma}\label{lem:U1nfpMono}
	Suppose $0<f\leq p<1$ and $n\geq 0$.  Then $U(n,f n,p)$ is strictly increasing in $n$ for $n\geq 0$ and is strictly increasing in  $f$ when $n>0$. If in addition $p(n+2)\geq 2$, $f n=0$, or $f n\geq 1$, then   $U(n,f n,p)$ is  nonincreasing in $p$. If in addition $f<p$, then 
	\begin{gather}
	U(n,f n,p)\leq V(n,f n,p,0)< \frac{(1-f)p}{p-f}=\frac{1}{1-r}. \label{eq:U12fUB}
	\end{gather}
\end{lemma}

\subsection{Proofs of \lsref{lem:BnkbnkRatioLB} and \ref{lem:BnkbnkUB2L}}\label{S4-3}

\begin{proof}[Proof of \lref{lem:BnkbnkRatioLB}]
	Let $\gamma=B_{n,k}/b_{n,k}$. Then \eref{eq:BnkbnkMunk} implies that 
	\begin{align}
	\gamma=\frac{p(n+1-\mu_{n+1,k})}{p(n+1)-\mu_{n+1,k}},\quad\mu_{n+1,k}=\frac{p(n+1)(\gamma-1)}{\gamma-p}.
	\end{align}
	By virtue of  \lref{lem:munk} we can further deduce that
	\begin{align}
	k+1\leq \gamma+\mu_{n,k} \leq \gamma+\mu_{n+1,k}	=\gamma+\frac{p(n+1)(\gamma-1)}{\gamma-p}, \label{eq:BnkbnkRatioProof}
	\end{align}
	which means 
	\begin{align}
	\gamma^2 +(p n-k-1)\gamma-(n-k)p\geq 0, \label{eq:BnkbnkRatioProof2}
	\end{align}
	given that  $\gamma=B_{n,k}/b_{n,k}\geq 1>p$. Solving this equation yields  
	\begin{align}
	\gamma\geq \frac{k+1-p n+\sqrt{(p n-k-1)^2+4(n-k)p}}{2}=L(n,k,p), \label{eq:BnkbnkRatioProof3}
	\end{align}
	which  confirms the first inequality in \eref{eq:BnkbnkRatioLB}. If $k=0$, then $B_{n,k}/b_{n,k}=L(n,k,p)=1$, so this inequality is saturated. If $k\geq 1$, then both inequalities in \eref{eq:BnkbnkRatioProof} are strict, and so are the inequalities in \eqsref{eq:BnkbnkRatioProof2}{eq:BnkbnkRatioProof3}, which means  the first inequality in \eref{eq:BnkbnkRatioLB} is strict. 
	
	The second  inequality in \eref{eq:BnkbnkRatioLB} follows from  the following equation,
	\begin{align}
	L(n,k,p)-1
	&=\frac{-(p n-k+1)+\sqrt{(p n-k+1)^2+4q k}}{2} \geq 0,
	\end{align}
where the inequality is saturated iff $k=0$.

\Eref{eq:BnkbnkRatioLB2} follows from \lsref{lem:Lnkpmono} and \ref{lem:Lnfpmono}.
\end{proof}

\begin{proof}[Proof of \lref{lem:BnkbnkUB2L}]
If $k=0$, then 	$B_{n,k}/b_{n,k}=1$. In addition, $V(n,k,p,0)=1$ and $V(n,k,p,a)>1$ when $a\geq 1$, which means $U(n,k,p)=1$. So the first inequality in \eref{eq:BnkbnkUB2L} holds and is saturated.

Next, suppose $k\geq 1$ and let $j,a$ be nonnegative integers. By \eref{eq:bnk-1/bnk},  $b_{n,j-1}/b_{n,j}< 1$ for $1\leq j\leq p n$, so 	$B_{n,k}/b_{n,k}<k+1$ given the assumption that $k\leq pn$. In addition,
\begin{align}
V(n,k,p,a)\geq k+1>\frac{B_{n,k}}{b_{n,k}}\quad \forall a\geq k,\quad
 \min_{a\in \bbN_0, a\geq k} V(n,k,p,a)\geq k+1>\frac{B_{n,k}}{b_{n,k}}. \label{eq:BnkbnkUB2Lproof}
\end{align}

If  $0\leq a\leq k-1$, then $b_{n,k-a-1}/b_{n,k-a}$ is strictly decreasing in $a$ and satisfies the inequalities  $0<b_{n,k-a-1}/b_{n,k-a}<1$ by \eref{eq:bnk-1/bnk}. Therefore,
	\begin{align}
	\frac{B_{n,k}}{b_{n,k}}&=\sum_{j=0}^k \frac{b_{n,j}}{b_{n,k}}\leq a+\sum_{j=0}^{k-a} \frac{b_{n,j}}{b_{n,k}}\leq  a+\sum_{j=0}^{k-a} \frac{b_{n,j}}{b_{n,k-a}}\leq  a+\sum_{j=0}^{k-a} \biggl(\frac{b_{n,k-a-1}}{b_{n,k-a}}\biggr)^{k-a-j}\\
	&< a+ \sum_{l=0}^\infty \biggl(\frac{b_{n,k-a-1}}{b_{n,k-a}}\biggr)^l 
=a+  \biggl(1-\frac{b_{n,k-a-1}}{b_{n,k-a}}\biggr)^{-1}\nonumber\\
	&=a+\biggl(1-\frac{q(k-a)}{p(n-k+a+1)}\biggr)^{-1}=a+\frac{p(n-k+a +1)}{p n+p-k+a}=V(n,k,p,a). \nonumber
	\end{align}
In conjunction with \eref{eq:BnkbnkUB2Lproof}, this equation shows that the first inequality in \eref{eq:BnkbnkUB2L} holds and is strict when $k\geq 1$. 

The second inequality in \eref{eq:BnkbnkUB2L} follows from \lref{lem:U1nkp2L}. \Eref{eq:BnkbnkUB2L2} follows from \eref{eq:U1nkpUB2} in 
\lref{lem:U1nkpMonoUB}.
\end{proof}

\subsection{Bounds for the partial mean $\mu_{n,k}$ and the ratio $B_{n,k}/B_{n+1,k}$}

The partial mean $\mu_{n,k}$ defined in \eref{eq:munk} 
plays a crucial role in the proof of \lref{lem:BnkbnkRatioLB} and \thref{thm:BnkbnkLBUB}, which reflects
 the importance of this quantity.
On the other hand, by virtue of  \lref{lem:BnkbnkRatioLB}, we can derive pretty good bounds for the partial mean   $\mu_{n,k}$ and the ratio $B_{n,k}/B_{n+1,k}$ as shown in the following proposition and proved in \aref{app:munkBBnkLBUB}. This result will be useful in studying statistical sampling and quantum verification\footnote{Quantum verification is actually the original motivation that leads to this work.}.

\begin{proposition}\label{pro:munkBBnkLBUB}
	Suppose  $k\in \bbN_0$, $n\in \bbN$, $k\leq n$, and $0<p<1$. Then 
	\begin{align}
	\mu_{n,k}(p)&\geq k+1-L(n-1,k,p)=\frac{pn+k+q-\sqrt{(pn-k+q)^2+4qk}}{2}	,\label{eq:mnkLBGen}\\
	&\frac{n-k+1}{(n+1)q}\leq \frac{B_{n,k}(p)}{B_{n+1,k}(p)}\leq \frac{n-k+L(n,k,p)}{(n+1)q} \label{eq:BnkBn+1kLBUB1} \\
  &\hphantom{\frac{B_{n,k}(p)}{B_{n+1,k}(p)}}	=\frac{(2-p )n-k+1+\sqrt{(p n-k+1)^2+4q k}}{2(n+1)q}. \nonumber
\end{align}
	If in addition $k\leq p n$, then 
	\begin{align}
	\mu_{n,k}(p)\geq &k+\frac{q}{2}-\frac{1}{2}\sqrt{q (4k+q )}\geq k-\sqrt{q k}, \label{eq:mnkLBsp}\\
	\frac{n-k+1}{(n+1)q}\leq &\frac{B_{n,k}(p)}{B_{n+1,k}(p)}\leq \frac{n-k+1+\frac{k\sqrt{q k}}{pn}}{(n+1)q }\leq \frac{n-k+1+\sqrt{q k}}{(n+1)q }. \label{eq:BnkBn+1kLBUB2}
	\end{align}
	If in addition $f=k/n<p$, then	
	\begin{align}
	\frac{n-k+1}{(n+1)q}&\leq \frac{B_{n,k}(p)}{B_{n+1,k}(p)}\leq \frac{n-k+\frac{(1-f)p}{p-f}}{(n+1)q }. \label{eq:BnkBn+1kLBUB3}
	\end{align}		
\end{proposition}

\section{Nearly tight bounds for the tail probabilities}\label{sec:TailBounds}

\subsection{Bounds for the probability $ b_{n,k}(p)$}\label{S5-1}
To evaluate the tail probability $B_{n,k}$, we need to clarify the properties of  the probability $b_{n,k}$
in this section.  Similar to \sref{S4-2},
here we do not assume that $k$ and $n$ are integers
except for  \psref{pro:bnkpphi} and \ref{pro:bnkpLUB},
because technically  it is easier to deal with continuous variables than discrete variables.
The proofs of \lsref{lem:phi} and \ref{lem:varphipm} below are relegated to \sref{sec:phiProof}. 

To start with we define the following functions for $n>0$ and $0\leq k\leq n$. 
\begin{align}
\rho(n)&:=\frac{\rme^n\Gamma(n+1)}{n^n},\quad 
\varrho(n):=\frac{n^{n+1/2}}{\rme^n\Gamma(n+1)},\label{eq:rhovarrho} \\
\phi(n,k)&:=\frac{\rho(n)}{\rho(k)\rho(n-k)}=\frac{\Gamma(n+1)}{\Gamma(k+1)\Gamma(n-k+1)}\frac{k^{k}(n-k)^{n-k}}{n^{n}}, \label{eq:phi(nk)}\\
\varphi(n,k)&:=\frac{\varrho(k)\varrho(n-k)}{\varrho(n)}=\frac{\Gamma(n+1)}{\Gamma(k+1)\Gamma(n-k+1)}\frac{k^{k+1/2}(n-k)^{n-k+1/2}}{n^{n+1/2}},\label{eq:varphi(nk)}
\end{align}
where it is understood that $0^0=1$. 
The following proposition clarifies the relation between $b_{n,k}$ and the functions $\phi(n,k)$ and $\varphi(n,k)$, which can be verified by simple calculation.
\begin{proposition}\label{pro:bnkpphi}
Suppose $n,k\in \bbN$; then  
\begin{gather}
\phi(n,k) =\binom{n}{k}\frac{k^{k}(n-k)^{n-k}}{n^{n}},\quad 
\varphi(n,k)=\binom{n}{k}\frac{k^{k+1/2}(n-k)^{n-k+1/2}}{n^{n+1/2}}, \label{eq:binomphi}\\[1ex]
\binom{n}{k}=\frac{n^{n}}{k^{k}(n-k)^{n-k}}\phi(n,k)
=\frac{n^{n+1/2}}{k^{k+1/2}(n-k)^{n-k+1/2}}\varphi(n,k),\\[1ex]
b_{n,k}(p)\rme^{nD(\frac{k}{n}\| p)}=\phi(n,k)=\frac{\sqrt{n}}{\sqrt{ k(n-k)}} \varphi(n,k). \label{eq:bnkphi}
\end{gather}
\end{proposition}

Thanks to this proposition, $b_{n,k}$ can be evaluated by using 
$\phi(n,k)$ and $\varphi(n,k)$.
As upper and lower bounds for $\varphi(n,k)$ (cf. \lref{lem:phi} below), we define 
\begin{align}
\varphi_+(n,k)&:=\frac{1}{\sqrt{2\pi}}\rme^{\frac{1}{12n+1}-\frac{1}{12k+1}-\frac{1}{12(n-k)+1}},\quad \varphi_-(n,k):=\frac{1}{\sqrt{2\pi}}\rme^{\frac{1}{12n}-\frac{1}{12k}-\frac{1}{12(n-k)}}. \label{eq:varphipm(nk)}
\end{align}
By virtue of the Stirling approximation \cite{Robb55,Mort10}
\begin{align}\label{eq:Stirling}
\sqrt{2\pi}\,x^{x+1/2}\rme^{-x}\rme^{\frac{1}{12x+1}}< \Gamma(x+1)< \sqrt{2\pi}\,x^{x+1/2}\rme^{-x}  \rme^{\frac{1}{12x}}\quad \forall x>0,
\end{align}
it is straightforward to prove that
\begin{gather}
\lim_{n\to \infty}n^{-1/2}\rho(n)=\sqrt{2\pi},\quad \lim_{n\to\infty}\varrho(n)=\frac{1}{\sqrt{2\pi}}, \label{eq:rhoLim} \\ 
\begin{split}
\lim_{n\to\infty}\phi(n,k)&=\frac{1}{\rho(k)}=\frac{k^k}{\rme^k\Gamma(k+1)},\quad 
\lim_{n\to\infty}\varphi(n,k)=\varrho(k)=\frac{k^{k+1/2}}{\rme^k\Gamma(k+1)}, \\
\lim_{n\to\infty} \varphi(n,fn)&=\lim_{n\to\infty} \varphi_+(n,fn)=\lim_{n\to\infty} \varphi_-(n,fn)=\frac{1}{\sqrt{2\pi}}\quad \forall 0<f<1. \label{eq:philim}
\end{split}
\end{gather}
When $n$ is large, $\varphi(n,fn)$ can be expressed as 
\begin{align}
\varphi(n,fn)=\frac{1}{\sqrt{2\pi}}\biggl(1-\frac{1-f+f^2}{12f(1-f)n}\biggr)+O(n^{-2}). 
\end{align}

Additional useful properties of $\phi(n,k)$ and $\varphi(n,k)$ are summarized in the following lemma, which is proved in \sref{sec:phiProof}.
\begin{lemma}\label{lem:phi}
	Suppose $0<k<n$ and $0<f<1$. Then $\phi(n,k)$ 	is  strictly decreasing in $n$  and strictly  logarithmically convex in $n$ and $k$, while  $\phi(n,f n)$ is strictly decreasing in $n$. By contrast,
	$\varphi(n,k)$ is strictly increasing in $n$ and	  strictly  logarithmically concave in $n$ and $k$, while  $\varphi(n,f n)$ is strictly increasing in $n$.
	Furthermore,
	\begin{gather}
	0<\frac{k^k \rme^{-k}}{\Gamma(k+1)}< \phi(n,k)<1, \label{eq:phiLUB}\\
	0<\varphi(n,k)<
	\frac{k^{k+\frac{1}{2}} \rme^{-k}}{\Gamma(k+1)}
	< \frac{1}{\sqrt{2\pi}}\rme^{-\frac{1}{12k+1}}<\frac{1}{\sqrt{2\pi}}, \label{eq:varphiLUB}\\
\varphi_-(n,k)< 
	\varphi(n,k)< \varphi_+(n,k)\leq \varphi_+(n,n/2)= \frac{1}{\sqrt{2\pi}}\rme^{-\frac{18n+1}{(6n+1)(12n+1)}}.	 \label{eq:varphiLUB2}
	\end{gather}
	If in addition $k\geq 1$, then
	\begin{align}
	\phi(n,k)> \frac{k^k \rme^{-k}}{\Gamma(k+1)}\geq \frac{1}{\rme\sqrt{k}}. \label{eq:phiLB2}
	\end{align}
	If in addition $j\leq k\leq n-j$ with $j\geq 1$, then
	\begin{align}
	\frac{1}{\sqrt{8}}\leq \varphi(2j,j)\leq  \varphi(n,k)<\frac{1}{\sqrt{2\pi}}. \label{eq:varphiLUB4}
	\end{align}
\end{lemma}
Note that (strict) logarithmic convexity implies (strict) convexity. In addition,
$\phi(n,k)=\phi(n,n-k)$ and $\varphi(n,k)=\varphi(n,n-k)$ by definition, so  \lref{lem:phi} implies that
$\phi(n,k)$ is strictly decreasing in $k$ when $0<k\leq n/2$ and strictly increasing in $k$ when $n/2\leq k<n$; by contrast, $\varphi(n,k)$ is strictly increasing in $k$ when $0<k\leq n/2$ and strictly decreasing in $k$ when $n/2\leq k<n$. 
The following lemma formalizes the intuition that the bounds $\varphi_\pm(n,k)$ in \eref{eq:varphiLUB2} become more and more accurate when $n$ increases and $k$ approaches $n/2$. 
\begin{lemma}\label{lem:varphipm}
	Suppose $0<k<n$ and $0<f<1$. Then the  functions $\varphi(n,k)/\varphi_-(n,k)$,
	$\varphi_+(n,k)/\varphi(n,k)$, and $\varphi_+(n,k)/\varphi_-(n,k)$	are strictly decreasing  in $n$ and strictly logarithmically convex in $n$ and $k$. Meanwhile, $\varphi(n,fn)/\varphi_-(n,fn)$,
	$\varphi_+(n,fn)/\varphi(n,fn)$, and $\varphi_+(n,fn)/\varphi_-(n,fn)$ 	are strictly decreasing in $n$ and strictly logarithmically convex in $f$. If in addition $1\leq k\leq n-1$, then
	\begin{gather}
	1< \frac{\varphi(n,n/2)}{\varphi_-(n,n/2)}\leq \frac{\varphi(n,k)}{\varphi_-(n,k)}\leq \frac{\varphi(2,1)}{\varphi_-(2,1)}= \frac{\sqrt{\pi}\,\rme^{1/8}}{2}, \label{eq:varphi/varphi-}\\
	1<
	\frac{\varphi_+(n,n/2)}{\varphi(n,n/2)}\leq \frac{\varphi_+(n,k)}{\varphi(n,k)}\leq \frac{\varphi_+(2,1)}{\varphi(2,1)}= \frac{2}{\sqrt{\pi}\,\rme^{37/325}}, \label{eq:varphi+/varphi}\\
	1< \frac{\varphi_+(n,n/2)}{\varphi_-(n,n/2)}\leq \frac{\varphi_+(n,k)}{\varphi_-(n,k)}\leq \frac{\varphi_+(2,1)}{\varphi_-(2,1)}= \rme^{29/2600}. \label{eq:varphi+/varphi-}
	\end{gather}	
\end{lemma}

By virtue of \lref{lem:phi} we can evaluate the probability
$b_{n,k}$ as follows.
\begin{proposition}\label{pro:bnkpLUB}
	Suppose  $k,n\in \bbN$, $1\leq k\leq n$, and $0<p<1$. Then  $b_{n,k}(p)\rme^{nD(\frac{k}{n}\| p)}$ is strictly decreasing in $n$,  but is independent of $p$. 	If in addition $k\leq n-1$, then
	\begin{gather}	
\frac{1}{\rme \sqrt{k}} \rme^{-nD(\frac{k}{n}\| p)}\leq \frac{k^k}{\rme^k k!}\rme^{-nD(\frac{k}{n}\| p)}<    b_{n,k}(p)< \sqrt{\frac{n}{n-k}}\frac{k^k}{\rme^k k!}\rme^{-nD(\frac{k}{n}\| p)},
\label{eq:bnkpLUB}\\
\frac{1}{\sqrt{2n}} \rme^{-nD(\frac{k}{n}\| p)}\leq \frac{\sqrt{n}}{ \sqrt{8k(n-k)}} \rme^{-nD(\frac{k}{n}\| p)}\leq 	b_{n,k}(p)< \frac{\sqrt{n}}{\sqrt{2\pi k(n-k)}} \rme^{-nD(\frac{k}{n}\| p)},
\label{eq:bnkpLUB2}\\
\frac{\sqrt{n}\,\varphi_-(n,k)}{ \sqrt{k(n-k)}} \rme^{-nD(\frac{k}{n}\| p)}
< b_{n,k}(p) < \frac{\sqrt{n}\,\varphi_+(n,k)}{\sqrt{k(n-k)}} \rme^{-nD(\frac{k}{n}\| p)}. \label{eq:bnkpLUB3}
	\end{gather}	
\end{proposition}
The constants in the three equations in \pref{pro:bnkpLUB} cannot be improved without further assumptions. Incidentally,   $b_{n,k}\rme^{nD(\frac{k}{n}\| p)}=1$ 
is independent of $n$ and~$p$ when $k=0$ or $k=n$. Note that this observation does not contradict the monotonicity property stated in \pref{pro:bnkpLUB} because here $k$ is proportional to $n$, but $k$ is fixed as a constant in \pref{pro:bnkpLUB}.
The lower bound in \eref{eq:bnkpLUB} still holds when $k=n$ given that $k^k/(\rme^k k!)<1$ for $k>0$ according to Theorem 1.1 in \rcite{Guo06}. 
Here
\eref{eq:bnkpLUB2}  follows from Lemma~4.7.1  in Ref.~\cite{Ash92book} and from [Chapter 10, Lemma 7] in Ref.~\cite{MacWS77book}; it implies the reverse Chernoff bound in \eref{eq:ChernoffRev2}.  The lower bound $\rme^{-nD(\frac{k}{n}\| p)}/\sqrt{2n}$ for $b_{n,k}$ is applicable whenever $n\geq1$ and implies the reverse Chernoff bound   in \eref{eq:ChernoffRevType}. In addition, \pref{pro:bnkpLUB}  provides several other alternative reverse Chernoff bounds, which improve slightly over
reverse Chernoff bounds presented in 
\rscite{CsisK11,Csis98,Haya17book,WataH17}.

\begin{proof}[Proof of \pref{pro:bnkpLUB}]
When $n>k\geq 1$,	according to \lref{lem:phi} and \eref{eq:bnkphi},  $b_{n,k}\rme^{nD(\frac{k}{n}\| p)}<1$	is strictly decreasing in $n$, but is independent of $p$. When $n=k$, we have $b_{n,k}\rme^{nD(\frac{k}{n}\| p)}=1$. So  $b_{n,k}\rme^{nD(\frac{k}{n}\| p)}$ is strictly decreasing in $n$,  but is independent of $p$ when $n\geq k\geq 1$.
	
 \Eref{eq:bnkpLUB} follows from \eqssref{eq:bnkphi}{eq:varphiLUB}{eq:phiLB2}.
\Eref{eq:bnkpLUB2}  follows from \eqsref{eq:bnkphi}{eq:varphiLUB4}. 
 	\Eref{eq:bnkpLUB3} follows from \eqsref{eq:bnkphi}{eq:varphiLUB2}. 
\end{proof}

\subsection{\label{sec:TailBnkp}Nearly tight bounds for the lower tail probability $B_{n,k}(p)$}
The main aim of this section is to prove Theorem \ref{TH2}, that is, 
to derive nearly tight bounds for the lower tail probability $B_{n,k}$. 
Before presenting our main results, we point out that 
the discussions in the previous sections can easily reproduce two existing results, the asymptotic limit in \eref{ACP} \cite{AG,BlacH59,BahaR60,DembZ10} and 
an upper bound for $B_{n,k}$ that is asymptotically tight  \cite{Ferr21}, as follows.

\begin{proposition}\label{pro:BnfRE}
	Suppose  $0<f<1$ and $n\in \bbN_f$, then   $\sqrt{n}B_{n,fn}(p)\rme^{nD(f\| p)}$ is strictly increasing in $n$. In addition,
	\begin{gather}
	b_{n,fn}(p)\leq  B_{n,fn}(p)< \frac{\rme^{-nD(f\| p)}}{\sqrt{2\pi n f(1-f)}}\frac{(1-f)p}{p-f}=\sqrt{\frac{1-f}{2\pi n f}}\,\frac{p}{p-f}\rme^{-nD(f\| p)}
	, \label{eq:BnkpREUB}\\
	\lim_{n\to \infty} \sqrt{n} B_{n,fn}(p)\rme^{nD(f\| p)}=\frac{1}{\sqrt{2\pi  f(1-f)}}\frac{(1-f)p}{p-f}=\sqrt{\frac{1-f}{2\pi f}}\,\frac{p}{p-f}. \label{eq:BnkpRElim0}
	\end{gather} 
\end{proposition}

\begin{proof}[Proof of \pref{pro:BnfRE}]
	According to \eref{eq:bnkphi} and \lref{lem:phi},  $\sqrt{n}\,b_{n,fn}\rme^{nD(f\| p)}$	is strictly increasing in $n$.  Meanwhile,  $B_{n,fn}/b_{n,fn}$ is strictly increasing in $n$  according to \lref{lem:Bnkpbnkpf} in \sref{sec:BinProb}, so $\sqrt{n}B_{n,fn}\rme^{nD(f\| p)}$ is strictly increasing in $n$.

	\Eref{eq:BnkpREUB} follows from  \eqsref{eq:BnfnbnfnUBlim}{eq:bnkpLUB2}.
	\Eref{eq:BnkpRElim0} follows from \eqssref{eq:BnfnbnfnUBlim}{eq:philim}{eq:bnkpLUB3}.
\end{proof}

\Eref{eq:BnkpREUB} reproduces  Eq.~(14) in \rcite{Ferr21}, which improves the familiar Chernoff bound for $B_{n,k}$ presented in \eref{eq:ChernoffB}. Although this bound is asymptotically tight, it is not so accurate when $n$ is not so large. To construct much better bounds, we need to introduce several additional functions. Define
\begin{align} 
\tilde{L}(n,k,p):=&\frac{n\varphi(n,k)}{\sqrt{k(n-k)}}L(n,k,p)=\sqrt{n}\phi(n,k) L(n,k,p),\label{eq:LTnkp}\\
\tilde{L}_-(n,k,p):=&\frac{n\varphi_-(n,k)}{\sqrt{k(n-k)}}L(n,k,p)=
\frac{nL(n,k,p)}{ \sqrt{2\pi k(n-k)}}\rme^{\frac{1}{12n}-\frac{1}{12k}-\frac{1}{12(n-k)}},\label{eq:LT-nkp}\\
\tilde{U}(n,k,p):=&\frac{n\varphi(n,k)}{\sqrt{k(n-k)}}U(n,k,p)=\sqrt{n}\phi(n,k) U(n,k,p), \label{eq:UTnkp}\\
\tilde{U}_+(n,k,p):=&\frac{n\varphi_+(n,k)}{\sqrt{k(n-k)}}U(n,k,p)=
\frac{nU(n,k,p)}{ \sqrt{2\pi k(n-k)}}\rme^{\frac{1}{12n+1}-\frac{1}{12k+1}-\frac{1}{12(n-k)+1}}, \label{eq:UT+nkp}
\end{align}
where  $L(n,k,p)$ and $U(n,k,p)$ are defined in \sref{sec:summary}. 
Then, using \thref{thm:BnkbnkLBUB}, \lsref{lem:phi}, \ref{lem:varphipm}, and \pref{pro:bnkpLUB} 
we can show the following evaluation of $B_{n,k}$
as a refinement  of \thref{TH2}.

\begin{theorem}\label{thm:TailBound}
	Suppose  $k,n\in \bbN$,   $0<p<1$, $k\leq pn$, and  $f=k/n$. Then 
	\begin{align}
	\frac{1}{\rme \sqrt{k}} <	\frac{L(n,k,p)}{\rme \sqrt{k}}& \leq \frac{k^kL(n,k,p)}{\rme^k k!}< B_{n,k}(p) \rme^{nD(\frac{k}{n}\| p)}< \sqrt{\frac{n}{n-k}}\frac{k^kU(n,k,p)}{\rme^k k!},	\label{eq:BnkpLUB1} 	\\[1ex] 
	\frac{1}{\sqrt{8}}< \frac{L(n,k,p)}{\sqrt{8}}  &< \sqrt{\frac{k(n-k)}{n}}B_{n,k}(p)\rme^{nD(\frac{k}{n}\| p)}	< \frac{U(n,k,p)}{\sqrt{2\pi}} < \frac{2L(n,k,p)}{\sqrt{2\pi }},  \label{eq:BnkpLUB2} \\
	\tilde{L}_-(n,k,p)&< \tilde{L}(n,k,p)
<\sqrt{n} B_{n,k}(p) \rme^{nD(\frac{k}{n}\| p)}	< \tilde{U}(n,k,p)  \label{eq:BnkpLUB3}\\  
&
< \tilde{U}_+(n,k,p) < \frac{89}{44}	\tilde{L}_-(n,k,p). \nonumber
\end{align}
If in addition  $f<p$, then 
\begin{gather}	
\sqrt{n} B_{n,k}(p) \rme^{nD(\frac{k}{n}\| p)}	<\tilde{U}(n,k,p)< \tilde{U}_+(n,k,p)< 	
	\sqrt{\frac{1-f}{2\pi f}}\,\frac{p}{p-f}. \label{eq:BnkpLUB4}
	\end{gather}
\end{theorem}
The  bounds $L(n,k,p)/\sqrt{8}$ and $U(n,k,p)/\sqrt{2\pi}$   are tight within a factor of $4/\sqrt{\pi}\approx 2.25676$ 
by \eref{eq:BnkpLUB2}. The bounds $\tilde{L}(n,k,p)$ and  $\tilde{U}(n,k,p)$ are tight within  a factor of 2 according to their definitions above and the inequality $U(n,k,p)< 2L(n,k,p)$ in 
\eref{eq:BnkbnkLBUB} in \thref{thm:BnkbnkLBUB} (cf. \lref{lem:U1nkp2L}). The  bounds $\tilde{L}_-(n,k,p)$ and $\tilde{U}_+(n,k,p)$ are tight within  a factor of $89/44$ by \eref{eq:BnkpLUB3}. In conjunction with \eref{eq:varphi/varphi-}, we can actually deduce that the  bound $\tilde{L}_-(n,k,p)$ is  tight within  a factor of $\sqrt{\pi}\,\rme^{1/8}\approx 2.00845$. In addition, the four bounds $\tilde{L}(n,k,p)$, $\tilde{U}(n,k,p)$,  $\tilde{L}_-(n,k,p)$,  and $\tilde{U}_+(n,k,p)$
 are  asymptotically tight.

Note that Eqs.~\eqref{eq:TH2E0}-\eqref{eq:TH2E2} in  \thref{TH2} are simple corollaries  of Eqs.~\eqref{eq:BnkpLUB2}-\eqref{eq:BnkpLUB4}, respectively.  In conjunction with \eqsref{eq:LUnkplim}{eq:LUnfplim} in \sref{sec:summary}, \thref{thm:TailBound} yields the following corollary.

\begin{corollary}\label{cor:TailLim}
	Suppose  $k,n\in \bbN$ and   $0<f<p<1$. Then
	\begin{gather}
	\lim_{n\to \infty}	\frac{L(n,k,p)}{(\rme/k)^k k!}= \lim_{n\to \infty} B_{n,k}(p) \rme^{nD(\frac{k}{n}\| p)}= \lim_{n\to \infty} \sqrt{\frac{n}{n-k}}\frac{U(n,k,p)}{(\rme/k)^k k!}=
	\frac{k^k }{\rme^{k} k!}, \label{eq:BnkpRElim}\\
	\lim_{n\rightarrow \infty} \tilde{L}_-(n,f n,p)
=\lim_{n\rightarrow \infty} \tilde{U}_+(n,f n,p)=\sqrt{\frac{1-f}{2\pi f}}\,\frac{p}{p-f}. \label{eq:BnfpRElim}		
\end{gather} 	
\end{corollary}
In conjunction with \eref{eq:BnkpLUB3},
\eref{eq:BnfpRElim} yields an alternative proof of \eref{eq:BnkpRElim0}. It
implies that the bounds $B_{n,k}^{\downarrow}(p)$ and $B_{n,k}^{\uparrow}(p)$ defined in \eref{eq:BnkpArrow} satisfy the condition of asymptotic tightness in \eqsref{eq:ATratio}{eq:TIGHT}.

\begin{proof}[Proof of \crref{cor:TailLim}]
	\Eref{eq:BnkpRElim} is a simple corollary of 
	\eqsref{eq:LUnkplim}{eq:BnkpLUB1}.
	\Eref{eq:BnfpRElim} is a simple corollary of  \eref{eq:LUnfplim} in addition to the definitions in \eqsref{eq:LT-nkp}{eq:UT+nkp}. 
\end{proof}

\begin{proof}[Proof of \thref{thm:TailBound}]
Thanks to \eref{eq:bnkphi},
	$B_{n,k}\rme^{nD(\frac{k}{n}\| p)}$ can be expressed as follows,
	\begin{gather}\label{eq:BnkRE}
	B_{n,k}\rme^{nD(\frac{k}{n}\| p)}=\frac{B_{n,k}}{b_{n,k}}b_{n,k}\rme^{nD(\frac{k}{n}\| p)}=\frac{B_{n,k}}{b_{n,k}}\phi(n,k)=\frac{B_{n,k}}{b_{n,k}}\frac{\sqrt{n}}{\sqrt{ k(n-k)}} \varphi(n,k),
	\end{gather}	
so  \thref{thm:TailBound} follows from \thref{thm:BnkbnkLBUB}, 
 \lsref{lem:phi}, \ref{lem:varphipm}, and \pref{pro:bnkpLUB}.	
	More specifically,  \eref{eq:BnkpLUB1} follows from \eqsref{eq:BnkbnkLBUB}{eq:bnkpLUB}, note that all inequalities in \eref{eq:BnkbnkLBUB} are strict given the assumption $k\geq 1$. 
\Eref{eq:BnkpLUB2} follows from \eqsref{eq:BnkbnkLBUB}{eq:bnkpLUB2}.

The first and fourth inequalities in 
\eref{eq:BnkpLUB3} 	follow from 
\eref{eq:varphiLUB2} in addition to the definitions in Eqs.~\eqref{eq:LTnkp}-\eqref{eq:UT+nkp}. 
The second and third inequalities in 
\eref{eq:BnkpLUB3} follow from \eqsref{eq:BnkbnkUBasymp}{eq:bnkphi}. 
The last inequality in \eref{eq:BnkpLUB3} can be proved as follows,
	\begin{align}
	\tilde{U}_+(n,k,p)\leq \frac{2\varphi_+(n,k)}{\varphi_-(n,k)}\tilde{L}_-(n,k,p)\leq 2\rme^{\frac{29}{2600}}\tilde{L}_-(n,k,p)
	<\frac{89}{44}\tilde{L}_-(n,k,p),\label{NZD}
	\end{align}	
	where the first inequality  follows from the last inequality in  \eref{eq:BnkbnkLBUB} and the definitions in \eqsref{eq:LT-nkp}{eq:UT+nkp}, while the second inequality follows from \eref{eq:varphi+/varphi-}.
	
The first two inequalities in \eref{eq:BnkpLUB4} follow from \eref{eq:BnkpLUB3}, and the last inequality in \eref{eq:BnkpLUB4} follows from
\eref{eq:BnkbnkUBasymp} and the fact that $\varphi_+(n,k)<1/\sqrt{2\pi}$. 
\end{proof}

\subsection{Auxiliary results}
\Eref{eq:BnkpLUB3} in \thref{thm:TailBound} offers four bounds  for the quantity $\sqrt{n} B_{n,k} \rme^{nD(\frac{k}{n}\| p)}$
that are universally bounded and asymptotically tight. Here we discuss the properties of these bounds, which may be useful in certain applications. As in \sref{S4-2},
here we do not assume that $k$ and $n$ are integers.

\begin{proposition}\label{pro:LUTmono}
	Suppose $0<f\leq p<1$ and $n\geq 0$.  Then $\tilde{L}(n,fn,p)$, $\tilde{L}_-(n,fn,p)$, $\tilde{U}(n,fn,p)$, and $\tilde{U}_+(n,fn,p)$ are  strictly increasing in $n$. 	
\end{proposition}

\begin{proof}[Proof of \pref{pro:LUTmono}]
By \lsref{lem:Lnfpmono} and \ref{lem:U1nfpMono}, $L(n,fn,p)$ and $U(n,fn,p)$ are strictly increasing in $n$. In addition, $\varphi(n,fn)$ is strictly increasing in $n$ by \lref{lem:phi}, while  $\varphi_+(n,fn)$ and $\varphi_-(n,fn)$ are strictly increasing in $n$ by straightforward calculation. 
 Therefore, $\tilde{L}(n,fn,p)$, $\tilde{L}_-(n,fn,p)$,  $\tilde{U}(n,fn,p)$, and $\tilde{U}_+(n,fp,p)$ are  strictly increasing in $n$ given their definitions in Eqs.~\eqref{eq:LTnkp}-\eqref{eq:UT+nkp}.  
\end{proof}

When $0<f<p<1$ and $n$ is sufficiently large, $\tilde{L}(n,fn,p)$,  $\tilde{L}_-(n,fn,p)$,  $\tilde{U}(n,fn,p)$, and $\tilde{U}_+(n,fn,p)$ can be approximated 
as  follows according to their definitions: 
\begin{align}
\tilde{L}_-(n,fn,p)&=\tilde{L}(n,fn,p)+O(n^{-2})=\sqrt{\frac{1-f}{2\pi f}}\,\frac{p}{p-f}-\frac{\beta_1}{n}+O(n^{-2}),\\
\tilde{U}_+(n,fn,p)&=\tilde{U}(n,fn,p)+O(n^{-2})=\sqrt{\frac{1-f}{2\pi f}}\,\frac{p}{p-f}
-\frac{\beta_2}{n}+O(n^{-2}),
\end{align}
where the coefficients $\beta_1$ and $\beta_2$ are defined as
\begin{align}
\beta_1&:=\frac{p[13 f^2 - 13 f^3 + f^4 + (-2 f - 10 f^2 + 10 f^3) p + (1 - f + f^2) p^2]}{12\sqrt{2\pi f(1-f)}\,f(p-f)^3},\\ 
\beta_2&:=\frac{p[-f + 13 f^2 - f^3 + (1 - f - 11 f^2) p]}{12\sqrt{2\pi f(1-f)}\,f(p-f)^2}.
\end{align}
For all these bounds, the deviations from the asymptotic limits have order $O(1/n)$.

\subsection{Nearly tight bounds for the upper tail probability $\bar{B}_{n,k}(p)$}
Here  we clarify the  properties of the upper tail probability $\bar{B}_{n,k}(p)$ defined as follows,
\begin{align}
\bar{B}_{n,k}(p):=& \sum_{j=k}^nb_{n,j}(p)= \sum_{j=k}^n {n \choose j} p^j q^{n-j}=B_{n,n-k}(q)=1-B_{n,k-1}(p), \label{eq:UTailBinom}
\end{align} 
where $q=1-p$.
Thanks to this equation, 
most results on the lower tail probability $B_{n,k}$ have analogs for the upper tail probability $\bar{B}_{n,k}$. 
For simplicity here we present a few main results. The odds ratio tied to the upper tail probability is defined as  
\begin{align}
\bar{r}:=\frac{(1-f)p }{fq}=\frac{1}{r},
\end{align}
where $r$ is the odds ratio tied to the lower tail probability as presented in \eref{eq:OddsRatio}. By virtue of  Eqs.~\eqref{eq:Lnkp}-\eqref{eq:Unkp} we can 
define
\begin{align}
\bar{L}(n,k,p)&:=L(n,n-k,q)=\frac{p n-k+1+\sqrt{(p n-k+1)^2+4q k}}{2} \label{eq:LnkpBar}
\end{align}
and define $\bar{V}(n,k,p,a)$ and $\bar{U}(n,k,p)$  in a similar way. The following theorem is a simple corollary of  \eref{eq:UTailBinom} and \thref{thm:BnkbnkLBUB}. 
\begin{theorem}\label{thm:BnkbnkLBUBbar}
	Suppose  $k,n\in \bbN$, $0<p<1$,    $p n \leq k\leq n$, and  $f=k/n$. Then 
	\begin{gather}
	1\leq \bar{L}(n,k,p)\leq 	\frac{\bar{B}_{n,k}(p)}{b_{n,k}(p)}\leq \bar{U}(n,k,p)<  2\bar{L}(n,k,p),
	\end{gather}	
where all inequalities are strict when $k\leq n-1$.	If in addition $p<f<1$, then 	
	\begin{gather}		
	1<\bar{L}(n,k,p)<\frac{\bar{B}_{n,k}(p)}{b_{n,k}(p)}< \bar{U}(n,k,p)\leq  \bar{V}(n,k,p,0)< \frac{fq}{f-p}=\frac{1}{1-\bar{r}}=\frac{r}{r-1}.
	\end{gather}
\end{theorem}
In analogy to \thref{thm:BnkbnkLBUB},
here  the upper bound $\bar{U}(n,k,p)$ and lower bound $\bar{L}(n,k,p)$ are tight within a factor of 2 and are  asymptotically tight; in addition, they  can be computed in $O(1)$ time. Previously,  McKay also derived good upper and lower bounds for the ratio $\bar{B}_{n,k}(p)/b_{n,k}(p)$ based on a completely different approach \cite{McKa89}. Comparison 
between our bounds and his bounds is presented in \aref{app:comparison}.

Next, by virtue of Eqs.~\eqref{eq:LTnkp}-\eqref{eq:UT+nkp} we can define
\begin{align}
\begin{split}
\hat{L}(n,k,p)&:=\tilde{L}(n,n-k,q)
\end{split}
\end{align}
and define $\hat{L}_-(n,k,p)$,  $\hat{U}(n,k,p)$,  $\hat{U}_+(n,k,p)$ in a similar way. Thanks to \eref{eq:UTailBinom} and the equality $D(f\|p)=D(1-f\|q)$,  \thref{thm:UTailBound} and \crref{cor:UTailLim} below  are simple corollaries of \thref{thm:TailBound} and \crref{cor:TailLim}, respectively. 
\begin{theorem}\label{thm:UTailBound}
	Suppose  $k,n\in \bbN$,   $0<p<1$,  $pn\leq k<n$, and  $f=k/n$. Then 
	\begin{align}
\frac{1}{\rme \sqrt{m}}& <	\frac{\bar{L}(n,k,p)}{\rme \sqrt{m}} \leq \frac{m^m\bar{L}(n,k,p)}{\rme^{m}m!}
< \bar{B}_{n,k}(p) \rme^{nD(\frac{k}{n}\| p)}< \sqrt{\frac{n}{k}}\frac{m^m\bar{U}(n,k,p)}{\rme^{m}m!},	\label{eq:BnkpLUB1Bar} \\[1ex]
\frac{1}{\sqrt{8}}&< \frac{\bar{L}(n,k,p)}{\sqrt{8}}  < \sqrt{\frac{km}{n}}\bar{B}_{n,k}(p)\rme^{nD(\frac{k}{n}\| p)}	< \frac{\bar{U}(n,k,p)}{\sqrt{2\pi}} 
< \frac{2\bar{L}(n,k,p)}{\sqrt{2\pi }}, \label{eq:BnkpLUB2Bar}\\[1ex]
&\hat{L}_-(n,k,p)< \hat{L}(n,k,p)<	\sqrt{n} \bar{B}_{n,k}(p) \rme^{nD(\frac{k}{n}\| p)}	< \hat{U}(n,k,p) \label{eq:BnkpLUB3Bar}\\
&\hphantom{\hat{L}_-(n,k,p)}< \hat{U}_+(n,k,p)< \frac{89}{44}	\hat{L}_-(n,k,p), \nonumber
\end{align}	
where $m=n-k$.	
If in addition  $f>p$, then 
	\begin{gather}
\sqrt{n} \bar{B}_{n,k}(p) \rme^{nD(\frac{k}{n}\| p)}< \hat{U}(n,k,p)< \hat{U}_+(n,k,p)<	
	\sqrt{\frac{f}{2\pi (1-f)}}\,\frac{q}{f-p}. \label{eq:BnkpLUB4Bar}
	\end{gather}	
\end{theorem}

In analogy to \thref{thm:TailBound}, the  bounds $\bar{L}(n,k,p)/\sqrt{8}$ and $\bar{U}(n,k,p)/\sqrt{2\pi}$   are tight within a factor of $4/\sqrt{\pi}$. The bounds $\hat{L}(n,k,p)$ and  $\hat{U}(n,k,p)$ are tight within  a factor of 2. The  bounds $\hat{L}_-(n,k,p)$ and $\hat{U}_+(n,k,p)$ are tight within  a factor of $89/44$.  In addition, the four bounds $\hat{L}(n,k,p)$, $\hat{U}(n,k,p)$,  $\hat{L}_-(n,k,p)$,  and $\hat{U}_+(n,k,p)$
are asymptotically tight.

\begin{corollary}\label{cor:UTailLim}
	Suppose  $k,n\in \bbN$ and   $0<p<f<1$; then 
	\begin{gather}
	\lim_{n\to \infty}	\frac{k^kL(n,k,p)}{\rme^k k!}= \lim_{n\to \infty} \bar{B}_{n,n-k}(q) \rme^{nD(\frac{k}{n}\| p)}= \lim_{n\to \infty} \frac{\sqrt{n}\,k^kU(n,k,p)}{\sqrt{n-k}\,\rme^k k!}=
	\frac{k^k }{\rme^{k} k!},\\
	\lim_{n\rightarrow \infty} \hat{L}_-(n,f n,p)=\lim_{n\rightarrow \infty} \hat{U}_+(n,f n,p)=\sqrt{\frac{f}{2\pi (1-f)}}\,\frac{q}{f-p}.  \label{eq:BnfpRElimBar} 
	\end{gather}
\end{corollary}

If $f$ is a rational number that satisfies $p<f<1$ and $n\in \bbN_f$, then \eqsref{eq:BnkpLUB3Bar}{eq:BnfpRElimBar} imply the following result \cite{AG,BlacH59,BahaR60,DembZ10}:
\begin{gather}
\lim_{n\to \infty}\sqrt{n} \bar{B}_{n,fn}(p) \rme^{ nD(f\| p)}=\sqrt{\frac{f}{2\pi (1-f)}}\,\frac{q}{f-p}. 
\end{gather}

Thanks to \eref{eq:UTailBinom} again, the following two propositions are simple corollaries of  \psref{pro:BnfRE} and \ref{pro:LUTmono}, respectively. 
\begin{proposition}
	Suppose  $0<f<1$ and $n\in \bbN_f$, then  $\sqrt{n}\bar{B}_{n,fn}(p)\rme^{nD(f\| p)}$ is strictly increasing in $n$. 
\end{proposition}

\begin{proposition}
	Suppose $0<p\leq f<1$ and $n\geq 0$.  Then $\hat{L}(n,fn,p)$, $\hat{L}_-(n,fn,p)$, $\hat{U}(n,fn,p)$, and $\hat{U}_+(n,fn,p)$ are  strictly increasing in $n$. 	
\end{proposition}

\subsection{\label{sec:phiProof}Proofs of \lsref{lem:phi} and \ref{lem:varphipm}}
Similar to  \sref{S5-1},
in this section we do not assume that $k$ and $n$ are integers.

To prove \lref{lem:phi}, we need to prepare several auxiliary  lemmas. 
Recall that a function $\omega(x)$ is  completely monotonic \cite{SchiSV12} over an open interval $I$ if it has derivatives of all orders and 
\begin{align}
(-1)^m\omega^{(m)}(x)\geq 0\quad \forall x\in I,\quad  m=0,1,2,\ldots.  \label{eq:CMcondition}
\end{align}
The function  $\omega(x)$ is strictly completely monotonic if the inequality in \eref{eq:CMcondition} is always strict. Note that a (strictly) completely monotonic function is  in particular (strictly) decreasing and (strictly) convex. A function $g(x)$  is  (strictly) logarithmically completely monotonic if $-[\ln g(x)]'$ is (strictly) completely monotonic \cite{SchiSV12}. It is known that any function that is (strictly) logarithmically completely monotonic  is (strictly) completely monotonic \cite{QiC04}. By definition the sum of two (strictly) completely monotonic functions is (strictly) completely monotonic; the product of two (strictly) logarithmically completely monotonic functions is (strictly) logarithmically completely monotonic. The following lemma is also a simple corollary of the above definitions. 

\begin{lemma}\label{lem:SCM}
 Suppose $\omega(y)$ is (strictly) completely monotonic in $y\in (0,\infty)$ and $0<f<1$; then $\omega(fy)+\omega((1-f)y)-\omega(y)$ is (strictly) decreasing in $y\in (0,\infty)$ and (strictly) convex in $f$. If in addition $0<x<y$, then $\omega(y-x)-\omega(y)$ and 
$\omega(x)+\omega(y-x)-\omega(y)$ are (strictly) completely monotonic in $y$ and (strictly) convex in $x$. 
\end{lemma}

 Define
\begin{align}
\omega_-(x)&:=\ln \Gamma(x+1)-\ln\sqrt{2\pi}-\Bigl(x+\frac{1}{2}\Bigr)\ln x+x-\frac{1}{12x},\\
\omega_+(x)&:=\ln \Gamma(x+1)-\ln\sqrt{2\pi}-\Bigl(x+\frac{1}{2}\Bigr)\ln x+x-\frac{1}{12x+1}, \\
\varrho_+(x)&:=\rme^{\frac{1}{12x}}\varrho(x)=\frac{1}{\sqrt{2\pi}} \rme^{-\omega_-(x)}
,\quad \varrho_-(x):=\rme^{\frac{1}{12x+1}}\varrho(x)=\frac{1}{\sqrt{2\pi}} \rme^{-\omega_+(x)}, \label{eq:varrhopm}
\end{align}
where $\varrho(x)$ is defined in \eref{eq:rhoLim}.

\begin{lemma}\label{lem:SCMomega}
	The functions $\omega_+(x)$ and $-\omega_-(x)$ are strictly completely monotonic over over $x\in (0,\infty)$. 
\end{lemma}
This lemma is a combination of Theorems 1 and 2 in  \rcite{Mort10}, which state that $-\omega_-(x)$ and $\omega_+(x)$  are  completely monotonic  over $x\in (0,\infty)$. The proof of Theorem 1  in  \rcite{Mort10} actually shows that  $-\omega_-(x)$ is strictly completely monotonic. A mistake in the proof of Theorem 2 in  \rcite{Mort10} is corrected in \aref{app:SCM}. Note that the Stirling approximation in 
\eref{eq:Stirling} is a simple corollary of 
\lref{lem:SCMomega}.
The following lemma is also proved in \aref{app:SCM}.

\begin{lemma}\label{lem:SCMrho}
	Suppose $y>0$, then $1/\rho(y)$, $1/\varrho(y)$, $1/\varrho_-(y)$, and $\varrho_+(y)$	are strictly logarithmically completely monotonic and strictly completely monotonic in $y$. If in addition $0<x<y$, then 
	\begin{align}
	\rme^{\frac{1}{12y}-\frac{1}{12x}}< \frac{\varrho(x)}{\varrho(y)}=\frac{x^{x+1/2}\rme^y\Gamma(y+1)}{y^{y+1/2}\rme^x\Gamma(x+1)}< \rme^{\frac{1}{12y+1}-\frac{1}{12x+1}}. \label{eq:varrhokn}
	\end{align}
\end{lemma}
This lemma in particular implies that
$\rho(y)$ and  $\varrho(y)$ defined in \eref{eq:rhovarrho} are strictly increasing and strictly logarithmically concave  in $y$ for $y>0$ \cite{Guo06}.

\begin{proof}[Proof of \lref{lem:phi}]
	To prove \lref{lem:phi}, we first establish the monotonicity and convexity/concavity properties of $\phi(n,k)$, $\phi(n,f n)$, $\varphi(n,k)$, and $\varphi(n,f n)$. 
	The definitions in  \eqsref{eq:phi(nk)}{eq:varphi(nk)} imply that
	\begin{align}
	\ln \phi(n,k)=\ln \rho(n)-\ln \rho(k)-\ln \rho(n-k),\quad \ln \varphi(n,k)=\ln \varrho(k)+\ln \varrho(n-k)-\ln \varrho(n). \label{eq:phi(nk)Proof}
	\end{align}
	In addition,  $[\ln\rho(n)]'$ and  $[\ln\varrho(n)]'$ are strictly  completely monotonic by \lref{lem:SCMrho}, from  which it is straightforward to  deduce  
	the monotonicity and convexity/concavity properties of $\phi(n,k)$, $\phi(n,f n)$, $\varphi(n,k)$, and $\varphi(n,f n)$ stated in \lref{lem:phi} (cf. \lref{lem:SCM}). 
	
    \Eref{eq:phiLUB} follows from the limits in  \eref{eq:philim} and the fact that  $\lim_{n\to k}\phi(n,k)=1$, give that $\phi(n,k)$ is strictly decreasing in $n$. The first and fourth inequalities in  \eref{eq:varphiLUB} are obvious; the second inequality   
follows from \eref{eq:philim}, given that $\varphi(n,k)$ is strictly increasing in $n$; the third inequality follows from the Stirling approximation in \eref{eq:Stirling}.

	The first  and second inequalities in \eref{eq:varphiLUB2}  follow from the Stirling approximation in \eref{eq:Stirling} and \lref{lem:SCMrho}; 
	the third inequality in  \eref{eq:varphiLUB2}  is straightforward to verify and is saturated when $k=n/2$.

	Finally, we consider \eqsref{eq:phiLB2}{eq:varphiLUB4}. 
	The first inequality in \eref{eq:phiLB2} follows from
 \eref{eq:phiLUB};	the second inequality follows from the fact that $\varrho(k)$ is strictly increasing in $k$ by \lref{lem:SCMrho} and the fact that $\varrho(1)=1/\rme$. The third  inequality in  \eref{eq:varphiLUB4} follows from \eref{eq:varphiLUB}; the second and first inequalities  in \eref{eq:varphiLUB4} can be proved as follows,
	\begin{align}
\begin{split}
	\varphi(n,k)\geq &\varphi(j+k,k)= \varphi(j+k,j)\geq  \varphi(2j,j)\geq \varphi(j+1,j)\\
	=&
	\varphi(j+1,1)\geq \varphi(2,1)=
	\frac{1}{\sqrt{8}}.
\end{split}
	\end{align}
	Here all the inequalities follow from the assumption $1\leq j\leq k\leq n-j$ and the fact that $\varphi(n,k)$ is strictly increasing in $n$; the first two equalities follow from the fact that $	\varphi(n,k)=	\varphi(n,n-k)$. Incidentally, the first inequality  in \eref{eq:varphiLUB4} can also be regarded as a special case of the second inequality.
\end{proof}

To prove \lref{lem:varphipm}, we need to introduce one more auxiliary lemma. Define
\begin{align}
\tau(n):=\frac{1}{12n}-\frac{1}{12n+1}=\frac{1}{12n(12n+1)},\quad  \xi(n,k):=\tau(k)+\tau(n-k)-\tau(n). \label{eq:uv}
\end{align}
\begin{lemma}\label{lem:xi}
	Suppose $0<k<n$ and $0<f<1$; then $\tau(n)$ is
strictly completely monotonic. Meanwhile, 	$\xi(n,k)$ is strictly completely monotonic in $n$ and strictly convex in $k$, while $\xi(n,fn)$ is strictly decreasing in $n$ and  strictly convex in $f$.  If in addition $1\leq k\leq n-1$, then 
\begin{align}
	0< \xi(n,k)\leq \frac{29}{2600}. \label{eq:xiUB}
	\end{align}
\end{lemma}

\begin{proof}[Proof of \lref{lem:xi}]
By definition it is easy to verify that the function $1/(12n)$ is strictly completely monotonic, so $\tau(n)$ is strictly completely monotonic  according to \lref{lem:SCM}. 
Thanks to \lref{lem:SCM} again, 	$\xi(n,k)$ is strictly completely monotonic in $n$ and strictly convex in $k$, while $\xi(n,fn)$ is strictly decreasing in $n$ and  strictly convex in $f$. 

If in addition $1\leq k\leq n-1$, then 
	\begin{align}
	\xi(n,k)\leq \xi(k+1,k)=\xi(k+1,1)\leq \xi(2,1)=\frac{29}{2600}, 
	\end{align}
given that	$\xi(n,k)$	is strictly decreasing in $n$ and that 	$\xi(n,k)=\xi(n,n-k)$.  
\end{proof}

\begin{proof}[Proof of \lref{lem:varphipm}]
Straightforward calculation shows that
	\begin{align}
	\frac{\varphi(n,k)}{\varphi_-(n,k)}&=\frac{\sqrt{2\pi}\varrho_+(k)\varrho_+(n-k)}{\varrho_+(n)}=\rme^{\omega_-(n)-\omega_-(k)-\omega_-(n-k)},\\
\frac{\varphi_+(n,k)}{\varphi(n,k)}&=\frac{\varrho_-(n)}{\sqrt{2\pi}\varrho_-(k)\varrho_-(n-k)}=\rme^{-\omega_+(n)+\omega_+(k)+\omega_+(n-k)}.
\end{align}
In addition, $\omega_+(x)$ and $-\omega_-(x)$ are strictly completely monotonic over $x\in (0,\infty)$ according to \lref{lem:SCMomega}, so $\omega_-(n)-\omega_-(k)-\omega_-(n-k)$ and $-\omega_+(n)+\omega_+(k)+\omega_+(n-k)$
are  strictly  completely monotonic in $n$ and strictly  convex in $k$ by  \lref{lem:SCM}. It follows that
 $\varphi(n,k)/\varphi_-(n,k)$ and
	$\varphi_+(n,k)/\varphi(n,k)$	are  strictly logarithmically completely monotonic in $n$ and strictly logarithmically convex in $k$; in particular, they are strictly decreasing in $n$ and strictly logarithmically convex in $n$ and $k$.  Meanwhile, $\varphi(n,fn)/\varphi_-(n,fn)$ and $\varphi_+(n,fn)/\varphi(n,fn)$
	are strictly decreasing in $n$ and strictly  logarithmically convex in  $f$.

 According to \lref{lem:xi} and the following equation 
	\begin{align}
	\frac{\varphi_+(n,k)}{\varphi_-(n,k)}&=\rme^{\xi(n,k)},
	\end{align}
	the function $\varphi_+(n,k)/\varphi_-(n,k)$	is
strictly logarithmically completely monotonic in $n$ and strictly logarithmically convex in $k$; in particular, it is strictly decreasing in $n$ and strictly logarithmically convex in $n$ and $k$. Meanwhile, $\varphi_+(n,fn)/\varphi_-(n,fn)$ is strictly decreasing in $n$ and strictly logarithmically convex in $f$. 

As a corollary of the above discussions, the functions $\varphi(n,k)/\varphi_-(n,k)$,
	$\varphi_+(n,k)/\varphi(n,k)$, and $\varphi_+(n,k)/\varphi_-(n,k)$ are strictly decreasing in  $k$ when $0<k\leq n/2$ and strictly increasing in $k$ when $n/2\leq k<n$, given that  these functions are invariant when $k$ is replaced by $n-k$. 
	
	The first inequality  in \eref{eq:varphi/varphi-} follows from \eref{eq:varphiLUB2}, and the second inequality  follows from the monotonicity property of $\varphi(n,k)/\varphi_-(n,k)$ with respect to $k$ as established above. The third inequality in \eref{eq:varphi/varphi-} can be proved as follows,
	\begin{gather}
	\frac{\varphi(n,k)}{\varphi_-(n,k)}\leq \frac{\varphi(k+1,k)}{\varphi_-(k+1,k)}\leq \frac{\varphi(k+1,1)}{\varphi_-(k+1,1)}\leq \frac{\varphi(2,1)}{\varphi_-(2,1)}, 
	\end{gather}
	and the equality in \eref{eq:varphi/varphi-} can be verified by straightforward calculation. \Eqsref{eq:varphi+/varphi}{eq:varphi+/varphi-} follow from a similar reasoning. 
\end{proof}

\bigskip
\section{A conjecture on the tail probability}\label{sec:conjecture}
\subsection{The conjecture}

\Thref{thm:BnkbnkLBUB} establishes  upper and lower bounds for the ratio $B_{n,k}(p)/b_{n,k}(p)$  that are tight within a factor of 2, which lead to nearly tight upper and lower bounds for the tail probability $B_{n,k}(p)$ itself. Numerical calculation illustrated in \fref{fig:BbLratio} shows that the lower bound $L(n,k,p)$ in \thref{thm:BnkbnkLBUB}  is tight within a factor of $180451625/143327232\approx 1.25902$. 
To stimulate further progresses, here we formulate the conjecture and prove this conjecture in a special case by virtue of a surprising connection with Ramanujan's equation \cite{Rama27,JogdS68}.

\begin{conjecture}\label{con:RatioConjecture}
	Suppose  $k,n\in \bbN_0$, $k\leq n$,   $0<p<1$. Then $B_{n,k}(p)/[b_{n,k}(p) L(n,k,p)]$  is nondecreasing in $k$ and nonincreasing in $n$ and $p$. In addition,
\begin{align}
	\frac{B_{n,k}(p)}{b_{n,k}(p)}&< \frac{180451625}{143327232} L(n,k,p) \quad \forall k\leq p n, \label{eq:RatioConjecture1}\\
	\frac{B_{n,k}(p)}{b_{n,k}(p)}&< \sqrt{\frac{\pi}{2}}L(n,k,p)\quad \forall k\leq pn-1. \label{eq:RatioConjecture2}
	\end{align}	
\end{conjecture}
Incidentally,
\begin{align}
1.25901<\frac{180451625}{143327232}=\frac{5^3\times 1443613}{2^{16}\times3^7}<1.25902, \quad 1.25331 <\sqrt{\frac{\pi}{2}}<1.25332. 
\end{align}
If \cref{con:RatioConjecture} holds, then by virtue of  \eref{eq:UTailBinom} we can deduce that
\begin{align}
\frac{\bar{B}_{n,k}(p)}{b_{n,k}(p)}&< \frac{180451625}{143327232} \bar{L}(n,k,p) \quad \forall k\geq p n,\\
\frac{\bar{B}_{n,k}(p)}{b_{n,k}(p)}&< \sqrt{\frac{\pi}{2}}\bar{L}(n,k,p) \quad \forall k\geq pn+1.
\end{align}	
In addition, many results in \thsref{thm:BnkbnkLBUB}, \ref{TH2}, \ref{thm:TailBound}, and \ref{thm:UTailBound} can be improved. Notably, the lower bound  $\tilde{L}(n,k,p)$ in \thref{thm:TailBound} will be tight within a factor of $180451625/143327232$, and the lower bound
$\tilde{L}_-(n,k,p)$ will be tight within a factor of 
\begin{align}
\frac{180451625}{143327232}\frac{\sqrt{\pi}\,\rme^{1/8}}{2}< 1.26434
\end{align}
thanks to \eref{eq:varphi/varphi-}. Accordingly,
the lower bound $B_{n,k}^{\downarrow}(p)$ in \thref{TH2} will be tight within this factor,  that is,
\begin{align}
B_{n,k}^{\downarrow}(p)< B_{n,k}(p)<   \frac{180451625}{143327232}\frac{\sqrt{\pi}\,\rme^{1/8}}{2}B_{n,k}^{\downarrow}(p)< 1.26434 B_{n,k}^{\downarrow}(p). 
\end{align}

\Eqsref{eq:RatioConjecture1}{eq:RatioConjecture2} will hold if $B_{n,k}(p)/[b_{n,k}(p) L(n,k,p)]$  is indeed nonincreasing in $p$. In that case we have
\begin{align}
\frac{B_{n,k}(p)}{b_{n,k}(p)L(n,k,p)}&\leq \frac{B_{n,k}(k/n)}{b_{n,k}(k/n)L(n,k,k/n)}
< \frac{180451625}{143327232}  \quad \forall k\leq p n, \label{eq:RatioConjecture11} \\
\frac{B_{n,k-1}(p)}{b_{n,k-1}(p)L(n,k-1,p)}&\leq \frac{B_{n,k-1}(k/n)}{b_{n,k-1}(k/n)L(n,k-1,k/n)}< \sqrt{\frac{\pi}{2}}\quad \forall 1\leq k\leq pn,  \label{eq:RatioConjecture22}
\end{align}
which imply \eqsref{eq:RatioConjecture1}{eq:RatioConjecture2}. Here the second inequality in \eref{eq:RatioConjecture11} follows from \lref{lem:Bnkbnknk} below, and the second
 inequality in \eref{eq:RatioConjecture22} follows from \lref{lem:Bnk1bnk1nk} below.

\subsection{Evidences for \cref{con:RatioConjecture}}
Next, provide three lemmas that resolve \cref{con:RatioConjecture} in certain special case. The proofs of  \lsref{lem:Bnjbnjk/n}-\ref{lem:Bnkbnknk} below are relegated to Appendices \ref{app:Bnjbnjk/n} and \ref{app:Bnjbnjk/n2}. 
Our analysis also shows that the constants in the two equations in \cref{con:RatioConjecture} are best possible. 

\begin{lemma}\label{lem:Bnjbnjk/n}
	Suppose $j\in \bbN_0$ and $k,n\in \bbN$ satisfy $j\leq k\leq n-1$, and $0<f<1$; then $B_{n,j}(k/n)/b_{n,j}(k/n)$ is equal to 1 when $j=0$ and is strictly increasing in $n$ when $j\geq 1$. In addition,
\begin{gather}
\lim_{n\to\infty}\frac{B_{n,j}(k/n)}{b_{n,j}(k/n)}=\sum_{l=0}^{j}\frac{\Gamma(j+1)}{k^l\Gamma(j-l+1)}, \label{eq:Bnj/bnjlim}\\
\lim_{n\to\infty}\frac{B_{n,k-1}(k/n)}{b_{n,k-1}(k/n)}
=\sum_{l=0}^{k-1}\frac{\Gamma(k)}{k^l\Gamma(k-l) }
=\frac{\rme^k k!}{2k^k}-\theta_k
\leq \sqrt{\frac{\pi k}{2}}-\sqrt{\frac{\pi }{2}}+1
,\label{eq:Bnkbnkm1lim}\\
\lim_{n\to\infty}\frac{B_{n,k}(k/n)}{b_{n,k}(k/n)}
=1+\sum_{l=0}^{k-1}\frac{\Gamma(k)}{k^l\Gamma(k-l) }=\frac{\rme^k k!}{2k^k}+1-\theta_k 
\leq  \sqrt{\frac{\pi k}{2}}-\sqrt{\frac{\pi }{2}}+2,\label{eq:Bnkbnknklim}\\
\lim_{n\to\infty}\frac{B_{n,\lfloor fn-j\rfloor}(f)}{b_{n,\lfloor fn-j\rfloor}(f)L(n,fn-j,f)}= \sqrt{\frac{\pi }{2}}, \label{eq:Bnfbnfnflim}
\end{gather}	
where $\theta_k$ is defined by  Ramanujan's equation \cite{Rama27,JogdS68}:
	\begin{align}\label{eq:Ramanujan}
	\frac{\rme^k}{2}=\frac{\theta_k k^k}{k!}+\sum_{i=0}^{k-1} \frac{k^i}{i!}.
	\end{align}
\end{lemma}

According to \rcite{Szeg28,Wats29}, $\theta_k$ is strictly decreasing in $k$ and satisfies 
\begin{align}\label{eq:thetakLBUB}
\frac{1}{3}<\theta_k\leq \theta_1=\frac{\rme-2}{2}. 
\end{align}
\Lref{lem:Bnjbnjk/n} implies that (given the assumptions in the lemma)
\begin{gather}
\frac{B_{n,j}(k/n)}{b_{n,j}(k/n)}\leq \sum_{l=0}^{j}\frac{\Gamma(j+1)}{k^l\Gamma(j-l+1)}, \label{eq:Bnj/bnjUB}\\
\frac{B_{n,k-1}(k/n)}{b_{n,k-1}(k/n)}
\leq \sum_{l=0}^{k-1}\frac{\Gamma(k)}{k^l\Gamma(k-l) }
=\frac{\rme^k k!}{2k^k}-\theta_k
\leq \sqrt{\frac{\pi k}{2}}-\sqrt{\frac{\pi }{2}}+1, \\
\frac{B_{n,k}(k/n)}{b_{n,k}(k/n)}
< 1+\sum_{l=0}^{k-1}\frac{\Gamma(k)}{k^l\Gamma(k-l) }=\frac{\rme^k k!}{2k^k}+1-\theta_k 
\leq  \sqrt{\frac{\pi k}{2}}-\sqrt{\frac{\pi }{2}}+2,
\end{gather}
and the inequality in \eref{eq:Bnj/bnjUB} is strict when $j\geq 1$.

\begin{lemma}\label{lem:Bnk1bnk1nk}
	Suppose $k,n\in\bbN$  satisfy $1\leq k\leq n-1$; then
	\begin{gather}
	\frac{B_{n,k-1}(k/n)}{b_{n,k-1}(k/n)}< \sqrt{\frac{\pi}{2}}L(n,k-1,k/n). \label{eq:Bnk1bnk1nk}
	\end{gather}
\end{lemma}

\begin{lemma}\label{lem:Bnkbnknk}
	Suppose $k,n\in\bbN$  satisfy  $1\leq k\leq n-1$; then
	\begin{gather}
	\frac{B_{n,k}(k/n)}{b_{n,k}(k/n)}< \frac{180451625}{143327232} L(n,k,k/n). \label{eq:Bnkbnknk}
	\end{gather}
	If in addition $n\leq 2k$, then 
	\begin{gather}
	\frac{B_{n,k}(k/n)}{b_{n,k}(k/n)}< \sqrt{\frac{\pi}{2}}L(n,k,k/n). \label{eq:Bnkbnknk2}
	\end{gather}
\end{lemma}

\Eref{eq:Bnfbnfnflim} shows that the constants in \eqsref{eq:Bnk1bnk1nk}{eq:Bnkbnknk2} cannot be improved. It turns out the constant in \eref{eq:Bnkbnknk} cannot be improved either. To see this, note that
\begin{align}
\lim_{n\to\infty} L(n,k,k/n)&=\lim_{n\to\infty}\frac{1}{2}\biggl(1+\sqrt{1+\frac{4k(n-k)}{n}}\,\biggr)=\frac{1}{2}\bigl(1+\sqrt{1+4k}\,\bigr).
\end{align}
In conjunction with \lref{lem:Bnjbnjk/n} we can deduce that 
\begin{align}
\lim_{n\to\infty}\frac{B_{n,k}}{b_{n,k}L(n,k,k/n)}&= \frac{\rme^k k!}{k^k (1+\sqrt{1+4k}\,)}+\frac{2(1-\theta_k)}{1+\sqrt{1+4k}}. \label{eq:BnkbnkLlim}
\end{align}
Now, direct calculation shows that this limit is equal to the constant in \eref{eq:Bnkbnknk}, that is, $180451625/143327232$, when $k=12$. However, this value cannot be approached when $k$ deviates from 12, in sharp contrast with \eqsref{eq:Bnk1bnk1nk}{eq:Bnkbnknk2}.

\subsection{Additional evidence for  \cref{con:RatioConjecture}}
Here we provide an additional evidence for \cref{con:RatioConjecture} by considering the regime of small deviation in connection with the CLT theorem. Let $k_n=pn-x\sqrt{pqn}$ with $x\geq 0$, then by definitions in Eqs.~\eqref{eq:Lnkp}-\eqref{eq:Unkp} and \eqref{eq:BnkpArrow} we can deduce that
\begin{align}
\lim_{n\to\infty }B_{n,\lfloor k_n\rfloor}^{\downarrow}(p)=\lim_{n\to\infty }B_{n,k_n}^{\downarrow}(p)=\lim_{n\to \infty} \frac{ L(n,k_n,p)}{\sqrt{2\pi pqn}}\rme^{-x^2/2}&=\ell(x)\rme^{-x^2/2},\label{eq:BnkpLBlimCLT}\\
\lim_{n\to\infty }B_{n,\lfloor k_n\rfloor}^{\uparrow}(p)=\lim_{n\to\infty }B_{n,k_n}^{\uparrow}(p)=\lim_{n\to \infty} \frac{ U(n,k_n,p)}{\sqrt{2\pi pqn}}\rme^{-x^2/2}&=\upsilon(x)\rme^{-x^2/2}, \label{eq:BnkpUBlimCLT}
\end{align}
where 
\begin{align}\label{eq:ellupsilon}
\ell(x):=\frac{1}{2\sqrt{2\pi}} \bigl(\sqrt{4 + x^2}-x\bigr), \quad \upsilon(x):=\frac{1}{\sqrt{2\pi}}\begin{cases}
2-x & x\leq 1,\\
\frac{1}{x} &x\geq 1.
\end{cases}
\end{align}
So  \thref{TH2} implies that 
\begin{align}
\ell(x)\rme^{-x^2/2}\leq \lim_{n\to\infty }B_{n,\lfloor k_n\rfloor}(p)\leq \upsilon(x)\rme^{-x^2/2}. \label{eq:nknLimLUB}
\end{align}
On the  other hand, the CLT implies that
\begin{align}
\lim_{n\to\infty }B_{n,\lfloor k_n\rfloor}(p)&=\Phi(-x):=\int_{-\infty}^{-x} \frac{1}{\sqrt{2\pi}}\rme^{-t^2/2}dt=\int_{x}^{\infty} \frac{1}{\sqrt{2\pi}}\rme^{-t^2/2}dt, \label{eq:BnkpLimCLT}
\end{align}
where $\Phi(x)$ is the cumulative distribution function of the standard Gaussian distribution. 
The two equations above together imply that
\begin{align}
\ell(x)\rme^{-x^2/2}\leq\Phi(-x)\leq \upsilon(x)\rme^{-x^2/2}. \label{eq:PhixLUB}
\end{align}
Here the lower bound was originally derived by Birnbaum \cite{Birn42}, but we are not aware of any previous derivation that is based on lower bounds for the binomial distribution, note that it is much more common to derive bounds for the binomial distribution based on the counterparts for the Gaussian distribution, but not in the other way.

To see the connection between the above discussion and \cref{con:RatioConjecture}, note that
\begin{align}
\lim_{n\to\infty }\sqrt{n}\, b_{n,\lfloor k_n\rfloor}(p)&=\frac{1}{\sqrt{2\pi pq}}\rme^{-x^2/2}
\end{align}
according to \eqsref{eq:philim}{eq:bnkpLUB3}. In conjunction with \eqsref{eq:BnkpLBlimCLT}{eq:BnkpLimCLT} we can deduce that
\begin{align}
\lim_{n\to\infty }\frac{B_{n,\lfloor k_n\rfloor}(p)}{b_{n,\lfloor k_n\rfloor}(p)L(n,k_n,p)}=\frac{\Phi(-x)}{\ell(x)\rme^{-x^2/2}}, \label{eq:BnkpbnkpLnkpLim}
\end{align}
so the properties of the ratio $\Phi(-x)\rme^{x^2/2}/\ell(x)$ are suggestive of the properties of the ratio $B_{n,\lfloor k_n\rfloor}(p)/[b_{n,\lfloor k_n\rfloor}(p)L(n,k_n,p)]$.

 According to \eqsref{eq:BnkpLBlimCLT}{eq:BnkpUBlimCLT} and the inequality $U(n,k,p)< 2L(n,k,p)$ in  \thref{thm:BnkbnkLBUB}, the upper and lower bounds in \eqsref{eq:nknLimLUB}{eq:PhixLUB} are  tight within a factor of~2. 
The following proposition proved in \aref{app:ellPhiProof} further shows that these bounds are asymptotically tight and that the lower bound in each equation is actually  tight within a factor of $\sqrt{\pi/2}$, which together with \eref{eq:BnkpbnkpLnkpLim} offers an additional evidence for \eref{eq:RatioConjecture2} in  \cref{con:RatioConjecture}.
In addition, $\Phi(-x)\rme^{x^2/2}/\ell(x)$ is strictly decreasing in $x$, so  the lower bound $\ell(x)\rme^{-x^2/2}$ becomes more and more accurate as $x$ increases, which is also
 reminiscent of the monotonicity properties stated in \cref{con:RatioConjecture}.
\begin{proposition}\label{pro:ellPhi}
	Suppose $x\geq 0$; then $\upsilon(x)/\ell(x)$ and  $\Phi(-x)\rme^{x^2/2}/\ell(x)$ are strictly decreasing in $x$. In addition,
	\begin{gather}
	\lim_{x\to \infty}\frac{\Phi(-x)\rme^{x^2/2}}{\ell(x)}=
	\lim_{x\to \infty}\frac{\upsilon(x)}{\ell(x)}=\lim_{x\to \infty}\sqrt{2\pi} \,x\ell(x)=1, \label{eq:ellupPhiLim}	\\	
	\ell(x)< \upsilon(x)\leq 2\ell(x),\label{eq:ellupsilonLUB}\\	\ell(x)\rme^{-x^2/2}< \Phi(-x)\leq  \sqrt{\frac{\pi}{2}}\,\ell(x)\rme^{-x^2/2},\label{eq:ellPHi} 
	\end{gather}
	where the second inequality in \eref{eq:ellupsilonLUB} and that in \eref{eq:ellPHi}  are saturated iff $x=0$.	
\end{proposition}

\section{Conclusion}\label{sec:conclusion}
We derived simple but nearly tight upper bound $B_{n,k}^{\uparrow}(p)$ and lower bound $B_{n,k}^{\downarrow}(p)$
for the binomial tail probability $B_{n,k}(p)$. These bounds have a number of appealing properties, including 
(C1) $O(1)$-computability, 
(C2) Universal boundedness of the ratio $B_{n,k}^{\uparrow}(p) /B_{n,k}^{\downarrow}(p)$, 
(C3) Asymptotic tightness in the regime of large deviation, and (C3') Asymptotic tightness in the regime of moderate deviation. To the best of our knowledge, no bounds for the tail probability $B_{n,k}(p)$ known in the literature satisfy  these criteria simultaneously. By virtue of these universal  bounds, we  derived the asymptotic expansion of the tail probability $B_{n,k}(p)$
up to a constant multiplicative factor. In the course of study, we derived nearly tight upper and lower bounds for  the ratio $B_{n,k}(p)/b_{n,k}(p)$, which are of independent interest. Furthermore, we believe that our lower bound for the ratio $B_{n,k}(p)/b_{n,k}(p)$ is more accurate than what can be proved rigorously, as stated in \cref{con:RatioConjecture} and supported by strong evidences. If this conjecture holds, then the lower bound $B_{n,k}^{\downarrow}(p)$ will be tight within a factor of $1.26434$. We hope that our work can stimulate further progresses in this direction. 

In the future, it would be desirable to generalize our results to other probability distributions, such as multinomial distributions. Such extension, if available, may find diverse applications in various research areas, including statistical sampling, quantum verification, channel coding with higher order \cite{Moul17,Ferr21,Haya18}, 
security analysis with higher order \cite{Haya18,Haya19},
quantum thermodynamics \cite{TajiH17,ItoH18}, and local discrimination \cite{HayaO17}.

\begin{appendix}
\numberwithin{equation}{section}

\section{Derivation of \eref{ACP}}\label{A1}
The references \cite[Theorem 4]{BlacH59}, \cite[Case 2]{BahaR60}, \cite[Theorem 3.7.4]{DembZ10} derived general formulas for strong large deviation, but did not give the explicit formula for the binomial distribution.
The aim of this appendix is to derive the explicit formula of strong large deviation  for the binomial distribution as presented in \eref{ACP} from general results mentioned above. Here our derivation is mainly based on \cite[Theorem 3.7.4]{DembZ10}, which
yields the following proposition  in the lattice case with lattice span $d$. Note that
 the lattice span is $1$ in the case of the binomial distribution. Incidentally, a simple alternative derivation of \eref{ACP} is presented in the proof of \pref{pro:BnfRE}.

We define the cumulant generating function 
$\Lambda (s):= \ln E [e^{sX}]$, where
$E [X]$ expresses the expectation of the random variable $X$.
The inverse function of the derivative $\Lambda' (s) $ is denoted by $\eta$.

\begin{proposition}
	Suppose $X$ is a lattice variable with lattice span $d$ and its cumulant generating function  $\Lambda (s)$ is finite in some neighborhood of 0. If $R < E[X]$, 
	then the $n$-iid sum $X_n$ of the random variable $X$ satisfies 
	\begin{align}
	{\rm Pr}\{ X_n \ge n R\}=e^{-n \chi_0(R)}
	\frac{d }{ \sqrt{2\pi n \Lambda''(\eta(R) ) }\, (1-e^{-d \eta(R)})}
	[1+O(n^{-1})],\label{JJ1}
	\end{align}
	where
	\begin{align}
	\chi_0(R) := R \eta(R) - \Lambda (\eta (R)) .\label{AMF}
	\end{align}
\end{proposition}

Now, we apply the above proposition to the 
binomial upper tail probability $\bar{B}_{n,fn}(p)$ defined in \eref{eq:UTailBinom}, assuming that $f$ is a rational number and $n\in \bbN_f$.
In this case, we have
\begin{align}
\Lambda(s)= \ln(1-p+ p\rme^s), \quad \Lambda'(s)= \frac{p\rme^s}{1-p+ p\rme^s}. 
\end{align}
The inverse function of $\Lambda'$ reads
\begin{align}
\eta(f)=\ln \frac{f(1-p)}{(1-f)p},\label{AXO}
\end{align}
from which we can deduce  that
\begin{align}
\Lambda''(\eta(f))=\eta'(f)^{-1}=f(1-f).\label{AJI}
\end{align}
In addition, the definition \eqref{AMF} implies that 
\begin{align}\label{MZP}
&\chi_0(f)=f \eta(f) - \Lambda (\eta (f))=f \ln \frac{f(1-p)}{(1-f)p}-
\ln \biggl[ (1-p)+ p\frac{f(1-p)}{(1-f)p}\biggr]
= D(f\|p ).
\end{align}

Substituting Eqs.~\eqref{AXO}-\eqref{MZP} into 
\eqref{JJ1}, we can deduce that
\begin{align}
&\lim_{n \to \infty} \sqrt{n} \bar{B}_{n,fn}(p)\rme^{nD (f\|p)}
=\frac{1}{\sqrt{2\pi f(1-f)}\,\bigl[1-
	\frac{(1-f)p}{f(1-p)}
	\bigr] } \label{eq:UTailLimApp}\\
=&\frac{{f(1-p)}}{\sqrt{2\pi f(1-f)}\, (
f-p)}
=\sqrt{\frac{f}{2\pi (1-f)}}\,\frac{q}{f-p},  \nonumber
\end{align}
where $q=1-p$. 
As a simple corollary this equation yields 
\begin{align}
&\lim_{n \to \infty} \sqrt{n} B_{n,fn}(p)\rme^{nD (f\|p)}\sqrt{n}=\lim_{n \to \infty} \sqrt{n} \bar{B}_{n,(1-f)n}(q)\rme^{nD (1-f\|q)}
=\sqrt{\frac{1-f}{2\pi f}}\,\frac{p}{p-f},
\end{align}
which confirms \eref{ACP}. 

Incidentally,  reference 
\cite[(5.41)]{Moul17} derived a 
 limit for $\sqrt{n}\bar{B}_{n,fn}\rme^{nD (f\|p)}$ that is different from \eref{eq:UTailLimApp} because 
its derivation is problematic.

\section{\label{app:LUproperties}Proofs of \lsref{lem:U1nkp2L}-\ref{lem:U1nfpMono}}

\subsection{Proof of Lemma \ref{lem:U1nkp2L}}
To prove Lemma \ref{lem:U1nkp2L}, we need to consider two different parameter ranges depending on the value of $k$ in comparison with $\kappa_1(n,p)$ defined in \eref{eq:kappa1} and  the following function 
\begin{align}
\kappa_2(n,p)&:=p(n+1)-\frac{q}{4}-\frac{1}{4}\sqrt{q(8p n+7p +1)}, \label{eq:kappa2}
\end{align}
where $q=1-p$. 
Here we first prepare an auxiliary lemma to clarify the properties of $\kappa_1(n,p)$ and $\kappa_2(n,p)$ as well as their relations. 
\begin{lemma}\label{lem:kappa}
	Suppose $n\geq -1$ and $0<p<1$. Then 	$\kappa_1(n,p)$ 	is strictly convex in $n$; it is strictly decreasing in $n$ for $-1\leq  n\leq [q/(4p)]-1$, but  strictly increasing  for $n\geq [q/(4p)]-1$. 
Meanwhile, $\kappa_2(n,p)$ and $\kappa_2(n,p)-\kappa_1(n,p)$ are  strictly increasing in $n$.	In addition,
\begin{align}
\kappa_1(n,p)\geq -\frac{q}{4},\quad \kappa_2(n,p)-\kappa_1(n,p)\geq -\frac{q}{2},\quad \kappa_2(n,p)-\kappa_1(n,p)\geq -\frac{q}{2}. \label{eq:kappaIneq}
\end{align}	
Furthermore, the following four conditions are equivalent:
	\begin{enumerate}
		\item  $n> (1/p)-2$;
		\item $\kappa_1(n,p)> 0$;
		\item $\kappa_2(n,p)> 0$;
		\item $\kappa_2(n,p)-\kappa_1(n,p)> 0$. 
	\end{enumerate} 
\end{lemma}

\begin{proof}[Proof of \lref{lem:kappa}]
 According to the following equation, 
	\begin{align}
	\frac{\partial \kappa_1(n,p)}{\partial n}=p\biggl(1-\frac{q}{2\sqrt{pq( n+1)}}\biggr),\quad 
	\frac{\partial^2\kappa_1(n,p)}{\partial n^2}=\frac{\sqrt{pq( n+1)}}{4(n+1)^2} \quad \forall  n>-1,
	\end{align}
	$\kappa_1(n,p)$ is strictly decreasing in $n$ for $-1\leq  n\leq [q/(4p)]-1$, but is strictly increasing in $n$ for $n\geq [q/(4p)]-1$, given that  $\kappa_1(n,p)$ is continuous in $n$ when $n\geq -1$. 
At $n= [q/(4p)]-1$, $\kappa_1(n,p)$ attains its minimum value $-q/4$, which implies the first inequality in \eref{eq:kappaIneq}. 
	In addition, $\kappa_1(n,p)$ is strictly convex in $n$, which is also clear from its definition in \eref{eq:kappa1}.
	Furthermore, $\kappa_1(n,p)=0$ when $n=-1$ or $n= (1/p)-2$. 
So   $\kappa_1(n,p)> 0$ iff $n> (1/p)-2$ given  that $\kappa_1(n,p)$ 	is strictly convex.

	According to the following equations,
	\begin{align}
	\frac{\partial \kappa_2(n,p)}{\partial n}&=p\biggl(1-\frac{q}{\sqrt{q(8p n+7p+1)}}\biggr)>0\quad \forall   n\geq -1,\\
	\frac{\partial [\kappa_2(n,p)-\kappa_1(n,p)]}{\partial n}&=\frac{p q \bigl[\sqrt{q(8p n+7p+1)}-2\sqrt{p q(n+1)}\,\bigr]}{2\sqrt{p q(n+1)}\sqrt{q(8p n+7p+1)}}>0\quad \forall  n>-1,
	\end{align}
	$\kappa_2(n,p)$ and $\kappa_2(n,p)-\kappa_1(n,p)$ are strictly increasing in $n$ for $n\geq -1$, given that they are continuous in $n$ for $n\geq -1$. Now the second and third inequalities in   \eref{eq:kappaIneq} follow from this fact and the following equation
	\begin{align}
	\kappa_1(-1,p)=0,\quad \kappa_2(-1,p)-\kappa_1(-1,p)=\kappa_2(-1,p)=-\frac{q}{2}.
	\end{align}
	When  $n= (1/p)-2$, calculation shows that 
	\begin{align}
	\kappa_2(n,p)- \kappa_1(n,p)=\kappa_2(n,p)= \kappa_1(n,p)=0. 
	\end{align}
	Therefore,  $\kappa_2(n,p)> 0$ iff $n> (1/p)-2$; similarly, $\kappa_2(n,p)-\kappa_1(n,p)> 0$ iff $n> (1/p)-2$. 	 This observation completes the proof of \lref{lem:kappa}. 
\end{proof}

\begin{proof}[Proof of \lref{lem:U1nkp2L}]
	If $\kappa_1(n,p)\leq k\leq pn$, then the definition in \eref{eq:Unkp} implies that
	\begin{align}
	U(n,k,p)\leq V(n,k,p,\lceil\tilde{a}\rceil )< W(n,k,p),
	\end{align}
	where $\tilde{a}=k-\kappa_1(n,p)$ and 
	\begin{align}
	W(n,k,p):= 1+V(n,k,p,\tilde{a} )=1+k-p n+2\sqrt{p q(n+1)}. 
	\end{align}
	In conjunction with  \eref{eq:Lnkp} we can deduce that
	\begin{align}
	2L(n,k,p)- W(n,k,p)=\sqrt{(p n-k+1)^2+4q k}-2\sqrt{p q(n+1)}\geq 0,
	\end{align}
	which implies \eref{eq:U1nkp2L}. 
	Here the inequality follows from the following equation
	\begin{align}
	(p n-k+1)^2+4q k-4 p q(n+1)=[p(n+2)-k-1]^2\geq 0.
	\end{align}	
	
	If  $0\leq k<\kappa_1(n,p)$ and $k\leq pn$,
	then $\kappa_2(n,p)> \kappa_1(n,p)> k\geq 0$ and  $n> (1/p)-2$ by \lref{lem:kappa}. In addition,
 the definition in \eref{eq:Unkp} implies that $U(n,k,p)=V(n,k,p,0)$.  Let 
	\begin{align}
	\Delta=2L(n,k,p)-1- V(n,k,p,0)=k-p n+\sqrt{(p n-k+1)^2+4q k}-\frac{p(n+1-k)}{p(n+1)-k};
	\end{align}
	then to prove \eref{eq:U1nkp2L} it suffices to prove the inequality $\Delta\geq 0$.  Solving the equation $\Delta=0$ yields two solutions for $k$ that are not larger than  $pn+p$, that is, $	k=0$ or $k=\kappa_2(n,p)$. 
In addition, the inequality 	$n> (1/p)-2$ means 
	\begin{align}
	\frac{\partial \Delta}{\partial k}\Big|_{k=0}=\frac{q (p n+2p -1)}{p(n+1)(1+p n)}>0.
	\end{align}
	Note that  $\Delta$ is continuous in $k$ for $0\leq k\leq \kappa_2(n,p)$ and has no zero in this interval except for the end points. 	  So $\Delta\geq 0$ for  $0\leq k\leq \kappa_2(n,p)$, which implies \eref{eq:U1nkp2L} and completes the proof of \lref{lem:U1nkp2L}. 
\end{proof}

\subsection{Proof of \lref{lem:Lnkpmono}}
\Eref{eq:Lnkp} implies that
\begin{align}\label{eq:Lnkpgeq1}
L(n,k,p)-1
&=\frac{-(p n-k+1)+\sqrt{(p n-k+1)^2+4q k}}{2} \geq 0,
\end{align}	
so  $L(n,k,p)\geq 1$, and this inequality is strict when $0<p<1$ and $k>0$. Meanwhile, $L(n,k,p)$ is continuous in $n,k,p$  in the parameter range specified in \lref{lem:Lnkpmono}, so in the following discussion we can focus on the interior of this parameter range, that is, $n>-1$, $k>0$, and $0<p<1$. Then  $L(n,k,p)$ is the solution to the following equation that is larger than 1 [cf. \eref{eq:BnkbnkRatioProof}],
	\begin{align}
	L+\frac{p(n+1)(L-1)}{L-p}=k+1. \label{eq:kL}
	\end{align} 
	This equation shows that $k$ is strictly increasing and concave in $L$, so $L$ is strictly increasing and convex in $k$.  	
	
	From \eref{eq:kL} we can deduce  that 
	\begin{align}
	n=\frac{k L+L-L^2-p k}{p(L-1)}=\frac{k-L}{p}+\frac{k q }{p(L-1)},
	\end{align}
	which implies that $n$ is strictly decreasing and convex in $L$, so $L$ is strictly decreasing and convex in $n$.  Now \eqsref{eq:LnkpLBUB}{eq:LnkpLBUB2} follow from \eref{eq:Lnkpgeq1} and the following  equation, 
	\begin{gather}
	L(-1,k,p)=1+k,\quad 
	L(k/p,k,p)=\frac{1}{2}\bigl(1+\sqrt{4q k+1}\,\bigr). 
	\end{gather}
	
	If in addition $n=k>0$, then $L(n,k,p)=1+k-pk$, which is strictly decreasing and linear in $p$. 
	If  $n> k>0$, then 
	from \eref{eq:kL} we can deduce  that 
	\begin{align}
	p=\frac{L(k+1-L)}{n(L-1)+k}=\frac{kn+k-n L}{n^2}+\frac{k(n-k)(n+1)}{n^2[n(L-1)+k]},
	\end{align}
	which implies that $p$ is strictly decreasing and convex in $L$, so $L$ is strictly decreasing and convex  in $p$.

\subsection{Proof of \lref{lem:LnkpLUB2}} The proof is divided into three steps: In the first step we prove \eref{eq:LnkpUBkleq1}, in the second step we prove the lower bound in \eref{eq:LnkpUBkg1}, and in the third step we prove  the upper bound in \eref{eq:LnkpUBkg1}.\\

\textbf{Step 1:} Proof of \eref{eq:LnkpUBkleq1}. Modifying the function $L(n,k,p)$, we define
	\begin{gather}
	g(n,k,p):=n [L(n,k,p)-1],\quad 	g(\infty, k,p ):=\lim_{n\rightarrow \infty}g(n,k,p)=\frac{q k}{p}.\label{eq:gnkp}
	\end{gather}
	Then  
	\begin{gather}
	\frac{\partial g(n,k,p)}{\partial n}
	=\frac{s_2+(k-1-2p n)\sqrt{s_1}}{2\sqrt{s_1}},
	\end{gather}
	where $s_1$ and $s_2$ are defined as follows,
	\begin{gather}
	s_1: =(p n-k+1)^2+4q k>0,\\ 
      s_2:=s_1+p n(1-k+p n)= 2p^2 n^2 +3p(1-k)n+(k+1)^2-4p k, 
	\end{gather}
	which satisfy
		\begin{gather}
	s_2^2-(k-1-2p n)^2 s_1=4q k[(k-1)^2+4q k-2p(k-1)n]. 
	\end{gather}
	
	When $0\leq k\leq 1$, we have $s_2\geq 0$ and $s_2^2-(k-1-2p n)^2 s_1\geq 0$, which implies that $\partial g(n,k,p)/\partial n\geq 0$, so $g(n,k,p)$ is nondecreasing in $n$. In conjunction with \eref{eq:gnkp} we conclude that $g(n,k,p)\leq q k/p$, which implies  the upper bound in \eref{eq:LnkpUBkleq1}. Note that this upper bound holds even if $p n<k$. If $p n\geq k$ as stated in the assumption, then 
	\begin{align}
	g(n,k,p)\geq g(k/p, k,p)=\frac{k(\sqrt{1+4q k}-1)}{2p}, 
	\end{align}
	which implies  the lower  bound in \eref{eq:LnkpUBkleq1}.\\
	
\textbf{Step 2:} Proof of the lower bound in \eref{eq:LnkpUBkg1}. We first  assume that $1<k\leq p n$ to start with.
Then   $\partial g(n,k,p)/\partial n$ has a unique zero at \begin{align}
	n=n_0:=\frac{(k-1)^2+4q k}{2p(k-1)}. 
	\end{align}
	Note that 
	\begin{align}
	k-1-2p n_0=-\frac{4q k}{k-1}\leq 0,\quad s_2|_{n= n_0}=\frac{2q k[(k-1)^2+4q k]}{(k-1)^2}\geq 0. 
	\end{align}
	Meanwhile, 
	\begin{align}
	\frac{\partial g(n,k,p)}{\partial n}\bigg|_{n=0}=\frac{1}{2}\bigl[\sqrt{(k-1)^2+4q k}+k-1\bigr]>0,
	\end{align}
	while $\partial g(n,k,p)/\partial n<0$  when $n$ is sufficiently large, given that
\begin{align}
(k-1-2p n)<0,\quad s_2^2-(k-1-2p n)^2 s_1<0
\end{align}	
in that case. 	Therefore,  $\partial g(n,k,p)/\partial n\geq 0$ when $0\leq n\leq n_0$ and  $\partial g(n,k,p)/\partial n\leq  0$ when $n\geq n_0$, which means 
	\begin{align}
	g(n,k,p)&\geq{\min}\{g(\infty,k,p),g(k/p,k,p)\}={\min}\biggl\{\frac{q k}{p },\frac{k(\sqrt{1+4q k}-1)}{2p } \biggr\}\\
	&=\begin{cases}
	\frac{q k}{p } &k \geq 1+q, \\[1ex]
	\frac{k(\sqrt{1+4q k}-1)}{2p } &k \leq 1+q,
	\end{cases} \nonumber
	\end{align}
	and confirms the lower bound in \eref{eq:LnkpUBkg1} given the assumption $1<k\leq pn$.  By continuity the lower bound holds  when $1\leq k\leq pn$. \\

\textbf{Step 3:} Proof of the upper bound in \eref{eq:LnkpUBkg1}.  We also  assume that $1<k\leq p n$ to start with.	Then the above analysis implies that 
	\begin{gather}
	g(n,k,p)\leq g(n_0,k,p)=\frac{(k-1)^2+4q k}{4p},\quad 
	n_0-\frac{k}{p}=\frac{4q k -(k^2-1)}{2(k-1)p}.
	\end{gather}
	If $k\geq 2+\sqrt{5}$, then $(k^2-1)/(4k)\geq 1\geq q$. Therefore, $n_0\leq k/p$, and $g(n,k,p)$ is nonincreasing in $n$ for $n\geq k/p$, which means 
	\begin{align}\label{eq:gnkpUB}
	g(n,k,p)\leq g(k/p,k,p)=\frac{k(\sqrt{1+4q k}-1)}{2p}\leq \frac{k\sqrt{q k}}{p}\quad \mbox{if}\quad  p n\geq k. 
	\end{align}
	If $1<k< 2+\sqrt{5}$ and  $0<q \leq (k^2-1)/(4k)$, then \eref{eq:gnkpUB} holds due to a similar reason. 
	
	It remains to consider the case with $1<k< 2+\sqrt{5}$ and  $(k^2-1)/(4k)<q<1$. Let 
	\begin{align}
	s_3:=4k\sqrt{q k}-4p g(n_0,k,p )=4k\sqrt{q k}-(k-1)^2-4q k; 
	\end{align}
	then $s_3$ is concave in $q$ for $0<q <1$. In addition,
	\begin{gather}
	s_3=2k\bigl[\sqrt{k^2-1}-(k-1)\bigr]\geq 0 \quad \mbox{if} \quad  q=\frac{k^2-1}{4k}, \\
	\lim_{q\rightarrow 1} s_3=4k^{3/2}-(k+1)^2\geq 0,
	\end{gather}
	given that $1<k< 2+\sqrt{5}$. Therefore, 
	\begin{align}
	s_3\geq 0, \;\; g(n,k,p)\leq g(n_0,k,p)\leq \frac{k\sqrt{q k}}{p}\quad \mbox{if} \;\; 1<k< 2+\sqrt{5},\;\;  \frac{k^2-1}{4k}< q <1. 
	\end{align}

	In summary, we have
	\begin{align}
	g(n,k,p)\leq  \frac{k\sqrt{q k}}{p} \quad \mbox{if} \quad 1< k\leq p n. 
	\end{align}
By continuity this upper bound holds  when $1\leq k\leq pn$, which implies the upper bound in  \eref{eq:LnkpUBkg1} and  completes the proof of \lref{lem:LnkpLUB2}.

\subsection{Proof of \lref{lem:Lnfpmono}}
The proof is divided into two steps: In the first step we prove the monotonicity and convexity/concavity properties of $L(n,f n,p)$ and in the second step we prove \eref{eq:LnfpLBUB}. \\

\textbf{Step 1:} Proofs of the monotonicity and convexity/concavity properties of $L(n,f n,p)$. Note that $L(n,f n,p)$ is continuous in $n,f,p$  in the parameter range specified in \lref{lem:Lnfpmono}, so we can focus on the interior of this parameter range, that is, $n>0$ and $0<f,p<1$. Then  $L(n,f n,p)> 1$ [cf. \eref{eq:Lnkpgeq1}]  and $L(n,f n,p)$ is the solution to the following equation that is larger than~1 [cf. \eqsref{eq:BnkbnkRatioProof}{eq:kL}],
	\begin{align}
	L+\frac{p(n+1)(L-1)}{L-p}=f n+1. \label{eq:fnL}
	\end{align} 
	This equation shows that $f$ is strictly increasing and concave in $L$, so $L$ is strictly increasing and convex in $f$.  Alternatively, this conclusion follows from \lref{lem:Lnkpmono} and its proof. Meanwhile, \lref{lem:Lnkpmono} implies that 	$L$ is nonincreasing and convex in $p$.

	According to the following equation,
	\begin{align}
	\frac{2\partial L(n,f n,p)}{\partial n}=f-p +\frac{(p-f)^2n+f+p-2fp }{\sqrt{(p n-f n+1)^2+4q f n}}>0,
	\end{align}
	$L(n,f n,p)$ is strictly increasing in $n$. Here the inequality holds because
	\begin{gather}
	(p-f)^2n+f+p-2fp>0,\\
	[(p-f)^2n+f+p-2fp]^2-(f-p)^2[(p n-f n+1)^2+4q f n]=4f(1-f)p q >0.
	\end{gather}

	In addition, by virtue of \eref{eq:fnL} we can deduce  that 
	\begin{align}
	n=\frac{L(L-1)}{p(1-f)-(p-f)L},\quad \frac{\partial^2 n }{\partial L^2}=\frac{2f(1-f)p q }{[p(1-f)-(p-f)L]^3}>0, \label{eq:nL}
	\end{align}
	which implies that $p(1-f)>(p-f)L$ and $n$ is strictly convex in $L$, so $L$  is strictly  concave in $n$ given that it is strictly increasing in $n$.\\
	
\textbf{Step 2:} Proof of \eref{eq:LnfpLBUB}, assuming that $n>0$ and $0\leq f<p<1$. If $f=0$, then both bounds in \eref{eq:LnfpLBUB} are equal to 1 and $L(n,f n,p)=1$, so \eref{eq:LnfpLBUB} holds and both inequalities are saturated.

If $n>0$ and $0< f<p<1$, then $L(n,f n,p)$
is strictly increasing in $n$ as proved above. 	In conjunction with \eref{eq:LUnfplim} and the equality $L(0,0,p)=1$
we can deduce that 
\begin{gather}\label{eq:LnfpLBUBproof}
1< L(n,f n,p)< \frac{(1-f)p}{p-f}=\frac{1}{1-r}.
\end{gather}
Alternatively, the upper bound follows from \eref{eq:nL}.

	Next, we turn to the lower bound in 	\eref{eq:LnfpLBUB}, assuming that $n>0$ and $0<f<p<1$. Let 
	\begin{align}
	h(n,f,p):=n\biggl[\frac{(1-f)p}{p-f}-L(n,f n,p)\biggr],\quad  
	h(\infty,f,p):=\lim_{n\to \infty}h(n,f,p)=\frac{fpq(1-f) }{(p-f)^3}.
	\end{align}
Calculation shows that
	\begin{align}
	\frac{\partial h(n,f,p)}{\partial n}=\frac{t_2 \sqrt{t_1}-t_3}{2\sqrt{t_1}},
	\end{align}
	where $t_1$, $t_2$, and $t_3$ are defined as follows:
	\begin{gather}
	t_1 :=(p n-f n+1)^2+4q f n, \quad t_2 :=\frac{f+p-2f p +2(p-f)^2n}{p-f},\\
	t_3 :=n[f+p-2f p +(p-f)^2n]+t_1=2(p-f)^2n^2+3(f+p-2f p)n+1. 
	\end{gather}
which satisfy $t_1, t_2, t_3>0$. In addition,
	\begin{align}
	t_2^2 t_1-t_3^2=\frac{4fpq (1-f) [1+2(f+p-2fp )n]}{(p-f)^2}> 0,
	\end{align}
	which implies that $\partial h(n,f,p)/\partial n>0$ and that $h(n,f,p)$ is strictly increasing in $n$. So 
	\begin{align}
	h(n,f,p)< h(\infty,f,p)=\frac{fpq(1-f) }{(p-f)^3},\\
	L(n,f n,p)> \frac{(1-f)p}{p-f}-\frac{fpq(1-f) }{(p-f)^3n},
	\end{align}	
	which implies \eref{eq:LnfpLBUB} given \eref{eq:LnfpLBUBproof}. Moreover,  both inequalities  in \eref{eq:LnfpLBUB} are strict when $n, f>0$.

\subsection{Proof of \lref{lem:U1nkpAlt}}
The proof is divided into two steps, in the first step we prove \eref{eq:U1nkpAlt} and in the second step we prove \eref{eq:U1nkpAlt3}.\\

\textbf{Step 1:} Proof of \eref{eq:U1nkpAlt}. From \eref{eq:Vnkpa} we can deduce that
\begin{align}
\frac{\partial V(n,k,p,x)}{\partial x}&=1-\frac{p q (n+1)}{(p n+p-k+x)^2},\quad 
\frac{\partial^2 V(n,k,p,x)}{\partial x^2}=\frac{2q p (n+1)}{(p n+p-k+x)^3},
\end{align}
which means $V(n,k,p,x)$ is strictly convex in $x$ for $x\geq k-p n-p$  and has a unique minimum point at $x=\tilde{a}=k-\kappa_1(n,p)$, where $\kappa_1(n,p)$ is defined in \eref{eq:kappa1}. In addition, $V(n,k,p,x)$ is strictly decreasing in $x$ when $k-p n-p<x\leq \tilde{a}$ and strictly increasing in $x$ when  $x\geq \tilde{a}$. In conjunction with the  following equation
\begin{align}\label{eq:Vk+1k}
V(n,k,p,k+1)-V(n,k,p,k)=\frac{p(n+2)}{p n+p +1}>0,
\end{align}
we can deduce that the minimum of $V(n,k,p,a)$ over $a\in \bbN_0$ is attained when $a< k+1$, which implies  \eref{eq:U1nkpAlt}.\\

\textbf{Step 2:} Proof of \eref{eq:U1nkpAlt3}.	
If in addition $k$ is  a nonnegative integer, say $k=0$, then \eref{eq:Vk+1k} implies that the minimum of $V(n,k,p,a)$ over $a\in \bbN_0$  is attained when $a\leq k$, which implies \eref{eq:U1nkpAlt3}. 

If $p(n+2)\geq 2$, then 
\begin{align}
V(n,k,p,k)-V(n,k,p,k-1)=\frac{p(n+2)-2}{p n+p -1}\geq 0,
\end{align}
so the minimum of $V(n,k,p,a)$ over $a\in \bbN_0$ is attained when $a\leq k$, which implies \eref{eq:U1nkpAlt3}. 
If $k\geq 2$, then $p(n+2)>p n\geq k\geq 2$, so \eref{eq:U1nkpAlt3} holds. 

Finally, we consider the case $1\leq k< 2$. If $p n\geq 2$, then \eref{eq:U1nkpAlt3} holds according to the above discussion, so we can assume that $1\leq k\leq p n< 2$ in the following discussion. Calculation shows that
\begin{align}
V(n,k,p,2)-V(n,k,p,1)=\frac{k^2-k[3+2p(n+1)]+p^2(n^2+3n+2)+2p(n+1)+2}{(p n+p-k+1)(p n+p-k+2)}.
\end{align}
Here the denominator is positive; the numerator is strictly decreasing in $k$ for $0\leq k< 2$ and is equal to 
\begin{align}
2+2p -p n+p^2 (n+2)>0 \quad \mbox{if}\quad k=p n.
\end{align}
Therefore, $V(n,k,p,2)-V(n,k,p,1)\geq 0$ when $1\leq k\leq p n\leq 2$, which implies \eref{eq:U1nkpAlt3} and completes the proof of \lref{lem:U1nkpAlt}. 

\subsection{Proof of \lref{lem:U1nkpMonoUB}}
From \eref{eq:Vnkpa} we can deduce that
	\begin{align}
	\frac{\partial V(n,k,p,a)}{\partial k}&=\frac{p q (n+1)}{(p n+p-k+a)^2},
	\end{align}
which implies that	$V(n,k,p,a)$ is strictly increasing in $k$ when $a\geq 0$ given that $ k\leq p n$ by assumption. So 	$U(n,k,p)$ is strictly increasing in $k$ according to the definition in \eref{eq:Unkp} and \lref{lem:U1nkpAlt}. Consequently,	
	\begin{align}
	U(n,k,p)\leq U(n,pn,p)\leq V(n,pn,p,\lceil\tilde{a}\rceil)\leq 1+V(n,pn,p,\tilde{a})= 1+2\sqrt{p q (n+1)},
	\end{align}	
	where $\tilde{a}=pn-\kappa_1(n,p)=\sqrt{pq(n+1)}-p>-1$.

	According to the following equation,
	\begin{align}
	\frac{\partial V(n,k,p,a)}{\partial n}&=\frac{p q (a-k)}{(p n+p-k+a)^2},\quad 
	\frac{\partial V(n,k,p,a)}{\partial p}=\frac{ (a-k)(n-k+a+1)}{(p n+p-k+a)^2},
	\end{align}
	$V(n,k,p,a)$ is nonincreasing in $n$ and $p$ when $0\leq a\leq k$. If in addition  $p(n+2)\geq 2$, $k=0$, or $k\geq 1$, then	$U(n,k,p)=\min_{a\in \bbN_0,\, a\leq k} V(n,k,p,a)$ according to \eref{eq:U1nkpAlt3} in \lref{lem:U1nkpAlt}, so $U(n,k,p)$ is nonincreasing in $n$ and $p$. Consequently,	
	\begin{align}
	U(n,k,p)\leq U(k/p,k,p)\leq V(k/p,k,p,\lceil\tilde{a}\rceil)\leq 1+V(k/p,k,p,\tilde{a})=1+2\sqrt{q (k+p)},
	\end{align}		
	where $\tilde{a}=k-\kappa_1(k/p,p)=\sqrt{q (k+p)}-p>-1$.

\subsection{Proof of \lref{lem:U1nfpMono}}
From \eref{eq:Vnkpa} we can deduce that
	\begin{align}
	\frac{\partial V(n,f n,p,a)}{\partial n}&=\frac{p q (f+a)}{(p n-f n+p+a)^2},\quad
	\frac{\partial V(n,f n,p,a)}{\partial f}=\frac{p q n (n+1)}{(p n-f n+p+a)^2},
	\end{align}
which implies that	$V(n,f n,p,a)$ is strictly increasing in $n$ when $a\geq 0$.  So 	$U(n,f n,p)$ is strictly increasing in $n$ according to the definition in \eref{eq:Unkp} and \eref{eq:U1nkpAlt} in \lref{lem:U1nkpAlt}. If in addition $n>0$, then $V(n,f n,p,a)$ is strictly increasing in $f$ when $a\geq 0$, so 	$U(n,f n,p)$ is strictly increasing in  $f$.   Alternatively, this conclusion follows from \lref{lem:U1nkpMonoUB}.
	
From \eref{eq:Vnkpa} we can deduce that
	\begin{align}
	\frac{\partial V(n,f n,p,a)}{\partial p}&=-\frac{ (f n-a)(n-f n+a+1)}{(p n-f n+p+a)^2},
	\end{align}
so 	$V(n,f n,p,a)$ is nonincreasing in $p$  when $f n\geq a$. If in addition $p(n+2)\geq 2$, $f n=0$, or $f n\geq 1$, then $U(n,fn,p)=\min_{a\in \bbN_0,\, a\leq fn} V(n,fn,p,a)$ by \eref{eq:U1nkpAlt3} in \lref{lem:U1nkpAlt}, 
so 	$U(n,f n,p)$ is nonincreasing in $p$. Alternatively, this conclusion follows from \lref{lem:U1nkpMonoUB}.
	
	Next, suppose $ 0<f<p<1$. The first  inequality in \eref{eq:U12fUB} follows from  the definition in \eref{eq:Unkp}; the second inequality in  \eref{eq:U12fUB} follows from \eref{eq:LUnfplim} and the fact that  $V(n,f n,p,0)$ is strictly increasing in $n$.

\section{\label{app:munkBBnkLBUB}Proof of \pref{pro:munkBBnkLBUB}}
By virtue of \eqsref{eq:BnkbnkMunk}{eq:BnkbnkRatioLB} we can deduce that
\begin{align}\label{eq:mun+1kLB}
\mu_{n+1,k}&\geq \frac{p(n+1)[L(n,k,p)-1]}{L(n,k,p)-p}=\frac{p n+k+1-\sqrt{(p n-k+1)^2+4q k}}{2}\\
&=k+1-L(n,k,p) \quad \forall k=0,1,\ldots, n.  \nonumber
\end{align}
This equation also holds when $k=n+1$ given that $\mu_{n+1,n+1}=p(n+1)$ and
\begin{align}
k+1-L(n,n+1,p)&= \frac{1}{2}\bigl[2+n+p n-\sqrt{q (q n^2+4n+4)}\,\bigr]\\
&\leq  \frac{1}{2}\bigl[2+n+p n-(q n+2q)\bigr]=p(n+1). \nonumber
\end{align}
Therefore, 
\begin{align}\label{eq:muLBproof}
\mu_{n,k}&\geq k+1-L(n-1,k,p)\quad \forall k=0,1,\ldots, n,
\end{align}
which confirms \eref{eq:mnkLBGen}.

The lower bound in \eref{eq:BnkBn+1kLBUB1} follows from \lref{lem:Bbn+mkp}, and the upper bound follows from \eqsref{eq:BnkBn+1kMunk}{eq:mun+1kLB}.

Next, we assume that $k\leq pn$.  Then 
\begin{align}
\mu_{n,k}&\geq k+1-L(n-1,k,p)\geq k+1-L((k/p)-1,k,p)\\
&=k+\frac{q}{2}-\frac{1}{2}\sqrt{q (4k+q )}\geq k-\sqrt{q k},\nonumber
\end{align}
which confirms \eref{eq:mnkLBsp}. Here the first inequality follows from \eref{eq:muLBproof}, and the second inequality follows from the assumption that  $k\leq pn$ and the fact that $L(n-1,k,p)$ is nonincreasing in $n$ for $n\geq 0$ by \lref{lem:Lnkpmono} 
presented in \sref{S4-2}.

\Eref{eq:BnkBn+1kLBUB2} follows from \eref{eq:LnkpUBkg1} in \lref{lem:LnkpLUB2} and \eref{eq:BnkBn+1kLBUB1}.
\Eref{eq:BnkBn+1kLBUB3} follows from
\eref{eq:LnfpLBUB} in \lref{lem:Lnfpmono} and \eref{eq:BnkBn+1kLBUB1}.

\section{\label{app:SCM}Proofs of \lsref{lem:SCMomega} and \ref{lem:SCMrho}}

\subsection{Proof of \lref{lem:SCMomega}}
Before proving \lref{lem:SCMomega}, we need to prepare an auxiliary result. Define
\begin{align}
\nu(t):=t-\biggl(1+\frac{t}{2}+\frac{t^2}{12}\rme^{-t/12}\biggr)(1-\rme^{-t}). \label{eq:nut}
\end{align}
\begin{lemma}\label{lem:nut}
	Suppose $t>0$, then $\nu(t)>0$. 
\end{lemma}
\begin{proof}
	Straightforward calculation shows that
	\begin{align}
	\nu(t)\rme^t=t\rme^t-\biggl(1+\frac{t}{2}+\frac{t^2}{12}\rme^{-t/12}\biggr)(\rme^{t}-1)
	=\sum_{j=4}^\infty \frac{a_j t^j } {j!},
	\end{align}
	where the second equality is derived by considering the tailor expansion of the exponential function and $a_j$ reads
	\begin{align}
	a_j=\frac{j}{2}-1-\frac{j(j-1)}{12}\Bigl(\frac{11}{12}\Bigr)^{j-2}+\frac{j(j-1)}{12}\Bigl(-\frac{1}{12}\Bigr)^{j-2}. 
	\end{align}
	When $4\leq j\leq 12$, it is straightforward  to verify that $a_j>0$. When $j\geq 13$, the function $j(11/12)^{j-2}+j(1/12)^{j-2}$ decreases monotonically with $j$ and is bounded from above by 5, which implies that
	\begin{align}
	a_j&\geq \frac{j}{2}-1-\frac{j(j-1)}{12}\Bigl(\frac{11}{12}\Bigr)^{j-2}-\frac{j(j-1)}{12}\Bigl(\frac{1}{12}\Bigr)^{j-2}\geq \frac{j}{2}-1-\frac{5(j-1)}{12}\\
	&=\frac{j-7}{12}>0. \nonumber
	\end{align}
	In a word, $a_j>0$ for $j\geq 4$, which means $\nu(t)\rme^t>0$ and $\nu(t)>0$ for $t>0$. 
\end{proof}

\begin{proof}[Proof of \lref{lem:SCMomega}]
	According to Theorem 1  in \rcite{Mort10} and its proof,   $-\omega_-(x)$ is strictly completely monotonic. So it remains to prove that $\omega_+(x)$ is strictly completely monotonic. Here our proof follows the proof of Theorem 2 in  \rcite{Mort10} with a mistake corrected. Calculation shows that
	\begin{align}
	\omega_+'(x)&=\psi(x)+\frac{1}{2x}-\ln x+\frac{1}{12}\frac{1}{\bigl(x+\frac{1}{12}\bigr)^2},\\
	\omega_+''(x)&=\psi'(x)-\frac{1}{2x^2}-\frac{1}{x}-\frac{1}{6}\frac{1}{\bigl(x+\frac{1}{12}\bigr)^3}=
	\int_{0}^\infty\frac{\rme^{-xt}}{1-\rme^{-t}}\nu(t)dt,\\
	\omega_+^{(m)}(x)&=(-1)^m\int_{0}^\infty\frac{\rme^{-xt}}{1-\rme^{-t}}t^{m-2}\nu(t)dt \quad \forall m=2, 3, 4,\ldots,
	\end{align}
	where $\psi(x)$ is the digamma function, that is, the logarithmic derivative of the gamma function, and $\nu(t)$ is defined in \eref{eq:nut}. 
	According to \lref{lem:nut}, we have $\nu(t)>0$ for $t>0$ (the proof of this fact in  \rcite{Mort10} is problematic), which implies that $\omega_+''(x)$ is strictly completely monotonic. In particular, we have $\omega_+''(x)>0$ for $x>0$. 
	
	In addition,
	\begin{align}
	\lim_{x\to\infty} \omega_+'(x)=\lim_{x\to\infty} \omega_+(x)=0. 
	\end{align}
	Therefore, $\omega_+'(x)<0$ for $x>0$ given that $\omega_+''(x)>0$ for $x>0$. This result in turn implies that $\omega_+(x)>0$ for $x>0$, so  $\omega_+(x)$ is strictly completely monotonic. 
\end{proof}

\subsection{Proof of \lref{lem:SCMrho}}
	Direct calculation shows that
	\begin{align}
	[\ln \rho(y)]'=\psi(y)+\frac{1}{y}-\ln y,\quad
	[\ln \varrho(y)]'=\ln y-\psi(y)-\frac{1}{2y},
	\end{align}
	where $\psi(y)$ is the digamma function. 
	Both $\psi(y)+(1/y)-\ln y$ and $\ln y-\psi(y)-[1/(2y)]$ are strictly completely monotonic  according to Theorem 1.3 in \rcite{Qi07} (in the theorem the word "strictly" is not mentioned explicitly, but its proof  actually shows this stronger result), so  $[\ln \rho(y)]'$ and $[\ln \varrho(y)]'$ are strictly completely monotonic, which imply that $1/\rho(y)$ and $1/\varrho(y)$  are strictly logarithmically completely monotonic.

	Next, according to \lref{lem:SCMomega} and the definitions in \eref{eq:varrhopm}, $1/ \varrho_-(y)]$, and $\varrho_+(y)$ are strictly logarithmically completely monotonic. 
	
	Recall that any function that is strictly logarithmically completely monotonic  is strictly completely monotonic \cite{QiC04}. So $1/\rho(y)$,  $1/\varrho(y)$, $1/\varrho_-(y)$, and $ \varrho_+(y)$	are  strictly completely monotonic given that they are  strictly logarithmically completely monotonic as shown above. 
	
	Finally, \eref{eq:varrhokn} follows from the following equation
	\begin{align}
	\varrho_+(y)<\varrho_+(x),\quad \varrho_-(y)>\varrho_-(x),
	\end{align}
	given that $\varrho_+(y)$ is strictly decreasing, while $\varrho_-(y)$ is strictly increasing.

\section{\label{app:comparison}Comparison with bounds of McKay \cite{McKa89}}\label{APE}

\subsection{Comparison of asymptotic bounds for the ratio $\bar{B}_{n,k}(p)/b_{n,k}(p)$}
Here we compare bounds for the ratio $\bar{B}_{n,k}(p)/b_{n,k}(p)$ presented in  \thref{thm:BnkbnkLBUBbar} with the counterparts derived by McKay \cite{McKa89}, assuming that $n,k\in \bbN$, $0<p<1$, and $pn<k\leq n$. To simplify the discussion we will focus on the ratio of the upper bound over the lower bound in the large-$n$ limit.
 
Theorem 2 in  \rcite{McKa89} states that 
\begin{align}
&\sigma Y(x)
\le \frac{\bar{B}_{n,k}(p)}{b_{n-1,k-1}(p)} \le \sigma Y(x)
\rme^{E_{n,k}(p)},\label{eq:McKay}
\end{align}	
where 
\begin{gather}
\sigma=\sqrt{npq},\quad q=1-p, \quad x=\frac{|k-pn|}{\sigma},\\ Y(x)=\rme^{x^2/2}\int_x^\infty \rme^{-t^2/2}dt, \quad E_{n,k}(p)=\min\biggl\{\sqrt{\frac{\pi}{8npq}},
\frac{1}{|k-pn|}\biggr\}. 
\end{gather}
Thanks to  the equality $b_{n-1,k-1}(p)=(k/pn)b_{n,k}(p)$, which follows from the definition in \eref{eq:Bnkp}, \eref{eq:McKay} implies that
\begin{align}
&k\sqrt{\frac{q}{pn}}\,  Y(x)
\le \frac{\bar{B}_{n,k}(p)}{b_{n,k}(p)} \le k\sqrt{\frac{q}{pn}}\, Y(x) \rme^{E_{n,k}(p)}.\label{eq:McKay2}
\end{align}

\begin{figure}[t]
	\includegraphics[width=7cm]{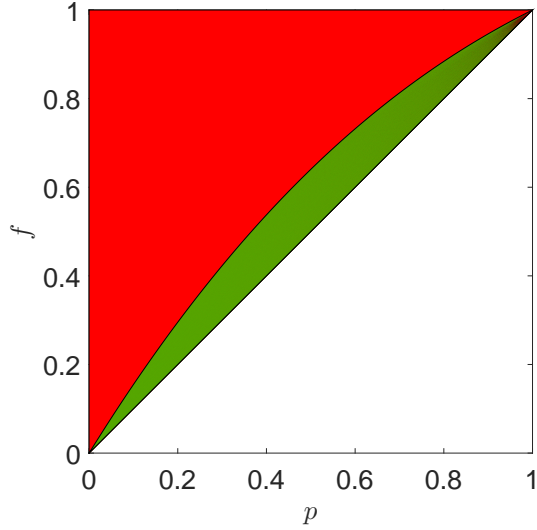}
	\caption{\label{fig:comp} 
		Comparison between our bounds for  the ratio $\bar{B}_{n,fn}(p)/b_{n,fn}(p)$	presented in \thref{thm:BnkbnkLBUBbar} and the counterparts derived by McKay \cite{McKa89} in the large-$n$ limit, where $\bar{B}_{n,fn}(p)$ is the upper tail probability. In the green region (with $\gamma_2>\gamma_1$) the bounds in \rcite{McKa89} are more accurate; in the red region (with $\gamma_2<\gamma_1$), our bounds are more accurate. The boundary is determined by \eref{eq:f*}.
	}
\end{figure}

Suppose $k=fn$ with $p<f\leq 1$ and $n$ is sufficiently large; then the ratio of the upper bound over the lower bound  in \eref{eq:McKay2} reads 
\begin{align}
\exp\biggl(\frac{1}{(f-p)n}\biggr)=1+ \frac{\gamma_1}{n}+O(n^{-2}),\quad \gamma_1 =\frac{1}{f-p}. 
\end{align}	
By contrast, our \thref{thm:BnkbnkLBUBbar} yields the following bounds,
\begin{gather}
\bar{L}(n,fn,p)\leq 	\frac{\bar{B}_{n,fn}(p)}{b_{n,fn}(p)}\leq \bar{U}(n,fn,p).
\end{gather}
The ratio of the upper bound over the lower bound reads
\begin{align}
\frac{\bar{U}(n,fn,p)}{\bar{L}(n,fn,p)}=  \frac{U(n,(1-f)n,q)}{L(n,(1-f)n,q)}=1+\frac{\gamma_2}{n}+O(n^{-2}), \quad \gamma_2= \frac{(1-f)p^2}{f(f-p)^2}. 
\end{align}	
Note that $\gamma_2/\gamma_1=(1-f)p^2/[f(f-p)]$ decreases monotonically with $f$ for $p<f\leq 1$. In addition,  $\gamma_2=\gamma_1$ iff $f=f^*$ with 
\begin{align}
f^*=\frac{p}{2}\bigl(q+\sqrt{4+q^2}\,\bigr).  \label{eq:f*}
\end{align}
If  $p<f<f^*$  ($f$ is close to $p$), then  $\gamma_2>\gamma_1$, so the bounds in \rcite{McKa89} are more accurate. If instead $f^*<f\leq 1$  ($f$ is not so  close to $p$), then $\gamma_2<\gamma_1$, so our bounds are more accurate.
The two parameter ranges are illustrated in \fref{fig:comp}.  In addition, our bounds do not involve integrals and are more explicit than the bounds  in \rcite{McKa89}.

\subsection{Derived bounds based on \rcite{McKa89}}
Here we  derive a number of related bounds for the ratio $\bar{B}_{n,k}(p)/b_{n,k}(p)$ and the upper tail probability  $\bar{B}_{n,k}(p)$
that are of independent interest.

First, we present a simple upper bound for  the function $E_{n,k}(p)$,  assuming that $n,k\in \bbN$, $0<p<1$, and $pn<k\leq n-1$ (the special case with $k=n$ is not essential).
\begin{proposition}\label{pro:Enkp}
	If $n,k\in \bbN$, $0<p<1$, and $pn< k\leq n-1$, then $E_{n,k}(p)\leq 3/2$. 
\end{proposition}
\begin{proof}
	By assumption we have $p\leq 1-(1/n)$.  If in addition $p\geq 1/(3n)$, then 
	\begin{align}
	E_{n,k}(p)\leq \sqrt{\frac{\pi}{8npq}}\leq \sqrt{\frac{\pi}{\frac{8}{3}\bigl(1-\frac{1}{3n}\bigr)}}\leq \frac{3}{4}\sqrt{\pi}<\frac{3}{2}. 
	\end{align}
	If $p<1/(3n)$ instead, then 
	\begin{align}
	E_{n,k}(p)\leq \frac{1}{k-pn}< \frac{1}{1-\frac{1}{3}}=\frac{3}{2}.
	\end{align}
\end{proof}

Next, we provide upper and lower bounds for the function $Y(x)$ by virtue of   \eqssref{eq:ellupsilon}{eq:BnkpLimCLT}{eq:PhixLUB} in \sref{sec:conjecture},  
\begin{align}\label{eq:YxLUB}
\tilde{\ell}(x)\leq Y(x)\leq \tilde{\upsilon}(x) \quad \forall x\geq 0,
\end{align}
where 
\begin{align}
\tilde{\ell}(x):=\sqrt{2\pi}\,\ell(x)=\frac{1}{2} \bigl(\sqrt{4 + x^2}-x\bigr), \quad \tilde{\upsilon}(x):=\sqrt{2\pi}\,\upsilon(x)=\begin{cases}
2-x & x\leq 1,\\
\frac{1}{x} &x\geq 1.
\end{cases}
\end{align}
Combining \eqsref{eq:McKay2}{eq:YxLUB} we can obtain bounds for $\bar{B}_{n,k}(p)/b_{n,k}(p)$ that do not involve integrals and are easy to compute, 
\begin{align}
&k\sqrt{\frac{q}{pn}}\, \tilde{\ell}(x)
\le \frac{\bar{B}_{n,k}(p)}{b_{n,k}(p)} \le k\sqrt{\frac{q}{pn}}\, \tilde{\upsilon}(x) \rme^{E_{n,k}(p)}\leq 2\rme^{3/2}k\sqrt{\frac{q}{pn}}\, \tilde{\ell}(x), 
\end{align}	
where the last inequality follows from \pref{pro:Enkp} above and \eref{eq:ellupsilonLUB} in \pref{pro:ellPhi}. The two propositions also show that the lower bound and the first upper bound for $\bar{B}_{n,k}(p)/b_{n,k}(p)$ in this equation are asymptotically tight and  universally bounded.

Combining \eqsref{eq:McKay2}{eq:YxLUB} with \eref{eq:bnkpLUB3} in \pref{pro:bnkpLUB}, we can further deduce upper and lower bounds for the upper tail probability $\bar{B}_{n,k}(p)$ as follows,
\begin{gather}\label{eq:UTailLUBa}
\sqrt{\frac{qk}{pm}}\, \varphi_-(n,k)\tilde{\ell} (x)
\rme^{-nD(\frac{k}{n}\| p)}\leq \sqrt{\frac{qk}{pm}}\, \varphi_-(n,k)Y(x)
\rme^{-nD(\frac{k}{n}\| p)}
< \bar{B}_{n,k}(p) \\
< \sqrt{\frac{qk}{pm}}\, \varphi_+(n,k)Y(x) \rme^{E_{n,k}(p)}\rme^{-nD(\frac{k}{n}\| p)} \leq  \sqrt{\frac{qk}{pm}}\, \varphi_+(n,k)\tilde{\upsilon} (x) \rme^{E_{n,k}(p)}\rme^{-nD(\frac{k}{n}\| p)},\nonumber
\end{gather}
where $m=n-k$ and $\varphi_\pm(n,k)$ are defined in \eref{eq:varphipm(nk)}. It is not difficult to very that  the final upper bound and the final lower bound in \eref{eq:UTailLUBa} satisfy criteria (C1-C3) presented in the introduction. To be specific, the ratio of the upper bound over the  lower bound reads 
\begin{align}
\frac{\varphi_+(n,k)\tilde{\upsilon} (x) \rme^{E_{n,k}(p)}}{\varphi_-(n,k)\tilde{\ell} (x)}\leq 2\rme^{\frac{3}{2}+\frac{29}{2600}}=2\rme^{\frac{3929}{2600}}\approx 9.06391,
\end{align}
where the inequality follows from \eref{eq:varphi+/varphi-} in  \lref{lem:varphipm}, \eref{eq:ellupsilonLUB} in  \pref{pro:ellPhi}, and \pref{pro:Enkp}. This bound is much larger than the upper bound $89/44$ that appears in \thref{thm:UTailBound} (cf. \thref{thm:TailBound}).

Thanks to the relation between the upper and lower tail probabilities presented in \eref{eq:UTailBinom}, all the above results have analogs for the lower tail probability $B_{n,k}(p)$. Notably, \eref{eq:UTailLUBa} has the following analog, assuming that $n,k\in \bbN$, $0<p<1$, and $1\leq k<pn$,
\begin{gather}
\sqrt{\frac{pm}{qk}}\, \varphi_-(n,k)\tilde{\ell} (x)
\rme^{-nD(\frac{k}{n}\| p)}\leq \sqrt{\frac{pm}{qk}}\, \varphi_-(n,k)Y(x)
\rme^{-nD(\frac{k}{n}\| p)}
< B_{n,k}(p) \\
< \sqrt{\frac{pm}{qk}}\, \varphi_+(n,k)Y(x) \rme^{E_{n,k}(p)}\rme^{-nD(\frac{k}{n}\| p)} \leq  \sqrt{\frac{pm}{qk}}\, \varphi_+(n,k)\tilde{\upsilon} (x) \rme^{E_{n,k}(p)}\rme^{-nD(\frac{k}{n}\| p)},\nonumber
\end{gather}
where all the functions involved are defined as before.

\section{\label{app:Bnjbnjk/n}Proofs of \lsref{lem:Bnjbnjk/n} and \ref{lem:Bnk1bnk1nk}}

\subsection{Proof of \lref{lem:Bnjbnjk/n}}
	When $j=0$, we have $B_{n,j}(k/n)/b_{n,j}(k/n)=1$	by definition and \eref{eq:Bnj/bnjlim} holds. 
	
	When $j\geq 1$, from \eref{eq:bnk-1/bnk} we can deduce that
	\begin{align}
	\frac{b_{n,j-1}(k/n)}{b_{n,j}(k/n)}=\frac{j(n-k)}{k(n-j+1)},
	\end{align}
	which implies that $b_{n,j-1}(k/n)/b_{n,j}(k/n)$ and $B_{n,j}(k/n)/b_{n,j}(k/n)$ for $j=1,2,\ldots, k$ are strictly increasing in $n$. In addition, the above equation implies that
	\begin{align}
	\lim_{n\to\infty}\frac{b_{n,j-1}(k/n)}{b_{n,j}(k/n)}=\frac{j}{k},\quad 
	\lim_{n\to\infty}\frac{B_{n,j}(k/n)}{b_{n,j}(k/n)}=\sum_{l=0}^{j}\frac{\Gamma(j+1)}{k^l\Gamma(j-l+1)},
	\end{align}
	which confirms \eref{eq:Bnj/bnjlim}. 
	
	The equalities in \eqsref{eq:Bnkbnkm1lim}{eq:Bnkbnknklim} follow from \eref{eq:Bnj/bnjlim} and the definition of $\theta_k$ in \eref{eq:Ramanujan}. When $k=1$, the inequality in \eref{eq:Bnkbnkm1lim} can be verified directly. When $k\geq 2$, the inequality can be proved by virtue of the Stirling approximation in \eref{eq:Stirling}	
	and the lower bound for $\theta_k$ in \eref{eq:thetakLBUB} as follows,
	\begin{align}
	\frac{\rme^k k!}{2k^k}-\theta_k- \sqrt{\frac{\pi k}{2}}< & \sqrt{\frac{\pi k }{2}}\biggl[\exp\Bigl(\frac{1}{12k}\Bigr)-1\biggr]-\theta_k 
	< \sqrt{\frac{\pi }{2}} \frac{2\sqrt{k}}{23k}-\frac{1}{3}<1-\sqrt{\frac{\pi }{2}}.
	\end{align}
	The inequality in \eref{eq:Bnkbnknklim} follows from the counterpart in  \eref{eq:Bnkbnkm1lim}.

	Next, by virtue of \eqssref{eq:Lnkp}{eq:philim}{eq:bnkpLUB3} we can deduce that 
	\begin{align}
	\lim_{n\to\infty}\frac{L(n,\lfloor fn-j\rfloor,f)}{\sqrt{n}}&=	\lim_{n\to\infty}\frac{L(n,fn-j,f)}{\sqrt{n}}= \sqrt{ f(1-f)},\\
\lim_{n\to\infty} \sqrt{n}\, b_{n,\lfloor fn-j\rfloor}(f)&= \frac{1}{\sqrt{2\pi f(1-f)}},
	\end{align}
	which imply  \eref{eq:Bnfbnfnflim}, given that $\lim_{n\to\infty}B_{n,\lfloor fn-j\rfloor}(f)= 1/2$.

\subsection{Proof of \lref{lem:Bnk1bnk1nk}}
The proof is divided into two steps: In the first step  we prove   \eref{eq:Bnk1bnk1nk} in the four special cases $k=1,2,n-1,n-2$ and in the second step we prove   \eref{eq:Bnk1bnk1nk} in the case $3 \le k \le n-3$.\\

\textbf{Step 1:} Proof of \eref{eq:Bnk1bnk1nk} in the four special cases $k=1,2,n-1,n-2$. 
If $k=1$, then 	$n\geq 2$ and
	\begin{align}
	B_{n,k-1}(k/n)=b_{n,k-1}(k/n)=\Bigl(\frac{n-1}{n}\Bigr)^n,\quad L(n,k-1,k/n)=1,
	\end{align}
	so \eref{eq:Bnk1bnk1nk} holds in the case $k=1$. 
	
	If $k=2$, then $n\geq 3$ and
	\begin{gather}
	B_{n,k-1}(k/n)=\frac{3n-2}{n-2}\Bigl(\frac{n-2}{n}\Bigr)^n,\quad b_{n,k-1}(k/n)=2\Bigl(\frac{n-2}{n}\Bigr)^{n-1},\\
	 L(n,k-1,k/n)=\sqrt{\frac{2n-2}{n}}.
	\end{gather}
	Therefore,
	\begin{align}
	\frac{B_{n,k-1}(k/n)}{b_{n,k-1}(k/n)L(n,k-1,k/n)}=\frac{3n-2}{\sqrt{8n(n-1)}}
\leq \sqrt{\frac{9}{8}}<\sqrt{\frac{\pi}{2}}, 
	\end{align}
	which confirms \eref{eq:Bnk1bnk1nk} in the case $k=2$.
	
	If $k=n-1$, then 
	\begin{gather}
	B_{n,k-1}(k/n)=1-\Bigl(\frac{k}{k+1}\Bigr)^{k+1}-\Bigl(\frac{k}{k+1}\Bigr)^{k},\\ b_{n,k-1}(k/n)=\frac{1}{2}\Bigl(\frac{k}{k+1}\Bigr)^k,\quad
	L(n,k-1,k/n)=\sqrt{\frac{2k}{k+1}}.
	\end{gather}
	Therefore,
	\begin{align}
	\frac{B_{n,k-1}(k/n)}{b_{n,k-1}(k/n)L(n,k-1,k/n)}&=\sqrt{2}\Biggl[\Bigl(\frac{k+1}{k}\Bigr)^{k+\frac{1}{2}}-\sqrt{\frac{k}{k+1}}-\sqrt{\frac{k+1}{k}}\,\Biggr]\\
	&\leq 4-2\sqrt{2}< \sqrt{\frac{\pi}{2}}, \nonumber
	\end{align}	
	which confirms \eref{eq:Bnk1bnk1nk} in the case $k=n-1$.
	
	If $k=n-2$, then 
	\begin{gather}
	B_{n,k-1}(k/n)=1-\Bigl(\frac{k}{k+2}\Bigr)^{k+2}-2\Bigl(\frac{k}{k+2}\Bigr)^{k+1}-2\frac{k+1}{k+2}\Bigl(\frac{k}{k+2}\Bigr)^k,\\ b_{n,k-1}(k/n)=\frac{4}{3}\frac{k+1}{k+2}\Bigl(\frac{k}{k+2}\Bigr)^k,\quad 
	L(n,k-1,k/n)=\sqrt{\frac{3k}{k+2}}.
	\end{gather}
	Therefore,
	\begin{align}
	&\frac{B_{n,k-1}(k/n)}{b_{n,k-1}(k/n)L(n,k-1,k/n)}=\frac{\sqrt{3}}{4}\Biggl[\frac{k+2}{k+1}\Bigl(\frac{k+2}{k}\Bigr)^{k+\frac{1}{2}}-\frac{3k+4}{k+1}\sqrt{\frac{k}{k+2}}-2\sqrt{\frac{k+2}{k}}\,\Biggr]\\
	&\leq \frac{\sqrt{3}}{4}\biggl(\frac{9\sqrt{3}}{2}-5\biggr)
	=\frac{27}{8}-\frac{5\sqrt{3}}{4}< \sqrt{\frac{\pi}{2}}, \nonumber
	\end{align}	
	which confirms \eref{eq:Bnk1bnk1nk} in the case $k=n-2$. Here
	the first inequality follows from the following inequalities:
	\begin{align}
	\frac{k+2}{k+1}\Bigl(\frac{k+2}{k}\Bigr)^{k+\frac{1}{2}}\leq \frac{9\sqrt{3}}{2},\quad \frac{3k+4}{k+1}\sqrt{\frac{k}{k+2}}+2\sqrt{\frac{k+2}{k}}\geq 5.
	\end{align}	
	Above analysis shows that \eref{eq:Bnk1bnk1nk} holds when 	$k=1,2 $, $n-1$, or $n-2$, so we can exclude these cases in the following discussion. \\

\textbf{Step 2:} Proof of \eref{eq:Bnk1bnk1nk} in the case $3 \le k \le n-3$. Define $\zeta_{n,k}$ by the following equation
	\begin{align}\label{eq:zetank}
	\frac{1}{2}=B_{n,k-1}(k/n)+\zeta_{n,k} b_{n,k}(k/n);
	\end{align}
	then 
	\begin{align}
	B_{n,k-1}(k/n)=\frac{1}{2}-\zeta_{n,k} b_{n,k}(k/n). \label{eq:Bnkzeta}
	\end{align}
	It is known that \cite{JogdS68}
	\begin{align}\label{eq:zetaLBUB}
	\begin{aligned}
	\frac{1}{3}<\zeta_{n,k}\leq \frac{1}{2} \quad \mbox{if}\quad  n\geq 2k, \\
	\frac{1}{2}\leq \zeta_{n,k}< \frac{2}{3} \quad \mbox{if}\quad  n\leq 2k.
	\end{aligned}
	\end{align}
By definition in \eref{eq:Bnkp} and the Stirling approximation in 
	\eref{eq:Stirling} (cf. \pref{pro:bnkpLUB}) we can deduce that
	\begin{align}
	b_{n,k}(k/n)&
	\geq \sqrt{\frac{n}{2\pi k(n-k)}}\exp\biggl[\frac{1}{12n+1}-\frac{1}{12 k}-\frac{1}{12(n-k)}\biggr]. \label{eq:bnknkLB0}
	\end{align}	
In addition, \eqsref{eq:Lnkp}{eq:bnk-1/bnk} yield 
\begin{align}
L(n,k-1,k/n)=\sqrt{\frac{k(n-k+1)}{n}}, \quad 	b_{n,k-1}(k/n)&=\frac{n-k}{n-k+1}b_{n,k}(k/n).\label{eq:bnkLnk}
	\end{align}
	The above two equations together imply that 
	\begin{align}\label{eq:BbLratioProof}
	&\frac{B_{n,k-1}(k/n)}{\sqrt{\frac{\pi}{2}} b_{n,k-1}(k/n) L(n,k-1,k/n)}=\sqrt{\frac{2}{\pi}}\sqrt{\frac{n(n-k+1)}{k(n-k)^2}}\biggl(\frac{1}{2b_{n,k}(k/n)}-\zeta_{n,k}\biggr)
	\\
	&\leq \sqrt{\frac{n-k+1}{n-k}}\biggl\{\exp\biggl[\frac{1}{12 k}+\frac{1}{12(n-k)}-\frac{1}{12n+1}\biggr]-\sqrt{\frac{2}{\pi}}
	\sqrt{\frac{n}{k(n-k)}}
	\zeta_{n,k}\biggr\}\nonumber \\
	&\leq  \biggl(1+\frac{1}{2(n-k)}\biggr)\biggl(1+\frac{n}{11k(n-k)}\biggr) -\sqrt{\frac{2}{\pi}} \sqrt{\frac{n}{k(n-k)}}\,
	\zeta_{n,k} \nonumber \\
&	\leq  1+\sqrt{\frac{n}{k(n-k)}}\biggl(\frac{1}{2\sqrt{n-k}}+
	\frac{7\sqrt{6}}{198}-\sqrt{\frac{2}{\pi}}\,
	\zeta_{n,k}
	\biggr).\nonumber
	\end{align}
Here the first equality follows from \eqsref{eq:Bnkzeta}{eq:bnkLnk}. The first inequality follows from \eref{eq:bnknkLB0}. The second inequality  follows from the following equation
\begin{gather}
1<\sqrt{\frac{n-k+1}{n-k}}<1+\frac{1}{2(n-k)},\quad \exp\biggl[\frac{1}{12 k}+\frac{1}{12(n-k)}-\frac{1}{12n+1}\biggr]< 1+\frac{n}{11k(n-k)}. 
\end{gather}
The third inequality in \eref{eq:BbLratioProof} follows from the following equation
\begin{equation}
\begin{gathered}
\biggl(1+\frac{1}{2(n-k)}\biggr)\frac{n}{11k(n-k)}\leq \frac{7}{66}\sqrt{\frac{n}{k(n-k)}}\sqrt{\frac{n}{k(n-k)}}\leq 	\frac{7\sqrt{6}}{198} \sqrt{\frac{n}{k(n-k)}},\\
\frac{1}{n-k}\leq \sqrt{\frac{n}{k(n-k)}}\frac{1}{\sqrt{n-k}},
\end{gathered}
\end{equation}
given  that $3 \le k \le n-3$, so that $n/[k(n-k)]\leq 2/3$.

If $ n/2 \le k \le n-3$, that is, $3\leq n-k\leq k$, then  $\zeta_{n,k}\geq 1/2$ by \eref{eq:zetaLBUB}, which implies that
	\begin{align}
	\frac{1}{2\sqrt{n-k}}+
	\frac{7\sqrt{6}}{198}-\sqrt{\frac{2}{\pi}}\,
	\zeta_{n,k}\leq \frac{1}{2\sqrt{3}}+
	\frac{7\sqrt{6}}{198}-\frac{1}{2}\sqrt{\frac{2}{\pi}}<0. 
	\end{align} 
	If $k \le n-8$, that is, $n-k\geq 8$, then  $\zeta_{n,k}> 1/3$ by \eref{eq:zetaLBUB}, which implies that
	\begin{align}
	\frac{1}{2\sqrt{n-k}}+
	\frac{7\sqrt{6}}{198}-\sqrt{\frac{2}{\pi}}\,
	\zeta_{n,k}< \frac{1}{2\sqrt{8}}+
	\frac{7\sqrt{6}}{198}-\frac{1}{3}\sqrt{\frac{2}{\pi}}<0. 
	\end{align} 
	In both cases \eref{eq:Bnk1bnk1nk} holds. In the remaining case with 
	$n-7\le k < n/2$,
	\eref{eq:Bnk1bnk1nk} can be verified by direct calculation
	because such a case can happen only when $n \le 13$.
	 This observation completes the proof of \lref{lem:Bnk1bnk1nk}.

\section{\label{app:Bnjbnjk/n2}Proof of \lref{lem:Bnkbnknk}}
\subsection{Auxiliary lemmas}
Here we prove two auxiliary lemmas that are required to prove \lref{lem:Bnkbnknk}, without  assuming that $k$ and $n$ are integers. 

\begin{lemma}\label{lem:Exp}
	Suppose $k\geq 2/7$ and $n\geq 5k/3$. Then 
	\begin{align}
	\exp\biggl[\frac{1}{12 k}+\frac{1}{12(n-k)}-\frac{1}{12n+1}\biggr]
	&\leq1+ \frac{n}{n-k} \biggl[\exp\Bigl(\frac{1}{12 k}\Bigr)-1\biggr]. \label{eq:Exp}
	\end{align}
\end{lemma}
\begin{proof}[Proof of \lref{lem:Exp}]
	The inequality in \eref{lem:Exp} is equivalent to the following inequality,
	\begin{align}
	k\exp\Bigl(-\frac{1}{12 k}\Bigr)+(n-k)\exp\biggl[\frac{1}{12(n-k)}-\frac{1}{12n+1}\biggr]
	&\leq n. \label{eq:Exp2}
	\end{align}

	By assumption we can deduce that 
	\begin{align}
	\frac{1}{12(n-k)}-\frac{1}{12n+1}&\leq \frac{1}{12k},\\ 
	\exp\biggl[\frac{1}{12(n-k)}-\frac{1}{12n+1}\biggr]&\leq 1+ 12k\biggl[\exp\Bigl(\frac{1}{12k}\Bigr)-1\biggr]\biggl[\frac{1}{12(n-k)}-\frac{1}{12n+1}\biggr].\label{eq:ExpProof1}
	\end{align}
	Therefore, 
	\begin{align}\label{eq:ExpProof2}
	&k\exp\Bigl(-\frac{1}{12 k}\Bigr)+(n-k)\exp\biggl[\frac{1}{12(n-k)}-\frac{1}{12n+1}\biggr]\\
	&\leq k\exp\Bigl(-\frac{1}{12 k}\Bigr)+n-k+12k(n-k)\biggl[\exp\Bigl(\frac{1}{12k}\Bigr)-1\biggr]\biggl[\frac{1}{12(n-k)}-\frac{1}{12n+1}\biggr]\nonumber\\
	&=n+k\biggl[\exp\Bigl(\frac{1}{12 k}\Bigr)+\exp\Bigl(-\frac{1}{12 k}\Bigr)-2\biggr]-\frac{12k(n-k)}{12n+1}\biggl[\exp\Bigl(\frac{1}{12 k}\Bigr)-1\biggr]\nonumber\\
	&\leq n+ \frac{1}{40}-\frac{n-k}{12n+1}\leq n+ \frac{1}{40}-\frac{2k}{3(20k+1)} \leq n+ \frac{1}{40}-\frac{4}{141}<n,    \nonumber
	\end{align}
	which confirms \eref{eq:Exp2} and implies \eref{eq:Exp}. Here the first inequality follows from \eref{eq:ExpProof1}; the
second inequality follows from the following two inequalities 
	\begin{align}
	\exp\Bigl(\frac{1}{12 k}\Bigr)+\exp\Bigl(-\frac{1}{12 k}\Bigr)-2\leq \frac{1}{40 k},\quad \exp\Bigl(\frac{1}{12 k}\Bigr)-1\geq \frac{1}{12 k},
	\end{align}
	given that $k\geq 2/7$ and $1/(12k)\leq 7/24$;  the third and fourth inequalities in \eref{eq:ExpProof2} follow from the assumption that $k\geq 2/7$ and $n\geq 5k/3$.
\end{proof}

\begin{lemma}\label{lem:LnkInverse}
	Suppose $k\geq 1$ and $n\geq k$. Then
	\begin{align}\label{eq:LnkInverse1}
	\frac{1+\sqrt{1+4k}}{2L(n,k,k/n)}\sqrt{\frac{n-k}{n}}\leq 1-\frac{k}{2n \sqrt{1+4k}}.
	\end{align}
	Meanwhile, the function 
	\begin{align}\label{eq:LnkInverse2}
	n\biggl[\frac{1+\sqrt{1+4k}}{2L(n,k,k/n)}-1\biggr]
	\end{align}
	is strictly decreasing in $n$. If in addition $n\geq jk$ with $j\geq 1$, then
	\begin{align}\label{eq:LnkInverse3}
	\frac{1+\sqrt{1+4k}}{2L(n,k,k/n)}\leq 1+ \frac{jk}{n}\biggl[\frac{\sqrt{j}\,(1+\sqrt{1+4k}\,)}{\sqrt{j}+\sqrt{j+4(j-1)k}}-1\biggr].
	\end{align}
\end{lemma}

\begin{proof}[Proof of \lref{lem:LnkInverse}]
	Let 
	\begin{align}
	s(k,x):=\frac{1+\sqrt{1+4k}}{1+\sqrt{1+4k-4kx}}\sqrt{1-x},\quad 0\leq x\leq 1;
	\end{align}
	then $s(k,x)$ is continuous in $x$ for $0\leq x\leq 1$. In addition, $s(k,0)=1$ and 
	\begin{gather}
	\frac{\partial s}{\partial x}=-\frac{1+\sqrt{1+4k}}{2\sqrt{1-x}\,(1+4y+\sqrt{1+4y}\,)}<0\;\; \forall 0\leq x<1,\;\;
\frac{\partial s}{\partial x}\bigg|_{x=0}=-\frac{1}{2\sqrt{1+4k}},\\[1ex]
\frac{\partial^2 s}{\partial x^2}=-\frac{(1+\sqrt{1+4k}\,)(2+12y-\sqrt{1+4y}\,)}{4(1-x)^{3/2}(1+4y)^{3/2}(1+\sqrt{1+4y}\,)}<0  \quad \forall 0\leq x<1,
\end{gather}
where $y=k(1-x)$.	Therefore, $s(k,x)$ is strictly decreasing and concave in $x$ for $0\leq x\leq 1$, which means
	\begin{align}
	s(k,x)\leq 1-\frac{x}{2\sqrt{1+4k}}\quad \forall 0\leq x\leq 1. 
	\end{align}
	This  equation in turn implies \eref{eq:LnkInverse1}, given that the left hand side in \eref{eq:LnkInverse1} is equal to $s(k,k/n)$.

	To prove the monotonicity of the function defined in \eref{eq:LnkInverse2}, it suffices to prove that
	\begin{align}
	n_1\biggl[\frac{1+\sqrt{1+4k}}{2L(n,k,k/n_1)}-1\biggr]> n_2\biggl[\frac{1+\sqrt{1+4k}}{2L(n,k,k/n_2)}-1\biggr]\quad \forall n_2 > n_1 \geq k. 
	\end{align}
	Let $x_1:=k/n_1$ and $x_2:=k/n_2$; then the above equation is equivalent to 
	\begin{align}	\frac{1}{x_1}\biggl[\frac{1+\sqrt{1+4k}}{1+\sqrt{1+4k-4kx_1}}-1\biggr]>\frac{1}{x_2}\biggl[\frac{1+\sqrt{1+4k}}{1+\sqrt{1+4k-4kx_2}}-1\biggr]
	\end{align}
	for $0< x_2 < x_1\leq 1$.
	Now this conclusion follows from the fact that the function
	\begin{align}
	\frac{1+\sqrt{1+4k}}{1+\sqrt{1+4k-4kx}}-1
	\end{align} 
	is strictly increasing and strictly convex in $x$ for $0\leq x\leq 1$ and is equal to 0 when $x=0$. Therefore, the function defined in \eref{eq:LnkInverse2} is strictly decreasing in $n$.
	
	If in addition $n\geq jk$ with $j\geq 1$, then
	\begin{align} n\biggl[\frac{1+\sqrt{1+4k}}{2L(n,k,k/n)}-1\biggr]\leq jk\biggl[\frac{1+\sqrt{1+4k}}{2L(jk,k,1/j)}-1\biggr]=jk\biggl[\frac{\sqrt{j}\,(1+\sqrt{1+4k}\,)}{\sqrt{j}+\sqrt{j+4(j-1)k}}-1\biggr],
	\end{align}
	which implies \eref{eq:LnkInverse3}.
\end{proof}

\subsection{Proof of \lref{lem:Bnkbnknk}}
The proof is divided into three steps: In the first step  we prove  \eref{eq:Bnkbnknk2},  in the second step we prove   \eref{eq:Bnkbnknk} for the case $k\geq 16$, and in the third step we prove   \eref{eq:Bnkbnknk} for the case $k\leq 15$. To simplify the notation, $B_{n,k}(k/n)$  and $b_{n,k}(k/n)$ are abbreviated as $B_{n,k}$ and $b_{n,k}$, respectively, in the following proof. Several auxiliary functions defined in the  proof  are independent of those functions defined in the  proofs of previous results.  \\

\textbf{Step 1:} Proof of \eref{eq:Bnkbnknk2}, assuming that $n\leq 2k$. By definitions in \eqsref{eq:Bnkp}{eq:Lnkp} and the Stirling approximation in 
\eref{eq:Stirling} we can deduce that
	\begin{align}
	b_{n,k}&=\frac{\Gamma(n+1)}{\Gamma(k+1)\Gamma(n-k+1)}\frac{k^k (n-k)^{n-k}}{n^n} \label{eq:bnknkLB} \\
	&\geq \frac{k^k}{\rme^k\Gamma(k+1)}\sqrt{\frac{n}{n-k}}
	\exp\biggl[\frac{1}{12n+1}-\frac{1}{12(n-k)}\biggr]\nonumber\\
	&\geq \sqrt{\frac{n}{2\pi k(n-k)}}\exp\biggl[\frac{1}{12n+1}-\frac{1}{12 k}-\frac{1}{12(n-k)}\biggr], \nonumber
	\\
	B_{n,k}&=B_{n,k-1}(k/n)+b_{n,k}=\frac{1}{2}+(1-\zeta_{n,k}) b_{n,k},\\
	L(n,k,k/n)&=\frac{1}{2}\biggl(1+\sqrt{1+\frac{4k(n-k)}{n}}\,\biggr), \label{eq:Lnknk}
	\end{align}
	where the two inequalities in \eref{eq:bnknkLB} follow from the Stirling approximation in \eref{eq:Stirling}, and $\zeta_{n,k}$ is defined in \eref{eq:zetank}. Therefore,
	\begin{align}
	\frac{B_{n,k}}{b_{n,k}}&=\frac{1}{2b_{n,k}}+1-\zeta_{n,k} \label{eq:BnkbnkUB3}	\\
	&\leq 
	\frac{\rme^k\Gamma(k+1)}{2k^k}\sqrt{\frac{n-k}{n}}
	\exp\biggl[\frac{1}{12(n-k)}-\frac{1}{12n+1}\biggr]+1-\zeta_{n,k}
\nonumber	\\
	&\leq \sqrt{\frac{\pi k(n-k)}{2n}}\exp\biggl[\frac{1}{12 k}+\frac{1}{12(n-k)}-\frac{1}{12n+1}\biggr]+1-\zeta_{n,k} 
\nonumber	\\
	&\leq \sqrt{\frac{\pi k(n-k)}{2n}}\biggl[1+\frac{n}{11k(n-k)}\biggr]
+1-\zeta_{n,k}. \nonumber
	\end{align}

		If $\zeta_{n,k}\geq 1-\sqrt{\pi/8}$, which holds when $n\leq 2k$ by \eref{eq:zetaLBUB}, then  \eref{eq:Lnknk} and the third inequality in \eref{eq:BnkbnkUB3} yield
	\begin{gather}
	\frac{B_{n,k}}{b_{n,k}}
	\leq \Bigl(1+\frac{1}{11z}\Bigr)\sqrt{\frac{\pi z}{2}}+1-\zeta_{n,k}\leq 
	 \Bigl(1+\frac{1}{11z}\Bigr)\sqrt{\frac{\pi z}{2}}+\sqrt{\frac{\pi}{8}},\\	
	\frac{B_{n,k}}{ b_{n,k}L(n,k,k/n)}=\frac{2B_{n,k}}{(1+\sqrt{1+4z}\,) b_{n,k}}\leq \sqrt{\frac{\pi}{2}}\,g(z),
	\end{gather}
	where 
	\begin{align}
	z:=\frac{k(n-k)}{n},\quad g(z):=\frac{2(1+\frac{1}{11z})\sqrt{z} +1}{1+\sqrt{1+4z}}.\label{eq:xgx}
	\end{align}
	The derivative of $g(z)$ over $z$ reads
	\begin{align}
	g'(z)=\frac{(11 z-1)\sqrt{1+4z}+3z-22z^{3/2}-1}{11z^{3/2}\sqrt{1+4z}\,(1+\sqrt{1+4z}\,)^2}.
	\end{align}
	By assumption we have $z\geq 1/2$ and 
	\begin{align}
	&(11 z-1)\sqrt{1+4z}+3z-22z^{3/2}-1\\
	&\geq (11 z-1)2\sqrt{z}\biggl(1+\frac{\sqrt{6}-2}{4z}\biggr)+3z-22z^{3/2}-1 \nonumber
\\
	&=\frac{11\sqrt{6}-26}{2}\sqrt{z}-\frac{\sqrt{6}-2}{2\sqrt{z}}+3z-1
	\geq  \frac{9\sqrt{6}-22}{2}\sqrt{z}+3z-1> 3z-1> 0,\nonumber
	\end{align}
which implies that  $g'(z)> 0$. Therefore,
	\begin{align}
	\frac{B_{n,k}}{b_{n,k}L(n,k,k/n)}< \sqrt{\frac{\pi}{2}}\lim_{z\to\infty} g(z)=\sqrt{\frac{\pi}{2}},
	\end{align}
	which confirms \eref{eq:Bnkbnknk2}. \\

\textbf{Step 2:} Proof of \eref{eq:Bnkbnknk} in the case $k\geq 16$. Thanks to \eref{eq:Bnkbnknk2} proved above, we can assume that $n> 2k$, which means $n-k> k$ and $1/2\leq k/2< z< k$ given the assumption $1\leq k\leq n-1$. 	
 Then \lref{lem:Exp} implies that 
	\begin{align}\label{eq:ExpApp}
	\exp\biggl[\frac{1}{12 k}+\frac{1}{12(n-k)}-\frac{1}{12n+1}\biggr]\leq 1+c_k\frac{n}{k(n-k)}=1+\frac{c_k}{z},
	\end{align}
	where the coefficient $c_k$ is defined as
	\begin{align}
	c_k :=k\exp\Bigl(\frac{1}{12 k}\Bigr)-k,
	\end{align}
which is strictly decreasing in $k$ and satisfies 
	\begin{align}
	\frac{1}{12}<c_k \leq \exp\biggl(\frac{1}{12 }\biggr)-1<\frac{1}{11}.
	\end{align}
	In addition, it is known that $\zeta_{n,k}\geq \theta_k $, where $\theta_k$ is defined by Ramanujan's equation in \eref{eq:Ramanujan}  \cite{JogdS68}.
	Therefore, the second inequality of \eref{eq:BnkbnkUB3} and \eref{eq:ExpApp} imply that
	\begin{gather}
	\frac{B_{n,k}}{b_{n,k}}\leq \sqrt{\frac{\pi z}{2}}\Bigl(1+\frac{c_k}{z}\Bigr)+1-\zeta_{n,k}
	\leq \sqrt{\frac{\pi}{2}}\Bigl(1+\frac{c_k}{z}\Bigr)\sqrt{z}+1-\theta_k,\\
	\frac{B_{n,k}}{ b_{n,k}L(n,k,k/n)}=\frac{2B_{n,k}}{(1+\sqrt{1+4z}\,) b_{n,k}}
	\leq \sqrt{\frac{\pi}{2}}\, h_k(z),
	\end{gather}
	where the function $h_k(z)$ is defined as
	\begin{align}
	h_k(z):=\frac{2(1+\frac{c_k}{z})\sqrt{z} +a_k }{1+\sqrt{1+4z}},\quad 
	a_k:=2\sqrt{\frac{2}{\pi}}(1-\theta_k).
	\end{align}
Here $\theta_k$  is strictly decreasing in $k$ and satisfies \eref{eq:thetakLBUB} \cite{Szeg28,Wats29}, so 	 $a_k$ is strictly increasing in $k$ and satisfies
	\begin{align}
	1.02266\approx \sqrt{\frac{2}{\pi}}(4-\rme)=a_1\leq a_k< \frac{4}{3}\sqrt{\frac{2}{\pi}}\approx 1.06385.
	\end{align}
	
	Calculation shows that
	\begin{align}
	h_k'(z)=\frac{u_k(z)}{z^{3/2}\sqrt{1+4z}\,(1+\sqrt{1+4z}\,)^2},
	\end{align}
	where the function $u_k(z)$ is defined as
	\begin{align}
	u_k(z):=&\,(z-c_k)\sqrt{1+4z}-2a_kz^{3/2}+(1-8c_k)z-c_k\\
	\leq&\, (z-c_k)2\sqrt{z}\Bigl(1+\frac{1}{8z}\Bigr)-2a_kz^{3/2}+(1-8c_k)z-c_k \nonumber	\\
	=&\,-2(a_k-1)z^{3/2}+(1-8c_k)z+\frac{1-8c_k}{4}\sqrt{z}-c_k-\frac{c_k}{4\sqrt{z}}\leq v_k(z). \nonumber
	\end{align}
Here the function $v_k(z)$ is defined as
	\begin{align}
	v_k(z):=&-2(a_k-1)z^{3/2}+\frac{1}{3}z+\frac{1}{12}\sqrt{z}-\frac{1}{12}-\frac{1}{48\sqrt{z}}, 
\end{align}	
and its derivative over $z$ reads	
\begin{align}
	v_k'(z)=&-3(a_k-1)\sqrt{z}+\frac{1}{3}+\frac{1}{24\sqrt{z}}+\frac{1}{96z^{3/2}}.
	\end{align}
If $z\geq 9$ and $k\geq 9$, then 
	\begin{align}
	v_k'(z)\leq -9(a_9-1)+\frac{1}{3}+\frac{1}{72}+\frac{1}{2592}< 0,\quad u_k(z) \leq v_k(z)\leq  v_k(9)\leq  v_9(9)< 0,
	\end{align}
	which means $h_k(z)$ is strictly decreasing in $z$.

If in addition $z,k\geq 20$, then 
	\begin{align}
	\frac{B_{n,k}}{ b_{n,k}L(n,k,k/n)}\leq \sqrt{\frac{\pi}{2}}h_k(z)\leq \sqrt{\frac{\pi}{2}} h_k(20)\leq\sqrt{\frac{\pi}{2}}\frac{2\bigl(1+\frac{c_{20}}{20}\bigr)\sqrt{20}+a_\infty}{1+\sqrt{1+80}}< \frac{180451625}{143327232},
	\end{align}
	where $a_\infty=\lim_{k\to \infty}a_k=4\sqrt{2/\pi}/3$, given that $c_k$ is strictly decreasing in $k$, while $a_k$ is strictly increasing in $k$. Therefore, \eref{eq:Bnkbnknk} holds when $k\geq 40$ and $n\geq 2k$, in which case $z\geq k/2=20$ given the definition of $z$ in \eref{eq:xgx}. 
	
	When $25\leq k\leq 39$ and $z\geq 2k/3$, direct calculation shows that 
	\begin{align}
	\sqrt{\frac{\pi}{2}}h_k(z)\leq \sqrt{\frac{\pi}{2}} h_k(2k/3)< \frac{180451625}{143327232},
	\end{align}
	so  \eref{eq:Bnkbnknk} also holds when $25\leq k\leq 39$ and $n\geq 3k$. When $25\leq k\leq 39$ and $n< 3k$,  \eref{eq:Bnkbnknk} can be verified by direct calculation. 
	
	When $16\leq k\leq 24$ and $z\geq 9k/10$,
	direct calculation shows that 
	\begin{align}
	\sqrt{\frac{\pi}{2}}h_k(z)\leq \sqrt{\frac{\pi}{2}} h_k(9k/10)< \frac{180451625}{143327232},
	\end{align}
	so  \eref{eq:Bnkbnknk} also holds when $16\leq k\leq 24$ and $n\geq 10k$. When $16\leq k\leq 24$ and $n< 10k$,  \eref{eq:Bnkbnknk} can be verified by direct calculation. 
	
The above analysis shows that  \eref{eq:Bnkbnknk} holds when $k\geq 16$. \\
	
\textbf{Step 3:} Proof of \eref{eq:Bnkbnknk} in the case $k\leq 15$. First, suppose $n\geq jk$ with $j\geq 2$.  By virtue of  \eref{eq:BnkbnkUB3}, \lref{lem:LnkInverse}, and the following equation
	\begin{align}\label{eq:ExpApp2}
	\exp\biggl[\frac{1}{12(n-k)}-\frac{1}{12n+1}\biggr]&\leq \exp\biggl[\frac{1}{11(n-k)}-\frac{1}{11n}\biggr]\leq  1+\frac{1}{10(j-1)n},
	\end{align}
	we can deduce that
	\begin{align}\label{eq:BnkbnkLProof}
	\frac{B_{n,k}}{b_{n,k}L(n,k,k/n)}&\leq \frac{\rme^k \Gamma(k+1)}{2k^k L(n,k,k/n)}\sqrt{\frac{n-k}{n}}\exp\biggl[\frac{1}{12(n-k)}-\frac{1}{12n+1}\biggr]+\frac{1-\theta_k}{L(n,k,k/n)}
	\\&\leq \frac{\rme^k \Gamma(k+1)}{k^k (1+\sqrt{1+4k}\,)}\biggl(1-\frac{k}{2n\sqrt{1+4k}}\biggr)\biggl[1+\frac{1}{10(j-1)n}\biggr]\nonumber\\
	&\quad +\frac{2(1-\theta_k)}{1+\sqrt{1+4k}}\biggl\{1+\frac{jk}{n}\biggl[\frac{\sqrt{j}\,(1+\sqrt{1+4k}\,)}{\sqrt{j}+\sqrt{j+4(j-1)k}}-1\biggr]\biggr\}\nonumber \\
	&\leq  \frac{\rme^k \Gamma(k+1)}{k^k (1+\sqrt{1+4k}\,)}+\frac{2(1-\theta_k)}{1+\sqrt{1+4k}}
	+\frac{w_k}{n},\nonumber
	\end{align}
	where  $\theta_k$ is defined by Ramanujan's equation in \eref{eq:Ramanujan}  and $w_k$ is defined as 
	\begin{align}
	w_k:=&\,\frac{\rme^k \Gamma(k+1)}{k^k (1+\sqrt{1+4k}\,)}\biggl[\frac{1}{10(j-1)}-\frac{k}{2\sqrt{1+4k}}\biggr] \\
	&\, +\frac{2jk(1-\theta_k)}{1+\sqrt{1+4k}}\biggl[\frac{\sqrt{j}\,(1+\sqrt{1+4k}\,)}{\sqrt{j}+\sqrt{j+4(j-1)k}}-1\biggr].\nonumber
	\end{align}
Here the first inequality in \eref{eq:BnkbnkLProof} follows from the first inequality in \eref{eq:BnkbnkUB3} and the inequality  $\zeta_{n,k}\geq \theta_k $  \cite{JogdS68}; the second inequality  in \eref{eq:BnkbnkLProof}  follows from \eqsref{eq:LnkInverse1}{eq:LnkInverse3}	in \lref{lem:LnkInverse} and \eref{eq:ExpApp2},
given that $1-\theta_k>0$ by \eref{eq:thetakLBUB}; the third inequality in \eref{eq:BnkbnkLProof}  follows from straightforward calculation.

	Now we choose $j=11$, then direct calculation shows that $w_k<0$ for $k=1,2,\ldots, 15$, so 
	\begin{align}
	\frac{B_{n,k}}{b_{n,k}L(n,k,k/n)}&< \frac{\rme^k \Gamma(k+1)}{k^k (1+\sqrt{1+4k}\,)}+\frac{2(1-\theta_k)}{1+\sqrt{1+4k}}\leq \frac{180451625}{143327232} 
	\end{align}
	for $k=1,2,\ldots, 15$ and $ n\geq 11 k$, which confirms \eref{eq:Bnkbnknk}. 
	When $1\leq k\leq 15$ and  $n<11k$, \eref{eq:Bnkbnknk} can be verified directly. This observation completes the proof of \lref{lem:Bnkbnknk}.

\section{\label{app:ellPhiProof}Proof of \pref{pro:ellPhi}}

By the definitions of $\ell(x)$ and $\upsilon(x)$ in \eref{eq:ellupsilon} we can deduce that
\begin{align}
\frac{d\bigl[\frac{\ell(x)}{\upsilon(x)}\bigr]}{dx}=\begin{cases}
\frac{2+x-\sqrt{4 + x^2}}{(2-x)^2\sqrt{4 + x^2}} &0\leq x\leq 1,\\[1ex]
\frac{2+x^2-x\sqrt{4 + x^2}}{\sqrt{4 + x^2}} &x\geq 1,
\end{cases}
\end{align}
which shows that the derivative is positive when $x>0$. Since  $\ell(x),\upsilon(x)>0$
 are continuous in $x$ for $x\geq 0$. It follows that  $\ell(x)/\upsilon(x)$ is  strictly increasing in $x$, and
$\upsilon(x)/\ell(x)$  is  strictly decreasing in $x$. By definition it is also straightforward to verify that 
\begin{align}
\frac{\upsilon(0)}{\ell(0)}=2,\quad \lim_{x\to \infty}\frac{\upsilon(x)}{\ell(x)}=\lim_{x\to \infty}\sqrt{2\pi} \,x\ell(x)=1,\label{eq:ellupsilonLimProof}
\end{align}
which implies \eref{eq:ellupsilonLUB} given that $\upsilon(x)/\ell(x)$  is  strictly decreasing in $x$; in addition,  the second inequality in  \eref{eq:ellupsilonLUB} is saturated iff $x=0$. In conjunction with \eref{eq:PhixLUB} we can then deduce that 
\begin{align}
\lim_{x\to \infty}\frac{\Phi(-x)\rme^{x^2/2}}{\ell(x)}=1.
\end{align}
The two equations above 
together  confirm \eref{eq:ellupPhiLim}. 

Finally, we are ready to prove \eref{eq:ellPHi}. Direct calculation yields
\begin{align}
&\frac{d}{{d} x} \biggl[\frac{\Phi(-x) \mathrm{e}^{x^{2} / 2}}{\ell(x)}\biggr]
=
\frac{1}{\ell(x)}  \left[\left(x+\frac{1}{\sqrt{x^2+4}} \right)\mathrm{e}^{x^{2} / 2} \Phi(-x)  -\frac{1}{\sqrt{2\pi}} \right] \\
&\leq 
\frac{1}{\sqrt{2\pi}\,\ell(x)}  \left[\left(x+\frac{1}{\sqrt{x^2+4}} \right)\frac{4 }{3 x+\sqrt{8+x^{2}}}  - 1 \right]
< 0\quad \forall x\geq 0, \nonumber
\end{align}
so  $\Phi(-x)\rme^{x^2/2}/\ell(x)$ is strictly decreasing in $x$ for $x\geq 0$. 
Here the first inequality follows from the following inequality proved by Sampford \cite{Samp53},
\begin{align}
\Phi(-x)<\frac{4 }{3 x+\sqrt{8+x^{2}}} \frac{\mathrm{e}^{-x^{2} / 2}}{\sqrt{2 \pi}}  \quad \forall x\geq 0,
\end{align}
and the second inequality can be proved as follows, assuming that $x\geq0$,
\begin{align}
\begin{aligned}
\left(x+\frac{1}{\sqrt{x^2+4}} \right)\frac{4 }{3 x+\sqrt{8+x^{2}}}-1 < 0 
\quad&\Leftrightarrow\quad
x < \sqrt{8+x^{2}}-\frac{4}{\sqrt{x^2+4}} \\
\quad\Leftrightarrow\quad
x^2 < 8+x^{2}+\frac{16}{x^2+4}-8\sqrt{\frac{x^{2}+8}{x^2+4}}  
\quad&\Leftrightarrow\quad
1+\frac{2}{x^2+4} > \sqrt{\frac{x^{2}+8}{x^2+4}} \\
\Leftrightarrow\quad
1+\frac{4}{x^2+4}+ \left(\frac{2}{x^2+4}\right)^2 > \frac{x^{2}+8}{x^2+4}  
\quad&\Leftrightarrow\quad
\left(\frac{2}{x^2+4}\right)^2 > 0. 
\end{aligned}
\end{align}
Now \eref{eq:ellPHi} follows from \eref{eq:ellupPhiLim} and the equality $\Phi(0)=\sqrt{\pi/2}\, \ell(0)$, given that the function $\Phi(-x)\rme^{x^2/2}/\ell(x)$ is strictly decreasing in $x$ for $x\geq 0$. In addition,  the second inequality in \eref{eq:ellPHi} is saturated iff $x=0$.
\end{appendix}

\begin{acks}[Acknowledgments]	
HZ and ZL are also affiliated to Institute for Nanoelectronic Devices and Quantum Computing, Fudan University and Center for Field Theory and Particle Physics, Fudan University.
MH is also affiliated to International Quantum Academy (SIQA) and Graduate School of Mathematics, Nagoya University.
\end{acks}

\begin{funding}
The work at Fudan is  supported by   the National Natural Science Foundation of China (Grants No.~11875110 and No.~92165109) and  Shanghai Municipal Science and Technology Major Project (Grant No.~2019SHZDZX01).
MH is supported in part by the National Natural Science Foundation of China (Grants No. 62171212 and No.~11875110) and
Guangdong Provincial Key Laboratory (Grant No. 2019B121203002).
\end{funding}

\bibliographystyle{imsart-number}
\bibliography{all_references}
\end{document}